\documentclass[11pt]{article}
\usepackage{amssymb,amsmath, amsthm, latexsym,mathrsfs}
\usepackage{dsfont}
\usepackage{color}

\newtheorem{thm}{Theorem}[section]
\newtheorem{defn}[thm]{Definition}
\newtheorem{corollary}[thm]{Corollary}

\newtheorem{lemma}[thm]{Lemma}

\newtheorem{remark}[thm]{Remark}

\newtheorem{assumption}[thm]{Assumption}

\newcommand\E{{\mathbb E}}

\newcommand\bP{\mathbb{P}}
\newcommand\bR{\mathbb{R}}
\newcommand\bH{\mathbb{H}}
\newcommand\bZ{\mathbb{Z}}

\newcommand\bD{\mathbb{D}}
\newcommand\bS{\mathbb{S}}
\newcommand\bN{\mathbb{N}}

\newcommand\cB{\mathcal{B}}
\newcommand\cD{\mathcal{D}}
\newcommand\cF{\mathcal{F}}
\newcommand\cH{\mathcal{H}}

\newcommand\cP{\mathcal{P}}

\newcommand\cO{\mathcal{O}}

\newcommand\frH{\mathfrak{H}}

\newcommand{\nnrm}[2]{\ensuremath{\| #1 \|_{#2}}}           
\newcommand{\gnnrm}[2]{\ensuremath{\big\| #1 \big\|_{#2}}}  
\newcommand{\sgnnrm}[2]{\ensuremath{\Big\| #1 \Big\|_{#2}}} 

\newcommand{\nrklam}[1]{(#1)} 
\newcommand{\rklam}[1]{\left(#1\right)}                     
\newcommand{\grklam}[1]{\big(#1\big)}                       
\newcommand{\sgrklam}[1]{\Big(#1\Big)}                      
\newcommand{\rrklam}[1]{\Bigg(#1\Bigg)}

\newcommand{\ggklam}[1]{\big\{#1\big\}}                     
\newcommand{\sggklam}[1]{\Big\{#1\Big\}}                    

\newcommand{\neklam}[1]{[#1]}                   
                     
\newcommand{\geklam}[1]{\big [#1 \big ]}                    
\newcommand{\sgeklam}[1]{\Big [#1 \Big ]}                   
\newcommand{\reklam}[1]{\Bigg [#1 \Bigg ]}                  

\newcommand{\wP}{\mathds{P}}
\newcommand{\ijk}{\ensuremath{{i,j,k}}}                             
\newcommand{\jk}{\ensuremath{{j,k}}}                               
\newcommand{\supp}{\ensuremath{\mathop{\operatorname{supp}}}}      

\newcommand{\gdomain}{\ensuremath{G}}
\newcommand{\domain}{\ensuremath{\mathcal{O}}}   
\newcommand{\domainc}{\ensuremath{M}}
\newcommand{\laplacedwp}{\ensuremath{\Delta^D_{p,w}}}              
\newcommand{\laplacedw}{\ensuremath{\Delta^D_{2,w}}} 

\title{
On the $L_q(L_p)$-regularity and Besov smoothness of stochastic parabolic equations on bounded Lipschitz domains\thanks{This work has been supported by the Deutsche Forschungsgemeinschaft (DFG, grants DA 360/13-1, DA 360/13-2, DA 360/11-1, DA 360/11-2,  SCHI 419/5-1, SCHI 419/5-2) and a doctoral scholarship of the Philipps-Universit\"{a}t Marburg. The research of the third author has been also supported by Basic Science Research Program through the National Research Foundation of Korea (NRF) funded by the Ministry of Education, Science and Technology (2011-0005597).}
}

\author{Petru~A.~Cioica, Kyeong-Hun~Kim,  Kijung~Lee,  Felix~Lindner}

\date{December 4, 2012}

\begin{document}
\maketitle

\begin{abstract}
We investigate the regularity of linear stochastic parabolic equations with zero Dirichlet boundary condition on bounded Lipschitz domains $\mathcal{O}\subseteq \mathbb{R}^d$ with both theoretical and numerical purpose. We use N.V. Krylov's framework of stochastic parabolic weighted Sobolev spaces $\mathfrak{H}^{\gamma,q}_{p,\theta}(\mathcal{O},T)$. The summability parameters $p$ and $q$ in space and time may differ. Existence and uniqueness of solutions in these spaces is established and the H\"older regularity in time is analysed. Moreover, we prove a general embedding of weighted $L_p(\mathcal{O})$-Sobolev spaces into the scale of Besov spaces $B^\alpha_{\tau,\tau}(\mathcal{O})$, $1/\tau=\alpha/d+1/p$, $\alpha>0$.
This leads to a H\"older-Besov regularity result for the solution process.
The regularity in this Besov scale determines the order of convergence that can 
be achieved by certain nonlinear approximation schemes.
\end{abstract}
\\[.5em]

\section{Introduction}\label{Introduction}

Let $\domain\subseteq\mathbb R^d$ be a bounded Lipschitz domain,
$T\in(0,\infty)$ and let
$(w^k_t)_{t\in[0,T]}$, $k\in\mathbb N$, be independent
one-dimensional standard Wiener processes defined on a probability
space $(\Omega,\mathcal F,\mathbb P)$. We are interested in the
regularity of the solutions to parabolic stochastic partial
differential equations (SPDEs, for short) with zero Dirichlet
boundary condition of the form
\begin{equation}\label{eq:mainEq}
\left.
\begin{alignedat}{3}
du&=(a^{ij}&&u_{x^ix^j}\,+f)dt+(\sigma^{ik}u_{x^i}+g^k)dw^k_t \quad \text{ on }\Omega\times[0,T]\times\domain,\\
u&=0&&\quad\text{ on } \Omega\times(0,T]\times\partial\domain,\\
u(0,\cdot)&=u_0&&\quad\text{ on }\Omega\times\domain,
\end{alignedat}
\quad\right\}
\end{equation}
where the indices $i$ and $j$ run from $1$ to $d$ and the index $k$ runs through $\bN=\{1,2,\ldots\}$.
Here and in the sequel we use the summation convention on the repeated indices $i,j,k$.
The coefficients $a^{ij}$ and $\sigma^{ik}$ depend on $(\omega,t)\in\Omega\times[0,T]$.
The force terms $f$ and $g^k$ depend on
$(\omega,t,x)\in\Omega\times[0,T]\times\domain$.
By the nature of the problem, in particular by the bad contribution of the infinitesimal differences of the Wiener processes, the second spatial derivatives of the solution may blow up at the boundary $\partial\domain$ even if the boundary is smooth, see, e.g., \cite{Kry1994}.
Hence, a natural way to deal with problems of type \eqref{eq:mainEq} is to consider $u$ as a stochastic process with values in weighted Sobolev spaces on $\domain$ that allow the derivatives of functions from these spaces to blow up near the boundary.
This approach has been initiated and developed by N.V.\ Krylov and collaborators, first as an $L_2$-theory for smooth domains $\domain$ (see \cite{Kry1994}), then
as an $L_p$-theory ($p\geq2$) for the half space (\cite{KryLot1999, KryLot1999b}), for smooth domains (\cite{Kim2004, KimKry2004}), and for general bounded domains allowing Hardy's inequality such as bounded Lipschitz domains (\cite{Kim2011}). Existence and uniqueness of solutions have been established within specific stochastic parabolic weighted Sobolev spaces, denoted by $\frH^{\gamma}_{p,\theta}(\domain,T)$ in \cite{Kim2011}.
These spaces consist of elements $\tilde u$ of the form $d\tilde u=\tilde fdt+\tilde g^kdw^k_t$, where $\tilde u$, $\tilde f$ and $\tilde g^k$, considered as stochastic processes with values in certain weighted $L_p(\domain)$-Sobolev spaces, are $L_p$-integrable w.r.t.\ $\wP\otimes dt$. We refer to Section~\ref{StWS-00} for the exact definition.

In this article we treat regularity issues concerning the solution $u$ of problem \eqref{eq:mainEq} which arise, besides others, in the context of adaptive numerical approximation methods.

The starting point of our considerations was the question whether we can improve the Besov regularity results in \cite{CioDahKin+2011} in time direction. In \cite{CioDahKin+2011} the spatial regularity of $u$ is measured in the scale of Besov spaces
\begin{equation}\label{NAscale}\tag{$\ast$}
B^{\alpha}_{\tau,\tau}(\domain),\quad\frac 1\tau=\frac\alpha d+\frac1 p,\quad\alpha>0,
\end{equation}
where $p\geq 2$ is fixed. Note that for $\alpha>(p-1)d/p$ the sumability parameter $\tau$ becomes less than one, so that in this case $B^\alpha_{\tau,\tau}(\domain)$ is not a Banach space but a quasi-Banach space. It is a known result from approximation theory that the smoothness of a target function $f\in L_p(\domain)$ within the scale \eqref{NAscale} determines the rate of convergence that can be achieved by adaptive and other nonlinear approximation methods if the approximation error is measured in $L_p(\domain)$;
see \cite[Chapter 4]{Coh2003}, \cite{DeV1998} or the introduction of \cite{CioDahKin+2011}.
Based on the $L_p$-theory in \cite{Kim2011}, it is shown in \cite{CioDahKin+2011} that the solution $u$ to problem \eqref{eq:mainEq} satisfies
\begin{equation}\label{oldBesovResult}
u\in L_\tau(\Omega\times[0,T],\cP,\bP\otimes dt;B^{\alpha}_{\tau,\tau}(\domain)),\quad\frac 1\tau=\frac\alpha d+\frac1 p,
\end{equation}
for certain $\alpha>0$ depending on the smoothness of $u_0$, $f$ and $g^k$, $k\in\bN$.
In general, the spatial regularity of $u$ in the Sobolev scale $W^s_p(\domain)$, $s\geq0$, which determines the order of convergence for uniform approximation methods in $L_p(\domain)$, is strictly less than the spatial regularity of $u$ in the scale \eqref{NAscale}.
It can be due to, e.g., the irregular behaviour of the noise at the boundary or the irregularities of the boundary itself; see \cite[Chapter 4]{Lin2011} for the latter case.
This justifies the use of nonlinear approximation methods such as adaptive wavelet methods for the numerical treatment of SPDEs, cf. \cite{CioDahDoe+2012,CioDahDoe+2012b}.
The proof of \eqref{oldBesovResult} relies on characterizations of Besov spaces by wavelet expansions and on weighted Sobolev norm estimates for $u$, resulting from the solvability of the problem \eqref{eq:mainEq} within the spaces $\frH^\gamma_{p,\theta}(\domain,T)$.

An obvious approach to improve \eqref{oldBesovResult} with respect to regularity in time is to try to combine the existing H{\"o}lder estimates in time for the elements of the spaces $\frH^\gamma_{p,\theta}(\domain,T)$ (see \cite[Theorem 2.9]{Kim2011}) with the wavelet arguments in \cite{CioDahKin+2011}. However, it turns out that a satisfactory result requires a more subtle strategy in three different aspects.

Firstly, we need an extension of the $L_p$-theory in \cite{Kim2011} to an $L_q(L_p)$-theory for SPDEs dealing with stochastic parabolic weighted Sobolev spaces $\frH^{\gamma,q}_{p,\theta}(\domain,T)$ with possibly different summability parameters $q$ and $p$ in time and space respectively.
These spaces consist of elements $\tilde u$ of the form $d\tilde u=\tilde fdt+\tilde g^kdw^k_t$, where $\tilde u$, $\tilde f$ and $\tilde g^k$, considered as stochastic processes with values in suitable weighted $L_p(\domain)$-Sobolev spaces, are $L_q$-integrable w.r.t.\ $\wP\otimes dt$. Such an extension is needed to obtain better H{\"o}lder estimates in time in a second step.
Satisfactory existence an uniqueness results concerning solutions in the spaces $\frH^{\gamma,q}_{p,\theta}(\domain,T)$ have been established in \cite{Kim2009} for domains $\domain$ with $C^1$-boundary. Unfortunately, the techniques used there do not work on general Lipschitz domains. 
Also, the $L_q(L_p)$-results that have been obtained in \cite{NeeVerWei2012} within the semigroup approach to SPDEs do not directly suit our purpose: On the one hand, for general Lipschitz domains $\cO$ the domains of the fractional powers of the leading linear differential operator cannot be characterized in terms of Sobolev or Besov spaces as in the case of a smooth domains $\cO$; see, e.g., the introduction of \cite{CioDahKin+2011} for details. On the other hand, even in the case of a smooth domain $\cO$ we need regularity in terms \emph{weighted} Sobolev spaces to obtain the optimal regularity in the scale \eqref{NAscale}.

Secondly, once we have established the solvability of SPDEs within the spaces $\frH^{\gamma,q}_{p,\theta}(\domain,T)$, we have to exploit the $L_q(L_p)$--regularity of the solution and derive improved results on the H\"older regularity in time for large $q$. 
For $\domain=\bR^d_+$ this has been done by Krylov \cite{Kry2001}. It takes quite delicate arguments to apply these results to the case of bounded Lip\-schitz domains via a boundary flattening argument.

Thirdly, in order to obtain a reasonable H{\"o}lder-Besov regularity result, it is necessary to generalize the wavelet arguments applied in \cite{CioDahKin+2011} to a wider range of smoothness parameters. This requires more sophisticated estimates.

In this article we tackle and solve the tasks described above. We organize the article as follows.
In Section~2 we recall the definition and basic properties of the (deterministic) weighted Sobolev spaces
$H^\gamma_{p,\theta}(\gdomain)$ introduced in \cite{Lot2000} (see also \cite[Chapter 6]{Tri1995}) on general domains $\gdomain\subseteq\bR^d$ with non-empty boundary.
In Section~3 we give the definition of the spaces $\frH^{\gamma,q}_{p,\theta}(\gdomain,T)$ and
specify the concept of a solution for equations of type \eqref{eq:mainEq} in these spaces. Moreover, we show that if we have a solution $u\in\frH^{\gamma,q}_{p,\theta}(\gdomain,T)$ with low regularity $\gamma\geq 0$, but $f$ and the $g^k$'s have high $L_q(L_p)$-regularity, then we can lift up the regularity of the solution (Theorem~\ref{thm:liftReg}).
In this sense the spaces $\frH^{\gamma,q}_{p,\theta}(\gdomain,T)$ are the right ones for our regularity analysis of SPDEs.

Section 4 is devoted to the solvability of Eq.\ \eqref{eq:mainEq} in $\frH^{\gamma,q}_{p,\theta}(\domain,T)$, $\domain\subseteq\bR^d$ being a bounded Lipschitz domain. The focus lies on the case $q>p\geq2$ and we restrict our considerations to equations with additive noise, i.e. $\sigma^{ik}\equiv0$.
In Subsection~\ref{subsec:LqLp_heat} we consider equations on domains with small Lipschitz constants and derive a result for general integrability parameters $q\geq p\geq2$ (Theorem~\ref{thm:LpLq}). We use an $L_q(L_p)$-regularity result for deterministic parabolic equations from \cite{DonKim2011} and an estimate for stochastic integrals in UMD spaces from \cite{NeeVerWei2007} to obtain a certain low $L_q(L_p)$-regularity of the solution. Then the regularity is lifted up with the help of Theorem~\ref{thm:liftReg}.
In Subsection~\ref{subsec:LqLp_heat}, we consider the stochastic heat equation on general bounded Lipschitz domains. Here we use the results from \cite{NeeVerWei2012} on maximal $L_q$-regularity of stochastic evolution equations (see also \cite{NeeVerWei2012b} and \cite{NeeVerWei2007}) to derive existence and uniqueness of a solution with low regularity. A main ingredient will be to prove that the domain of the square root of the weak Dirichlet Laplacian on $L_p(\domain)$ coincides with the closure of the test functions $C^\infty_0(\domain)$ in the $L_p(\domain)$-Sobolev space of order one (Lemma~\ref{thm:laplacedwpsqrt}). This stays true only for a certain range of $p\in[2,p_0)$ with $p_0>3$. Thus, so does our result (Theorem \ref{thm:heat_LqLp}).
In a second step, we again lift up the regularity by using Theorem~\ref{thm:liftReg}.
In both settings we derive suitable a-priori estimates.

In Section~\ref{StWS} we present our result on the H{\"o}lder regularity in time of the elements of $\frH^{\gamma,q}_{p,\theta}(\domain,T)$ (Theorem~\ref{thm:mainHoelder}). It is an extension of the H{\"o}lder estimates in time for the elements of
$\frH^{\gamma,q}_{p,\theta}(T)=\frH^{\gamma,q}_{p,\theta}(\bR^d_+,T)$
in \cite{Kry2000} to the case of bounded Lipschitz domains.
The implications for the H\"older regularity of the solutions of SPDEs are described in Theorem~\ref{thm:Hoelder_SPDEs}.

In Section~\ref{Bes-WS} we pave the way for the analysis of the spatial regularity of the solutions of SPDEs in the scale \eqref{NAscale}.
We  discuss the relationship between the weighted Sobolev spaces $H^\gamma_{p,\theta}(\domain)$ and Besov spaces. Our main result in this section, Theorem \ref{thm:WSBes-ptau}, is a general embedding of the spaces $H^\gamma_{p,d-\nu p}(\domain)$, $\gamma,\nu>0$, into the Besov scale \eqref{NAscale}. Its proof is an extension of the wavelet arguments in the proof of \cite[Theorem 3.1]{CioDahKin+2011}, where only integer valued smoothness parameters $\gamma$ are considered. It can also be seen as an extension of and a supplement to the Besov regularity results for deterministic
elliptic equations in \cite{DahDeV1997} and \cite{Dah1998,Dah1999,Dah2002,DahSic2008}.
To the best of our knowledge, no such general embedding has been proven before.
In the course of the discussion we also enlighten the fact that, for the relevant range of parameters $\gamma$ and $\nu$, the spaces $H^\gamma_{p,d-\nu p}(\domain)$ act like Besov spaces $B^{\gamma\wedge\nu}_{p,p}(\domain)$ with zero trace on the boundary (Remark \ref{rem:Dirichlet}).

In Section~\ref{RegHB} the results of the previous sections are combined in order to determine the H{\"o}lder-Besov regularity of the elements of the stochastic parabolic spaces $\frH^{\gamma,q}_{p,\theta}(\domain,T)$ and of the solutions of SPDEs within these spaces.
The related result in \cite{CioDahKin+2011} is significantly improved in several aspects; see Remark~\ref{rem:rel_LpLp} for a detailed comparison.
We obtain an estimate of the form
\begin{equation*}
\E\|u\|^q_{C^\kappa([0,T];B^\alpha_{\tau,\tau}(\domain))}
\leq C \nnrm{u}{\frH^{\gamma,q}_{p,\theta}(\domain,T)}^q,\quad\frac 1\tau=\frac\alpha d+\frac1 p,
\end{equation*}
for certain $\alpha$ depending on the smoothness and weight parameters $\gamma$ and $\theta$ and for certain $\kappa$ depending on
$q$ and $\alpha$ (Theorem~\ref{thm:HoeBes}).
Using the a-priori estimates from Section~\ref{SPDEs}, the right hand side of the above inequality can be estimated by suitable norms of $f$ and $g$ if $u$ is the solution to the corresponding SPDE (Theorem~\ref{thm:HoelderBesov_SPDEs}).

Let us mention the related work \cite{AimGom2012} on the Besov regularity  for the deterministic heat equation. The authors study the regularity of temperatures in terms of anisotropic Besov spaces of type
$B^{\alpha/2,\alpha}_{\tau,\tau}((0,T)\times\domain)$,
$1/\tau=\alpha/d+1/p$.
However, the range of admissible values for the parameter $\tau$ is a priori restricted to $(1,\infty)$, so that $\alpha$ is always less than $d(1-1/p)$. In our article the parameter $\tau$ in \eqref{NAscale} may be any positive number, including in particular the case where $\tau$ is less than $1$ and where $B^\alpha_{\tau,\tau}(\domain)$ is not a Banach space but a quasi-Banach space.

\vspace{0.2cm}
\noindent\textbf{Notation and Conventions.} Throughout this
paper, $\domain$ always denotes a bounded Lipschitz domain in
$\bR^d$, $d\geq 1$, as specified in Definition~\ref{domain} below. 
General subsets of $\bR^d$ are denoted by $\gdomain$. We write $\partial\gdomain$ for their boundary (if it is not empty) and $\gdomain^\circ$ for the interior.  $\bN:=\{1,2,\ldots\}$ denotes the set of strictly positive integers whereas $\bN_0:=\bN\cup\{0\}$.  Let $(\Omega,\cF,\wP)$
be a complete probability space and $\{\cF_{t},t\geq0\}$ be an
increasing filtration of $\sigma$-fields $\cF_{t}\subset\cF$, each
of which contains all $(\cF,\bP)$-null sets. By  $\cP$ we denote the
predictable $\sigma$-field generated by $\{\cF_{t},t\geq0\}$ and  we
assume that $\{(w^{1}_{t})_{t\in[0,T]}, (w^{2}_{t})_{t\in[0,T]},\ldots\}$ are independent
one-dimensional Wiener processes w.r.t.\ $\{\cF_{t},t\geq0\}$.
For $\kappa\in(0,1)$ and a quasi-Banach space $(X,\|\cdot\|_X)$ we
denote by $C^\kappa([0,T];X)$ the H{\"o}lder space of continuous
$X$-valued functions on $[0,T]$ with finite norm
$\nnrm{\cdot}{C^\kappa([0,T];X)}$ defined by
\begin{align*}
[u]_{C^\kappa([0,T];X)}&:=\sup_{s,t\in[0,T]}\frac{\|u(t)-u(s)\|_X}{|t-s|^\kappa},\\
\nnrm{u}{C([0,T];X)}&:=\sup_{t\in[0,T]}\|u(t)\|_X,\\
\nnrm{u}{C^\kappa([0,T];X)}&=\nnrm{u}{C([0,T];X)}+[u]_{C^\kappa([0,T];X)}.
\end{align*}
For $1<p<\infty$, $L_p(A,\Sigma,\mu;X)$ denotes the space of $\mu$-strongly measurable and $p$-Bochner integrable functions with values in $X$ on a $\sigma$-finite measure space $(A,\Sigma,\mu)$, endowed with the usual $L_p$-Norm. 
We write $L_p(\gdomain)$ instead of $L_p(\gdomain,\cB(\gdomain),\lambda^d;\bR)$ if $G\in\cB(\bR^d)$, where $\cB(\gdomain)$ and $\cB(\bR^d)$ are the Borel-$\sigma$-fields on $\gdomain$ and $\bR^d$.
Recall the Hilbert space $\ell_2:=\ell_2(\bN)=\{\mathbf{a}=(\mathbf{a}^{1},\mathbf{a}^{2},\ldots):|\mathbf{a}|_{\ell_2}=(\sum_k |\mathbf{a}^{k}|^2)^{1/2}<\infty\}$ with the inner product $\langle\mathbf{a},\mathbf{b}\rangle_{\ell_2}=\sum_i \mathbf{a}^{k}\mathbf{b}^{k}$, for $\mathbf{a},\mathbf{b}\in\ell_2$.
The notation $C^\infty_0(\gdomain)$ is used for the space of infinitely differentiable test functions with compact support in a domain $\gdomain\subseteq\bR^d$.
For any distribution $f$ on $G$ and any $\varphi\in C^\infty_0(\gdomain)$, $(f,\varphi)$ denotes the application of $f$ to $\varphi$. 
Furthermore, for any multi-index $\alpha=(\alpha_1,\ldots,\alpha_d)\in\bN_0^d$, we write $D^\alpha f=\frac{\partial^{|\alpha|}f}{\partial x_1^{\alpha_1}\ldots\partial x_d^{\alpha_d}}$ for the corresponding (generalized) derivative w.r.t. $x=(x_1,\ldots,x_d)\in\gdomain$, where $|\alpha|=\alpha_1+\ldots+\alpha_d$. 
By making slight abuse of notation, for $m\in\bN_0$, we write $D^m f$ for any (generalized) $m$-th order derivative of $f$ and for the vector of all $m$-th order derivatives of $f$. E.g.\ if we write $D^m f\in X$, where $X$ is a function space on $G$, we mean $D^\alpha f\in X$ for all $\alpha\in\bN_0^d$ with $|\alpha|=m$.
We also use the notation $f_{x^i x^j}=\frac{\partial^{2}f}{\partial x^i\partial x^j},\;f_{x^i}=\frac{\partial f}{\partial x^i}$. The notation $f_x$ (respectively $f_{xx}$) is used synonymously  for $Df:= D^1f$ (respectively for $D^2f$), whereas $\nnrm{f_x}{X}:=\sum_{i} \nnrm{u_{x^i}}{X}$ (respectively $\nnrm{f_{xx}}{X}:=\sum_{i,j} \nnrm{f_{x^i x^j}}{X}$). 
Moreover, $\Delta f:=\sum_i f_{x^i x^i}$, whenever it makes sense.
Given $p\in[1,\infty)$ and $m\in\bN_0$, $W^m_p(\gdomain)$ denotes the classical Sobolev space consisting of all $f\in L_p(\gdomain)$ such that 
$
 |f|_{W^m_p(\gdomain)}:= 
\sup_{\alpha\in\bN_0^d,\,|\alpha|=m}\nnrm{D^\alpha f}{L_p(\gdomain)}$
is finite. It is normed via $\nnrm{f}{W^m_p(\gdomain)}:=\nnrm{f}{L_p(\gdomain)}+|f|_{W^m_p(\gdomain)}$.
The closure of $C^\infty_0(\domain)$ in $W^{1}_p(\domain)$ is denoted by $\mathring{W}^1_p(\domain)$ and is normed by $\nnrm{f}{\mathring{W}^1_p(\domain)}:=(\sum_{i}\nnrm{f_{x^i}}{L_p(\domain)}^p)^{1/p}$.
If we have two quasi-normed spaces $(X_i,\nnrm{\cdot}{X_i})$, $i=1,2$, $X_1\hookrightarrow X_2$ means that $X_1$ is continuously linearly embedded in $X_2$. For a compatible couple $(X_1,X_2)$ of quasi-Banach spaces,  $[X_1,X_2]_\eta$ denotes the interpolation space of exponent $\eta\in(0,1)$ arising from the complex interpolation method. In general, $N$ will denote a positive finite constant, which may differ from line to line. The notation $N=N(a_1,a_2,\ldots)$ is used to emphasize the dependence of the constant $N$ on the set of parameters $\{a_1, a_2, \ldots\}$. In general, this set will \emph{not} contain all the parameters $N$ depends on.
$A\sim B$ means that $A$ and $B$ are equivalent.

\setcounter{equation}{0}
\section{Weighted Sobolev spaces}\label{WS}

We start by recalling the definition and some basic properties of the 
(deterministic and stationary) 
weighted Sobolev spaces $H^\gamma_{p,\theta}(\gdomain)$ introduced in \cite{Lot2000}. These spaces will serve as state spaces for the solution processes $u=(u(t))_{t\in[0,T]}$ to SPDEs of type \eqref{eq:mainEq} and they will play a fundamental role in all the forthcoming sections.

For $p\in(1,\infty)$ and $\gamma \in \bR$, let
$H^{\gamma}_p:=H^{\gamma}_p(\bR^d):=(1-\Delta)^{-\gamma/2}L_p(\bR^d)$ 
be the spaces of Bessel potentials, endowed with the norm
\begin{equation*}
\|u\|_{H^{\gamma}_p}:=\|(1-\Delta)^{\gamma/2}u\|_{L_p(\bR^d)}:=\|\cF^{-1}[(1+|\xi|^2)^{\gamma/2}\cF(u)(\xi)]\|_{L_p(\bR^d)},
\end{equation*}
where $\cF$ denotes the Fourier transform. It is well known that if $\gamma$ is a nonnegative integer, then
\begin{equation*}
H^{\gamma}_p
=\ggklam{u\in L_p \, :\, D^\alpha u\in L_p \text{ for all } \alpha \in\bN_0^d \text{ with } |\alpha |\leq \gamma}.
\end{equation*}

Let $\gdomain\subseteq \bR^d$ be an arbitrary domain with non-empty boundary $\partial\gdomain$. We denote by $\rho(x):=\rho_G(x):=\text{dist}(x,\partial\gdomain)$ the distance of a point $x\in\gdomain$ to the boundary $\partial\gdomain$. Furthermore, we fix  a bounded infinitely differentiable function $\psi$ defined on $\gdomain$ such that for all $x\in\gdomain$,
\begin{equation}\label{eq:PsiDist}
\rho(x)\leq N \psi(x),
\quad \rho(x)^{m-1}|D^{m}\psi(x)|\leq N(m)<\infty  \text{ for all } m\in \bN_0,
\end{equation}
where $N$ and $N(m)$ do not depend on $x\in\gdomain$. For a detailed construction of such a function see,  e.g., \cite[Chapter 3, Section 3.2.3]{Tri1995}.
Let $\zeta\in C^{\infty}_{0}(\bR_{+})$ be a non-negative  function
satisfying
\begin{equation}\label{eq:Part_start}
\sum_{n\in\bZ}\zeta(e^{n+t})>c>0 \text{ for all } t\in\bR. 
\end{equation}
Note that any non-negative smooth function $\zeta\in C^{\infty}_0(\bR_+)$ with $\zeta>0$ on $[e^{-1},e]$ satisfies (\ref{eq:Part_start}).
For $x\in \gdomain$ and $n\in\bZ$,  define
$$
\zeta_{n}(x):=\zeta(e^{n}\psi(x)).
$$
Then, there exists $k_0>0$ such that, for all $n\in\bZ$,  $\text{supp}\,\zeta_n \subset \gdomain_n:= \{x\in \gdomain:
e^{-n- k_0}<\rho(x)<e^{-n+ k_0}\}$, i.e.,
$\zeta_n \in C^{\infty}_0(\gdomain_n)$.
Moreover, $|D^m \zeta_n(x)|\leq N(\zeta,m) e^{mn}$ for all $x\in\gdomain$ and $m \in \bN_0$, and $\sum_{n\in\bZ}\zeta_{n}(x)\geq \delta >0 \text{ for all } x\in\gdomain$.
For $p\in (1,\infty)$ and $\gamma, \theta\in\bR$, we denote by $H^{\gamma}_{p,\theta}(\gdomain)$ the space of all
distributions $u$ on $\gdomain$ such that
\begin{equation*}
\|u\|_{H^{\gamma}_{p,\theta}(\gdomain)}^{p}:= \sum_{n\in\bZ}
e^{n\theta} \|\zeta_{-n}(e^{n} \cdot)u(e^{n}
\cdot)\|^p_{H^{\gamma}_p} < \infty.
\end{equation*}
It is well-known that
\begin{equation*}
L_{p,\theta}(\gdomain):=H^0_{p,\theta}(\gdomain)=L_p(\gdomain,\rho^{\theta-d}
dx),
\end{equation*}
and that, if $\gamma$ is a positive integer,
\begin{equation*}
H^{\gamma}_{p,\theta}(\gdomain)
=
\ggklam{u\in L_{p,\theta}(\gdomain) \, : \, 
\rho^{|\alpha|} D^\alpha u\in L_{p,\theta}(\gdomain)
\text{ for all } \alpha\in\bN_0^d \text{ with } |\alpha|\leq \gamma},
\end{equation*}
\begin{equation*}
\|u\|^p_{H^{\gamma}_{p,\theta}(\gdomain)}\sim \sum_{|\alpha|\leq
\gamma}\int_{\gdomain}\big|\rho^{|\alpha|}D^{\alpha}u\big|^p \rho^{\theta-d}\,dx;
\end{equation*}
see, e.g., \cite[Proposition 2.2]{Lot2000}. This is the reason why the space $H^\gamma_{p,\theta}(\gdomain)$ is called weighted Sobolev space of order $\gamma$, with summability parameter $p$ and weight parameter $\theta$.

For $p\in(1,\infty)$ and $\gamma\in\bR$ we write $H^\gamma_p(\ell_2)$ for the collection of all sequences $g=(g^{1},g^{2},\ldots)$ of distributions on $\bR^d$ with $g^{k}\in H^\gamma_p$ for each $k\in\bN$ and
\begin{equation*}
\nnrm{g}{H^\gamma_p(\ell_2)}
:=
\nnrm{g}{H^\gamma_p(\bR^d;\ell_2)}
:=
\nnrm{|(1-\Delta)^{\gamma/2}g|_{\ell_2}}{L_p}
:=
\sgnnrm{\sgrklam{\sum_{k=1}^{\infty}|(1-\Delta)^{\gamma/2}g^{k}|^2}^{1/2}}{L_p}
<
\infty.
\end{equation*}
Analogously,  for $\theta\in\bR$, a sequence $g=(g^{1},g^{2},\ldots)$ of distributions on $\gdomain$ is in $H^\gamma_{p,\theta}(\gdomain;\ell_2)$ if, and only if, $g^{k}\in H^\gamma_{p,\theta}(\gdomain)$ for each $k\in\bN$ and
\begin{equation*}
\nnrm{g}{H^{\gamma}_p(\gdomain;\ell_2)}^p
:=
\sum_{n\in\bZ} e^{n\theta} \|\zeta_{-n}(e^{n} \cdot)g(e^{n}
\cdot)\|^p_{H^{\gamma}_p(\ell_2)}
<
\infty.
\end{equation*}

Now we present some useful properties of the space
$H^{\gamma}_{p,\theta}(\gdomain)$ taken from \cite{Lot2000}, see  also \cite{Kry1999,Kry1999b}.

\begin{lemma}\label{lem:collection}
\textbf{\textup{(i)}} The space $C^{\infty}_0(\gdomain)$ is dense in
$H^{\gamma}_{p,\theta}(\gdomain)$.

\noindent\textbf{\textup{(ii)}} Assume that $\gamma-d/p=m+\nu$ for some $m\in\bN_0$, $\nu\in (0,1]$ and that $i,j\in\bN_0^d$ are multi-indices such that $|i|\leq m$ and $|j|=m$. Then for any $u\in H^{\gamma}_{p,\theta}(\gdomain)$, we have
$$
\psi^{|i|+\theta/p}D^iu \in C(\gdomain), \quad
\psi^{m+\nu+\theta/p}D^ju\in C^{\nu}(\gdomain),
$$
$$
|\psi^{|i|+\theta/p}D^iu|_{C(\gdomain)}+[\psi^{m+\nu+\theta/p}D^ju]_{C^{\nu}(\gdomain)}\leq
N \|u\|_{ H^{\gamma}_{p,\theta}(\gdomain)}.
$$

\noindent\textbf{\textup{(iii)}} 
$u\in H^\gamma_{p,\theta}(\gdomain)$ if, and only if, $u, \psi u_x\in H^{\gamma-1}_{p,\theta}(\gdomain)$ and
$$
\|u\|_{H^{\gamma}_{p,\theta}(\gdomain)} 
\leq 
N \|\psi u_x\|_{H^{\gamma-1}_{p,\theta}(\gdomain)}
+
N \|u\|_{H^{\gamma-1}_{p,\theta}(\gdomain)}
\leq 
N \|u\|_{H^{\gamma}_{p,\theta}(\gdomain)}.
$$
Also, $u\in H^\gamma_{p,\theta}(\gdomain)$ if, and only if, $u, (\psi u)_x\in H^{\gamma-1}_{p,\theta}(\gdomain)$ and
$$
\|u\|_{H^{\gamma}_{p,\theta}(\gdomain)} 
\leq 
N \|(\psi u)_x\|_{H^{\gamma-1}_{p,\theta}(\gdomain)}
+
N \|u\|_{H^{\gamma-1}_{p,\theta}(\gdomain)}
\leq 
N \|u\|_{H^{\gamma}_{p,\theta}(\gdomain)}.
$$

\noindent\textbf{\textup{(iv)}} For any $\nu, \gamma \in \bR$,
$\psi^{\nu}H^{\gamma}_{p,\theta}(\gdomain)=H^{\gamma}_{p,\theta-p\nu}(\gdomain)$
and
$$
\|u\|_{H^{\gamma}_{p,\theta-p\nu}(\gdomain)} 
\leq 
N \|\psi^{-\nu}u\|_{H^{\gamma}_{p,\theta}(\gdomain)}
\leq
N \|u\|_{H^{\gamma}_{p,\theta-p\nu}(\gdomain)}.
$$

\noindent\textbf{\textup{(v)}} If $\gamma\in (\gamma_0,\gamma_1)$ then, for any $\varepsilon>0$, there exists a constant $N=N(\gamma_0,\gamma_1,\theta,p,\varepsilon)$, such that
\begin{align*}
\|u\|_{H^{\gamma}_{p,\theta}(\gdomain)}
&\leq 
\varepsilon \|u\|_{H^{\gamma_1}_{p,\theta}(\gdomain)}
+
N(\gamma_0,\gamma_1,\theta,p,\varepsilon)\|u\|_{H^{\gamma_0}_{p,\theta}(\gdomain)}.
\end{align*}
Also, if $\theta\in (\theta_0,\theta_1)$ then, for any $\varepsilon>0$, there exists a constant $N=N(\theta_0,\theta_1,\gamma,p,\varepsilon)$, such that
\begin{align*}
\|u\|_{H^{\gamma}_{p,\theta}(\gdomain)}
&\leq 
\varepsilon \|u\|_{H^{\gamma}_{p,\theta_0}(\gdomain)}
+
N(\theta_0,\theta_1,\gamma,p,\varepsilon)\|u\|_{H^{\gamma}_{p,\theta_1}(\gdomain)}.
\end{align*}

\noindent\textbf{\textup{(vi)}}  There exists a constant $c_0>0$ depending on $p$, $\theta$, $\gamma$ and the function $\psi$ such that, for all $c\geq c_0$, the operator $\psi^2\Delta-c$ is a homeomorphism from $H^{\gamma+1}_{p,\theta}(\gdomain)$ to $H^{\gamma-1}_{p,\theta}(\gdomain)$. 
\end{lemma}

\begin{remark}\label{rem:NoMiddle}
Assertions (vi) and (iv) in Lemma \ref{lem:collection} imply the following: If $u\in H^{\gamma}_{p,\theta-p}(\gdomain)$ and $\Delta u \in H^{\gamma}_{p,\theta+p}(\gdomain)$, then $u\in H^{\gamma+2}_{p,\theta-p}(\gdomain)$ and there exists a constant $N$, which does not depend on $u$, such that
\begin{equation*}
\|u\|_{H^{\gamma+2}_{p,\theta-p}(\gdomain)}
\leq N 
\|\Delta u\|_{H^{\gamma}_{p,\theta+p}(\gdomain)}+N \|u\|_{H^{\gamma}_{p,\theta-p}(\gdomain)}.
\end{equation*}
\end{remark}

A proof of the following equivalent characterization of the weighted Sobolev spaces $H^\gamma_{p,\theta}(\gdomain)$ can be found in \cite[Proposition 2.2]{Lot2000}.

\begin{lemma}\label{lem:EquivWS}
Let $\{\xi_n : n\in\bZ\} \subseteq C^{\infty}_0(\gdomain)$ be such that  for all $n\in\bZ$ and $m\in\bN_0$,
\begin{equation}\label{eq:EquivWS}
|D^m\xi_n|\leq N(m)\, c^{nm} \quad\text{ and }\quad 
 \operatorname{supp} \xi_n \subseteq \{x\in \gdomain:
c^{-n-k_0}<\rho(x)<c^{-n+k_0}\}
\end{equation}
 for some
$c>1$ and $k_0>0$, where the constant $N(m)$ does not depend on $n\in\bZ$ and $x\in\gdomain$. Then, for any $u\in H^{\gamma}_{p,\theta}(\gdomain)$,
\begin{equation*}
\sum_{ n\in\bZ}  c^{n\theta}\|\xi_{-n}(c^n\cdot)u(c^n\cdot)\|^p_{H^{\gamma}_{p}}
\leq N\,
\|u\|^p_{H^{\gamma}_{p,\theta}(\gdomain)}.
\end{equation*}
If in addition
\begin{equation}\label{eq:EquivWS-b}
\sum_{ n\in\bZ} \xi_n(x)\geq\delta>0\text{ for all } x\in\gdomain
\end{equation}
then the converse inequality also holds.
\end{lemma}

\begin{remark}\label{rem:EquivWS}
\textup{\textbf{(i)}} 
It is easy to check that both 
$$\ggklam{\xi^{(1)}_n:=e^{-n} (\zeta_{n})_{x^i} \,:\, n\in\bZ} 
\quad\text{ and }\quad 
\ggklam{\xi^{(2)}_n:=e^{-2n}(\zeta_{n})_{x^ix^j}\,:\, n\in\bZ}$$ 
satisfy \eqref{eq:EquivWS} with $c:=e$. Therefore,
\begin{equation*}
\sum_{ n\in\bZ} e^{n\theta}
\sgrklam{\| e^n(\zeta_{-n})_{x^i}(e^n\cdot) u(e^n\cdot)\|^p_{H^{\gamma}_p}
+
\|e^{2n}(\zeta_{-n})_{x^ix^j}(e^n\cdot) u(e^n\cdot)\|^p_{H^{\gamma}_p}}
\leq N
\|u\|^p_{H^{\gamma}_{p,\theta}(\gdomain)}.
\end{equation*}

\noindent{\textup{\textbf{(ii)}}} 
Given $k_1\geq 1$, fix a function $\tilde{\zeta}\in C^\infty_0(\bR_+)$ with
\begin{equation*}
\tilde{\zeta}(t)=1 \quad \text{for all} \quad t\in\sgeklam{\frac{1}{N}\,2^{-k_1}\, ,\, N(0)\,2^{k_1}},
\end{equation*}
where $N$ and $N(0)$ are as in \eqref{eq:PsiDist}. Then, the sequence $\{\xi_n:n\in\bZ\}\subseteq C^\infty_0(\gdomain)$ defined by
\begin{equation*}
\xi_n := \tilde{\zeta}(2^n\psi(\cdot)),\qquad n\in\bZ,
\end{equation*}
fulfils the conditions \eqref{eq:EquivWS} and \eqref{eq:EquivWS-b} from Lemma \ref{lem:EquivWS} with $c=2$ and a suitable $k_0>0$. Furthermore, 
\begin{equation*}
\xi_n(x)=1 \quad \text{for all} \quad x\in \rho^{-1}\grklam{2^{-n}\geklam{2^{-k_1},2^{k_1}}}.
\end{equation*}
\end{remark}

In this paper, $\domain$ {\bf will always denote a bounded
Lipschitz domain in $\bR^d$}. More precisely:

\begin{defn}\label{domain}
We call a bounded domain $\cO\subset\bR^d$ a Lipschitz domain if, and only if, for any
$x_0=(x^1_0,x'_0)\in
\partial \cO$, there exists
 a Lipschitz continuous function
$\mu_0:\bR^{d-1}\to \bR$  such that, upon relabeling and reorienting
the  coordinate  axes if necessary, we have
\begin{itemize}
\item[{\bf(i)}] $\cO \cap B_{r_0}(x_0)=\{x=(x^1,x')\in B_{r_0}(x_0):
x^1>\mu_0(x')\}$, and
\item[{\bf(ii)}] $|\mu_0(x')-\mu_0(y')|\leq K_0|x'-y'|$, for any $x',y'\in
\bR^{d-1}$,
\end{itemize}
where $r_0,K_0$ are independent of $x_0$.
\end{defn}

\begin{remark}\label{kuf}
Recall that for a bounded Lipschitz domain $\domain\subseteq\bR^d$, 
\begin{equation*}
\mathring{W}^1_p(\domain) = H^{1}_{p,d-p}(\domain)
\end{equation*}
with equivalent norms. This follows from \cite[Theorem~9.7]{Kuf1980} and Poincar{\'e}'s inequality.
\end{remark}

\setcounter{equation}{0}
\section{Stochastic parabolic weighted Sobolev spaces and SPDEs}\label{StWS-00}

In this section, we first introduce the stochastic parabolic spaces $\frH^{\gamma,q}_{p,\theta}(\gdomain,T)$ for arbitrary domains $\gdomain\subseteq\bR^d$ with non-empty boundary in analogy to the spaces $\frH^{\gamma,q}_{p,\theta}(T)=\frH^{\gamma,q}_{p,\theta}(\bR^d_+,T)$ from \cite{Kry2000,Kry2001}. 
Then we show that they are suitable to serve as solution spaces for equations of type \eqref{eq:mainEq} in the following sense: If we have a solution $u\in\frH^{\gamma,q}_{p,\theta}(\gdomain,T)$ with low regularity $\gamma\geq 0$, but $f$ and the $g^k$'s have high $L_q(L_p)$-regularity, then we can lift up the regularity of the solution (Theorem~\ref{thm:liftReg}).

\begin{defn}
Let $G$ be a domain in $\bR^d$ with non-empty boundary. For  $p,q\in(1,\infty)$, $\gamma,\theta\in\bR$   and $T\in (0,\infty)$ we define
\begin{align*}
\bH^{\gamma,q}_{p,\theta}(\gdomain,T)&
:= L_q(\Omega\times [0,T], \cP ,\wP\otimes  dt; H^{\gamma}_{p,\theta}(\gdomain)),\\
\bH^{\gamma,q}_{p,\theta}(\gdomain,T;\ell_2)&
:= L_q(\Omega\times [0,T], \cP,\wP\otimes  dt; H^{\gamma}_{p,\theta}(\gdomain;\ell_2)),\\
U^{\gamma,q}_{p,\theta}(\gdomain)&
:=L_q(\Omega,\cF_0,\wP; \psi^{1-2/q} H^{\gamma-2/q}_{p,\theta}(\gdomain)).
\end{align*}
If $p=q$ we also write $\bH^{\gamma}_{p,\theta}(\gdomain,T)$, $\bH^{\gamma}_{p,\theta}(\gdomain,T;\ell_2)$ and $U^{\gamma}_{p,\theta}(\gdomain)$ instead of $\bH^{\gamma,p}_{p,\theta}(\gdomain,T)$, $\bH^{\gamma,p}_{p,\theta}(\gdomain,T;\ell_2)$ and $U^{\gamma,p}_{p,\theta}(\gdomain)$ respectively.
\end{defn}

From now on let
\begin{align*}
p\in[2,\infty),\quad q\in[2,\infty),\quad \gamma\in\bR, \quad \theta\in\bR.
\end{align*}

\begin{defn}\label{def:StWS}
Let $G$ be a domain in $\bR^d$ with non-empty boundary. 
We write $u \in \frH^{\gamma,q}_{p,\theta}(\gdomain,T)$  if, and only if,
 $u\in \bH^{\gamma,q}_{p,\theta-p}(\gdomain,T)$,
$u(0,\cdot)\in U^{\gamma,q}_{p,\theta}(\gdomain)$,
and there exist some
$f \in \bH^{\gamma-2,q}_{p,\theta+p}(\gdomain,T)$ and
$g \in \bH^{\gamma-1,q}_{p,\theta}(\gdomain,T;\ell_2)$ such that
\begin{equation*}
du=f\,dt +g^k \,dw^k_t
\end{equation*}
in the sense of distributions. That is, for any $\varphi \in
C^{\infty}_{0}(\gdomain)$,  with probability one,  the equality
\begin{equation*}
(u(t,\cdot),\varphi)= (u(0,\cdot),\varphi) + \int^{t}_{0}
(f(s,\cdot),\varphi) \, ds + \sum^{\infty}_{k=1} \int^{t}_{0}
(g^k(s,\cdot),\varphi)\, dw^k_s
\end{equation*}
holds for all $t\in[0,T]$,  where the series is assumed to converge uniformly on $[0,T]$ in probability.  In this situation
we  write $\bD u:=f$ and $\bS u:=g$.
The norm in $\frH^{\gamma,q}_{p,\theta}(\gdomain,T)$ is defined as
$$
\|u\|_{\frH^{\gamma,q}_{p,\theta}(\gdomain,T)}:=
\|u\|_{\bH^{\gamma,q}_{p,\theta-p}(\gdomain,T)} + \|\bD
u\|_{\bH^{\gamma-2,q}_{p,\theta+p}(\gdomain,T)}
 +
\|\bS u\|_{\bH^{\gamma-1,q}_{p,\theta}(\gdomain,T;\ell_2)} +
\|u(0,\cdot)\|_{U^{\gamma,q}_{p,\theta}(\gdomain)}.
$$
If $p=q$ we also write $\frH^{\gamma}_{p,\theta}(\gdomain,T)$ instead of $\frH^{\gamma,p}_{p,\theta}(\gdomain,T)$.
\end{defn}

\begin{remark}\label{rem:Rd-spaces}
Replacing $\gdomain$ by $\bR^d$ and omitting the weight parameter $\theta$ and the weight function $\psi$ in the definitions above, one obtains the spaces 
$\bH^{\gamma,q}_p(T)=\bH^{\gamma,q}_p(\bR^d,T)$, 
$\bH^{\gamma,q}_{p}(T;\ell_2)=\bH^{\gamma,q}_{p}(\bR^d,T;\ell_2)$,
$U^{\gamma,q}_p=U^{\gamma,q}_p(\bR^d)$, and
$\mathscr{H}^{\gamma,q}_p(T)$ as introduced in \cite[Definition 3.5]{Kry2001}. The latter are denoted by $\mathcal{H}^{\gamma,q}_{p}(T)$ in \cite{Kry2000}; if  $q=p$ they coincide with the spaces $\mathcal{H}^{\gamma}_p(T)$ introduced in \cite[Definition 3.1]{Kry1999}.
\end{remark}

We consider initial value problems of the form
\begin{equation}\label{eq:mainEqa}
du = (a^{ij} u_{x^i x^j} + f)dt +(\sigma^{ik} u_{x^i} + g^k) d w^k_t, \qquad u(0,\cdot)=u_0,
\end{equation}
on an arbitrary domain $\gdomain\subseteq\bR^d$ with non-empty boundary. 
We use the following solution concept.
\begin{defn}\label{def:solution}
We say that a stochastic process $u\in \frH^{\gamma,q}_{p,\theta}(\gdomain,T)$ is a solution of Eq.\ \eqref{eq:mainEqa} if, and only if, 
\begin{equation*}
u(0,\cdot)=u_0,\quad 
\bD u =  a^{ij} u_{x^i x^j} + f, \quad \text{and} \quad
\bS u = \grklam{\sigma^{ik} u_{x^i} + g^k}_{k\in\bN},
\end{equation*} 
in the sense of Definition \ref{def:StWS}.
\end{defn}

\begin{remark} 
Here and in the sequel we use the summation convention on the repeated indices $i,j,k$.
The question, in which sense, for a bounded Lipschitz domain $\domain\subseteq\bR^d$, the elements of $\frH^{\gamma,q}_{p,\theta}(\domain,T)$ fulfil a zero Dirichlet boundary condition as in Eq.\ \eqref{eq:mainEq}, will be answered in Remark~\ref{rem:Dirichlet}. 
\end{remark}

We make the following assumptions on the coefficients in Eq.\ \eqref{eq:mainEqa}. Throughout this paper, whenever we will talk about this equation, we will assume that they are fulfilled. 

\begin{assumption}\label{ass:coeff}
\textbf{\textup{(i)}} 
The coefficients $a^{ij}=a^{ij}(\omega,t)$ and
$\sigma^{ik}=\sigma^{ik}(\omega,t)$
are predictable. They do not depend on $x\in\gdomain$. Furthermore, $a^{ij}=a^{ji}$ for $i,j\in\{1,\ldots,d\}$.

\noindent\textbf{\textup{(ii)}} 
There exist  constants  $\delta_0,K>0$ such that  for any
$(\omega,t)\in\Omega\times[0,T]$ and $\lambda \in \bR^d$,
\begin{equation*}
\delta_0 |\lambda|^2 \leq \bar{a}^{ij}(\omega,t) \lambda^{i} \lambda^{j}
 \leq K |\lambda|^2,
\end{equation*}
where $\bar{a}^{ij}(\omega,t):=a^{ij}(\omega,t)-\frac{1}{2}(\sigma^{i\cdot}(\omega,t),\sigma^{j\cdot}(\omega,t))_{\ell_2}$, with 
$\sigma^{i\cdot}(\omega,t)=\grklam{\sigma^{ik}(\omega,t)}_{k\in\bN}\in\ell_2$.
\end{assumption}

We will use the following result taken from \cite[Lemma 2.3]{Kry2000}. 

\begin{lemma}\label{lem:Kry2000_2.3}
Let $p\geq 2$, $m\in\bN$, and, for
$i=1,2,\ldots,m$,
\begin{equation*}
\lambda_i\in (0,\infty), \quad 
\gamma_i \in \bR, \quad 
u^{(i)}\in \cH^{\gamma_i +2}_p(T), \quad u^{(i)}(0,\cdot)=0.
\end{equation*}
Denote $\Lambda_i:=(\lambda_i-\Delta)^{\gamma_i/2}$. Then
\begin{align*}
\E\sgeklam{\int^T_0 \prod_{i=1}^m \|\Lambda_i \Delta u^{(i)}\|^p_{ L_p} dt}
&\leq  N 
\sum_{i=1}^m \E\sgeklam{ \int^T_0 \left(\|\Lambda_i f^{(i)}\|^p_{ L_p} +\|\Lambda_ig^{(i)}_x\|^p_{ L_p(\ell_2)}\right) 
\prod_{\substack{j=1 \\ j\neq i}}^{m} \|\Lambda_j \Delta u^{(j)}\|^p_{ L_p} dt}\\
&\quad + N 
\sum_{1\leq i<j\leq m} \E \sgeklam{ \int^T_0 \|\Lambda_ig^{(i)}_x\|^p_{
L_p(\ell_2)} \|\Lambda_jg^{(j)}_x\|^p_{L_p(\ell_2)} \prod_{\substack{k=1 \\ k\neq i,j}}^{m}\|\Lambda_k
\Delta u^{(k)}\|^p_{ L_p} dt},
\end{align*}
where $f^{(i)}:=\bD u^{(i)}-a^{rs}u^{(i)}_{x^rx^s}$,  $g^{(i)k}:=\bS^k u^{(i)}-\sigma^{rk}u^{(i)}_{x^r}$ and $L_p(\ell_2):=H^0_p(\ell_2)$. The constant $N$ depends only on $m$, $d$, $p$, $\delta_0$, and $K$. 
\end{lemma}

Now we are able to prove that if we have a solution $u\in \frH^{\gamma+1,q}_{p,\theta}(\gdomain,T)$ to Eq.\ \eqref{eq:mainEqa} and if the regularity of the forcing terms $f$ and $g$ is high then we can lift the regularity of the solution. Note that in the next theorem there is no restriction, neither on the shape of the domain $\gdomain \subseteq\bR^d$ nor on the parameters $\theta,\gamma\in\bR$.

\begin{thm}\label{thm:liftReg}
Let  $\gdomain\subseteq \bR^d$ be an arbitrary domain with non-empty boundary.
Let $\gamma\in\bR$, $p\geq2$ and  $q=pm$ for some $m\in\bN$. Let $f\in
\bH^{\gamma,q}_{p,\theta+p}(\gdomain,T)$,  
$g\in\bH^{\gamma+1,q}_{p,\theta}(\gdomain,T;\ell_2)$ and let 
$u\in\frH^{\gamma+1,q}_{p,\theta}(\gdomain,T)$ be a solution to Eq.\ \eqref{eq:mainEqa} with $u_0 = 0$.
Then  $u\in\frH^{\gamma+2,q}_{p,\theta}(\gdomain,T)$, and
\begin{equation*}
\|u\|^q_{\bH^{\gamma+2,q}_{p,\theta-p}(\gdomain,T)} 
\leq N
\sgrklam{
\|u\|^q_{\bH^{\gamma+1,q}_{p,\theta-p}(\gdomain,T)}
+
\|f\|_{\bH^{\gamma,q}_{p,\theta+p}(\gdomain,T)}^{q}
+
\|g\|^{q}_{\bH^{\gamma+1,q}_{p,\theta}(\gdomain,T;\ell_2)}},
\end{equation*}
where the constant $N\in (0,\infty)$ does not depend on $u$, $f$ and $g$.
\end{thm}

\begin{proof} 
The case $m=1$, i.e., $p=q$ is covered by \cite[Lemma 3.2]{Kim2011}. 
Therefore, let $m\geq 2$. According to Remark~\ref{rem:NoMiddle} it is enough to show that
\begin{equation*}
\|\Delta u\|^q_{\bH^{\gamma,q}_{p,\theta+p}(\gdomain,T)}
\leq N
\sgrklam{
\|u\|^q_{\bH^{\gamma+1,q}_{p,\theta-p}(\gdomain,T)}
+
\|f\|_{\bH^{\gamma,q}_{p,\theta+p}(\gdomain,T)}^{ q} 
+
\|g\|^{q}_{\bH^{\gamma+1,q}_{p,\theta}(\gdomain,T;\ell_2)}
}.
\end{equation*}
Using the definition of weighted Sobolev spaces from Section \ref{WS}, we observe that
\begin{align*}
\nnrm{\Delta u}{\bH^{\gamma,q}_{p,\theta+p}(\gdomain,T)}^q
&=
\E \sgeklam{
\int_0^T \sgrklam{
\sum_{n\in\bZ} e^{n(\theta+p)} \nnrm{(\zeta_{-n}\Delta u(t))(e^n\cdot)}{H^{\gamma}_p}^p
}^m  dt} \\
&\leq N\,
\E \sgeklam{
\int_0^T \sgrklam{
\sum_{n\in\bZ} e^{n(\theta+p)} 
\sgrklam{
\nnrm{\Delta(\zeta_{-n}u(t))(e^n\cdot)}{H^{\gamma}_p}^p\\
&\phantom{\leq N \E \int_0^T ( \sum e^{n(\theta+p}( }+
\nnrm{(\Delta \zeta_{-n}u(t))(e^n\cdot)}{H^{\gamma}_p}^p
+
\nnrm{(\zeta_{-n x}u_x(t))(e^n\cdot)}{H^{\gamma}_p}^p
}
}^m  dt}.
\end{align*}
(Here $\zeta_{-n x}u_x$ is meant to be a scalar product in $\bR^d$.)
Now we can use Jensen's inequality and Remark \ref{rem:EquivWS}(i) to obtain
\begin{align*}
\nnrm{\Delta u}{\bH^{\gamma,q}_{p,\theta+p}(\gdomain,T)}^q
&\leq N \,
\E \sgeklam{
\int_0^T \sgrklam{
\sum_{n\in\bZ} e^{n(\theta+p)} 
\nnrm{\Delta(\zeta_{-n}u(t))(e^n\cdot)}{H^{\gamma}_p}^p}^m \\
&\phantom{N \E \sgeklam{\int_0^T \sgrklam{\sum_{n\in\bZ} e^{n(\theta+p)}}}}+
\nnrm{u(t)}{H^{\gamma}_{p,\theta-p}(\gdomain)}^q
+
\nnrm{u_x(t)}{H^{\gamma}_{p,\theta}(\gdomain)}^q
dt}.
\end{align*}
An application of Lemma \ref{lem:collection}(iii) and (iv) leads to
\begin{align*}
\nnrm{\Delta u}{\bH^{\gamma,q}_{p,\theta+p}(\gdomain,T)}^q
&\leq N\, 
\E \sgeklam{
\int_0^T \sgrklam{
\sum_{n\in\bZ} e^{n(\theta+p)} 
\nnrm{\Delta(\zeta_{-n}u(t))(e^n\cdot)}{H^{\gamma}_p}^p}^m dt}
+ N\,
\nnrm{u}{\bH^{\gamma+1,q}_{p,\theta-p}(\gdomain,T)}^q.
\end{align*}
Therefore, it is enough to estimate the first term on the right hand side, 
\begin{align*}
\E \sgeklam{ &
\int_0^T \sgrklam{
\sum_{n\in\bZ} e^{n(\theta+p)} 
\nnrm{\Delta(\zeta_{-n}u(t))(e^n\cdot)}{H^{\gamma}_p}^p}^m dt} \\
&=
\E \sgeklam{\int_0^T \sum_{n_1,\ldots,n_m\in\bZ} \!\!\!\! e^{\grklam{\sum_{i=1}^m n_i}(\theta+p)}
\prod_{i=1}^m \nnrm{\Delta(\zeta_{-n_i}u(t))(e^{n_i}\cdot)}{H^{\gamma}_p}^p
\, dt}.
\end{align*}
Tonelli's theorem together with the relation
\begin{equation}\label{eq:trivial}
\|u(c\,\cdot)\|^p_{H^{\gamma}_p}=c^{p\gamma-d}\|(c^{-2}-\Delta)^{\gamma/2}u\|^p_{
L_p} \quad \text{for } c\in (0,\infty),
\end{equation}
applied to $\Delta u^{(n_i)}$ with $u^{(n)}:=\zeta_{-n}u$ for $n\in \bZ$, show that we only have to handle
\begin{equation*}
\sum_{n_1,\ldots,n_m\in\bZ} \!\!\!\! e^{\grklam{\sum_{i=1}^m n_i}(\theta+p+p\gamma-d)}\,
\E \sgeklam{\int_0^T 
\prod_{i=1}^m \nnrm{(e^{-2 n_i}-\Delta)^{\gamma/2}\Delta u^{(n_i)}(t)}{L_p}^p
\, dt}.
\end{equation*}
Note that since $u\in \frH^{\gamma+1,q}_{p,\theta}(\gdomain,T)$ solves Eq.\ \eqref{eq:mainEqa} with vanishing initial value, $u^{(n)}$ is a solution of the equation
\begin{equation*}
dv=(a^{rs} v_{x^rx^s}
+f^{(n)})dt+(\sigma^{rk} v_{x^r}+g^{(n)k})dw^k_t, \qquad v(0,\cdot)=0,
\end{equation*}
on $\bR^d$, where 
$f^{(n)}=-2 a^{rs} (\zeta_{-n})_{x^{s}} u_{x^r} 
- a^{rs} (\zeta_{-n})_{x^rx^s} u+ \zeta_{-n} f$ and
$g^{(n)k}=-\sigma^{rk} (\zeta_{-n})_{x^r} u + \zeta_{-n} g^k$. Furthermore, applying \cite[Theorem 4.10]{Kry1999}, we have $u^{(n)}\in \cH^{\gamma+2}_{p}(T)$. Thus, we can use Lemma \ref{lem:Kry2000_2.3} to obtain
\begin{align*}
\E \sgeklam{\int_0^T 
\prod_{i=1}^m \nnrm{(e^{-2 n_i}-\Delta)^{\gamma/2}\Delta u^{(n_i)}(t)}{L_p}^p
\, dt}
&\leq  N 
\sum_{i=1}^m \grklam{I_{n_i}+I\!\! I_{n_i}}
+ N \sum_{1\leq i<j \leq m} \!\!I\!\! I\!\! I_{n_i n_j}
\end{align*}
where we denote
\begin{align*}
I_{n_i} &:= \E \sgeklam{ \int^T_0 \|\Lambda_{n_i} f^{(n_i)}(t)\|^p_{ L_p} 
\prod_{\substack{j=1 \\ j\neq i}}^{m} \|\Lambda_{n_j} \Delta u^{(n_j)}(t)\|^p_{ L_p} dt},\\
I\!\! I_{n_i} &:= \E \sgeklam{ \int^T_0 \|\Lambda_{n_i}g^{(n_i)}_x(t)\|^p_{L_p(\ell_2)}
\prod_{\substack{j=1 \\ j\neq i}}^{m} \|\Lambda_{n_j} \Delta u^{(n_j)}(t)\|^p_{ L_p} dt},\\
I\!\!I\!\!I_{n_i n_j} &:= \E \sgeklam{ \int^T_0 \|\Lambda_{n_i}g^{(n_i)}_x(t)\|^p_{
L_p(\ell_2)} \|\Lambda_{n_j}g^{(n_j)}_x (t)\|^p_{L_p(\ell_2)} \prod_{\substack{k=1 \\ k\neq i,j}}^{m}\|\Lambda_{n_k}
\Delta u^{(n_k)}(t)\|^p_{ L_p} dt},
\end{align*}
with $\Lambda_{n}:=(e^{-2n}-\Delta)^{\gamma/2}$.
Thus, it is enough to find a proper estimate for
\begin{equation*}
\sum_{n_1,\ldots,n_m\in\bZ} \!\!\!\! e^{\grklam{\sum_{i=1}^m n_i}(\theta+p+p\gamma-d)}
\sgrklam{\sum_{i=1}^m \grklam{I_{n_i}+I\!\!I_{n_i}}+\sum_{1\leq i<j\leq m} I\!\!I\!\!I_{n_i n_j}}.
\end{equation*}
Applying \eqref{eq:trivial} first, followed by  Tonelli's theorem, then H\"older's
and Young's 
inequality,  
leads to
\begin{align*}
&\sum_{n_1,\ldots,n_m\in\bZ} \!\!\!\! e^{\grklam{\sum_{i=1}^m n_i}(\theta+p+p\gamma-d)}
\sum_{i=1}^m I_{n_i}\\
&\,\, =
\sum_{n_1,\ldots,n_m\in\bZ} \!\!\!\! e^{\grklam{\sum_{i=1}^m n_i}(\theta+p)}
\sum_{i=1}^m 
\E \sgeklam{ \int^T_0 \| f^{(n_i)}(t,e^{n_i}\cdot)\|^p_{ H^{\gamma}_p} 
\prod_{\substack{j=1 \\ j\neq i}}^{m} \| \Delta u^{(n_j)}(t,e^{n_j}\cdot)\|^p_{ H^{\gamma}_p} \, dt}\\
&\,\, \leq N\,
\E \sgeklam{ \int^T_0 
\sgrklam{\sum_{n\in\bZ} e^{n(\theta +p)} \nnrm{f^{(n)}(t,e^n\cdot)}{H^\gamma_p}^p}
\sgrklam{\sum_{n\in\bZ} e^{n(\theta +p)} \nnrm{\Delta u^{(n)}(t,e^n\cdot)}{H^\gamma_p}^p}^{m-1}
dt}\\
&\,\, \leq N(\varepsilon)\, 
\E \sgeklam{\int_{0}^{T} \sgrklam{\sum_{n\in\bZ} e^{n(\theta +p)} \nnrm{f^{(n)}(t,e^n\cdot)}{H^\gamma_p}^p}^{\frac{q}{p}} \,dt}
+ \varepsilon \,
\E \sgeklam{ \int^T_0 
\sgrklam{\sum_{n\in\bZ} e^{n(\theta +p)} \nnrm{\Delta u^{(n)}(t,e^n\cdot)}{H^\gamma_p}^p}^\frac{q}{p}
dt}.
\end{align*}
Using the definition of $f^{(n)}$ and arguing as at the beginning of the proof, we get
\begin{align*}
\sum_{n\in\bZ} e^{n(\theta +p)} \nnrm{f^{(n)}(t,e^n\cdot)}{H^\gamma_p}^p
&\leq N\sgrklam{
\nnrm{u_x(t)}{H^\gamma_{p,\theta}(\gdomain)}^p
+ 
\nnrm{u(t)}{H^\gamma_{p,\theta-p}(\gdomain)}^p
+
\nnrm{f(t)}{H^{\gamma}_{p,\theta+p}(\gdomain)}^p}\\
&\leq N\sgrklam{
\nnrm{u(t)}{H^{\gamma+1}_{p,\theta-p}(\gdomain)}^p
+
\nnrm{f(t)}{H^{\gamma}_{p,\theta+p}(\gdomain)}^p}.
\end{align*}
Moreover,
\begin{align*}
\sum_{n\in\bZ} e^{n(\theta +p)} \nnrm{\Delta u^{(n)}(t,e^n\cdot)}{H^\gamma_p}^p
&\leq
\sum_{n\in\bZ} e^{n(\theta +p)} \nnrm{(\Delta \zeta_{-n} u (t))(e^n\cdot)}{H^\gamma_p}^p	\\
&\phantom{\leq }+
\sum_{n\in\bZ} e^{n(\theta +p)} \nnrm{(\zeta_{-nx} u_x (t))(e^n\cdot)}{H^\gamma_p}^p
+
\sum_{n\in\bZ} e^{n(\theta +p)} \nnrm{(\zeta_{-n} \Delta u(t))(e^n\cdot)}{H^\gamma_p}^p	\\
&\leq N\sgrklam{ 
\nnrm{u(t)}{H^{\gamma}_{p,\theta-p}(\gdomain)}^p
+
\nnrm{u_x(t)}{H^{\gamma}_{p,\theta}(\gdomain)}^p
+
\nnrm{\Delta u}{H^{\gamma}_{p,\theta+p}(\gdomain)}^p}	\\
&\leq N\sgrklam{
\nnrm{u(t)}{H^{\gamma+1}_{p,\theta-p}(\gdomain)}^p
+
\nnrm{\Delta u}{H^{\gamma}_{p,\theta+p}(\gdomain)}^p}.
\end{align*}
Combining the last three estimates, we obtain for any $\varepsilon>0$ a constant $N(\varepsilon)\in (0,\infty)$, such that
\begin{align*}
&\sum_{n_1,\ldots,n_m\in\bZ} \!\!\!\! e^{\grklam{\sum_{i=1}^m n_i}(\theta+p+p\gamma-d)} \sum_{i=1}^m I_{n_i}
\leq 
\varepsilon\, \nnrm{\Delta u}{\bH^{\gamma,q}_{p,\theta+p}(\gdomain,T)}^q
+
N(\varepsilon)\, \sgrklam{
\nnrm{f}{\bH^{\gamma,q}_{p,\theta+p}(\gdomain,T)}^q
+
\nnrm{u}{\bH^{\gamma+1,q}_{p,\theta-p}(\gdomain,T)}^q}.
\end{align*}
Using similar arguments we obtain
\begin{align*}
\sum_{n_1,\ldots,n_m\in\bZ} \!\!\!\! e^{\grklam{\sum_{i=1}^m n_i}(\theta+p+p\gamma-d)} &\sgrklam{
\sum_{i=1}^m I\!\!I_{n_i}
+\sum_{1\leq i<j\leq m} I\!\!I\!\!I_{n_i n_j}}\\
&\leq 
\varepsilon\, \nnrm{\Delta u}{\bH^{\gamma,q}_{p,\theta+p}(\gdomain,T)}^q
+
N(\varepsilon)\, \sgrklam{
\nnrm{g}{\bH^{\gamma+1,q}_{p,\theta}(\gdomain,T;\ell_2)}^q
+
\nnrm{u}{\bH^{\gamma+1,q}_{p,\theta-p}(\gdomain,T)}^q},
\end{align*}
which finishes the proof.
\end{proof}

Iterating this result has the following consequence.

\begin{corollary}\label{cor:liftReg}
Let $\gamma\geq 1$, $p\in[2,\infty)$ and $q=mp$ for some $m\in\bN$.
Furthermore, assume that 
$f\in \bH^{\gamma-2,q}_{p,\theta+p}(\gdomain,T)$,
$g\in\bH^{\gamma-1,q}_{p,\theta}(\gdomain,T;\ell_2)$, and that
$u\in 
\bH^{0,q}_{p,\theta-p}(\gdomain,T)$ satisfies
Eq.\ \eqref{eq:mainEqa} with $u_0=0$. 
Then $u\in\frH^{\gamma,q}_{p,\theta}(\gdomain,T)$, and 
\begin{equation*}
\|u\|^q_{\frH^{\gamma,q}_{p,\theta}(\gdomain,T)}
\leq N
\sgrklam{
\|u\|^q_{\bH^{0,q}_{p,\theta-p}(\gdomain,T)}
+
\|f\|^q_{\bH^{\gamma-2,q}_{p,\theta+p}(\gdomain,T)}
+
\|g\|^q_{\bH^{\gamma-1,q}_{p,\theta}(\gdomain,T;\ell_2)}},
\end{equation*}
where the constant $N\in(0,\infty)$ does not depend on $u$, $f$ and $g$.
\end{corollary}

\begin{remark}
An extension of the results above to the case where the coefficients depend on the space variable $x\in\gdomain$ can be proved along the lines of \cite{Kim2008, Kim2011}. Also, the symmetry of $a^{ij}$ can be dropped. To keep the expositions at a reasonable level, we do not discuss these cases.
\end{remark}

\setcounter{equation}{0}
\section{Solvability of SPDEs within $\frH^{\gamma,q}_{p,\theta}(\domain,T)$}\label{SPDEs}

In this section we prove existence and uniqueness of solutions to equations of type \eqref{eq:mainEqa} on bounded Lipschitz domains $\domain\subseteq\bR^d$ in the spaces $\frH^{\gamma,q}_{p,\theta}(\domain,T)$. 
We are mainly interested in the case $q>p$. 
The main ingredient will be Corollary~\ref{cor:liftReg} which allows us to lift up the regularity of the solution once we have established a certain low $L_q(L_p)$-regularity and if $f$ and the $g^k$'s have high $L_q(L_p)$-regularity. 
In this section we restrict ourselves to equations of type \eqref{eq:mainEqa} with $\sigma\equiv 0$ and vanishing initial condition, i.e., we consider the problem
\begin{equation}\label{eq:mainEqa2}
du = (a^{ij}u_{x^ix^j}+f)dt +g^kdw^k_t,\qquad u(0,\cdot)=0.
\end{equation}
We expect, however, that the lifting argument in Corollary~\ref{cor:liftReg} can be used to derive similar results for general equations of type \eqref{eq:mainEqa}.
We establish existence of solutions with low $L_q(L_p)$-regularity  in two different ways which correspond to two different restrictions in our assumptions. 
First, in Subsection \ref{subsec:LqLp_Dong} we consider Lipschitz domains with sufficiently small Lipschitz constants. Here we use an $L_q(L_p)$-regularity result for deterministic PDEs and basic estimates for stochastic integrals in UMD Banach spaces to derive a result for general integrability parameters $q\geq p\geq2$.
Then, in Subsection \ref{subsec:LqLp_heat} we consider the case of general bounded Lipschitz domains. 
Applying techniques from the semigroup approach to stochastic evolution equations in Banach spaces, we are able prove existence and uniqueness of solutions of the stochastic heat equation in $\frH^{\gamma,q}_{p,d}(\domain,T)$ for integrability parameters $p\in[2,p_0)$ and $q\geq p$.

\subsection{A result for domains with small Lipschitz constant}\label{subsec:LqLp_Dong}

We need the following result concerning existence and uniqueness of
solutions to SPDEs of the form \eqref{eq:mainEqa} in
$\frH^{\gamma}_{p,\theta}(\domain,T)=\frH^{\gamma,p}_{p,\theta}(\domain,T)$, i.e., for the case $p=q$. 
It is taken from \cite{Kim2011}, see Theorem 2.12 and Remark 2.13 therein. 
Note that it also holds under weaker assumptions on the parameters and for more general equations than stated here.

\begin{thm}\label{thm:LpLp}
Let $\domain$ be a bounded Lipschitz domain in $\bR^d$ and $\gamma\in\bR$. For  $i,j\in\{1,\ldots,d\}$ and $k\in\bN$, let $a^{ij}$, $\sigma^{ik}$ be given coefficients satisfying Assumption \ref{ass:coeff}.
\\
\emph{\textbf{(i)}} For $p\in[2,\infty)$, there exists a constant
$\kappa_0\in(0,1)$, depending only on $d$, $p$, $\delta_0$, $K$ and $\domain$, such that  for any
$\theta\in(d+p-2-\kappa_0,d+p-2+\kappa_0)$, $f\in\bH^{\gamma}_{p,\theta+p}(\domain,T)$,
$g\in\bH_{p,\theta}^{\gamma+1}(\mathcal O,T;\ell_2)$ and $u_0\in
U^{\gamma+2}_{p,\theta}(\mathcal O)$, Eq.\ \eqref{eq:mainEqa} has a
unique solution $u$ in the class $\mathfrak
H^{\gamma+2}_{p,\theta}(\domain,T)$. For this solution
\begin{equation}\label{eq:LpLpEst}
\|u\|_{\frH^{\gamma+2}_{p,\theta}(\domain,T)}^p\leq
N\sgrklam{\|f\|_{\bH^{\gamma}_{p,\theta+p}(\domain,T)}^p
+\|g\|_{\bH^{\gamma+1}_{p,\theta}(\domain,T;\ell_2)}^p
+\|u_0\|_{U^{\gamma+2}_{p,\theta}(\domain)}^p
},
\end{equation}
where the constant $N$ depends only on $d$, $p$, $\gamma$, $\theta$,
$\delta_0$, $K$, $T$ and $\mathcal O$.
\\
\emph{\textbf{(ii)}} There exists $p_0>2$, such that the following
statement holds: if $p\in[2,p_0)$, then there exists a constant
$\kappa_1\in(0,1)$, depending only on $d$, $p$, $\delta_0$, $K$ and $\domain$, such that  for any
$\theta\in(d-\kappa_1,d+\kappa_1)$, $f\in\bH^{\gamma}_{p,\theta+p}(\domain,T)$,
$g\in\bH_{p,\theta}^{\gamma+1}(\domain,T;\ell_2)$ and $u_0\in
U^{\gamma+2}_{p,\theta}(\domain)$,  Eq.\ \eqref{eq:mainEqa} has a
unique solution $u$ in the class
$\frH^{\gamma+2}_{p,\theta}(\domain,T)$. For this solution, estimate
\eqref{eq:LpLpEst} holds.
\end{thm}

Here is the main result of this subsection.

\begin{thm}\label{thm:LpLq}
Let $\domain$ be a bounded Lipschitz domain in $\bR^d$ and let $\gamma\geq 0$. 
For given coefficients $a^{ij}$, $i,j\in\{1,\ldots,d\}$, let Assumption \ref{ass:coeff} be satisfied with $\sigma\equiv 0$. 
Then, for $2\leq p\leq q<\infty$, there exists a constant $c=c(d, p, q, \delta_0, K)$, such that, if the Lipschitz constant $K_0$ in Definition~\ref{domain} satisfies $K_0\leq c$, the following holds: For 
\[f\in\bH^{\gamma,q}_{p,d+p}(\domain,T)\cap\bH^{0,q}_{p,d}(\domain,T)\quad
\text{ and } \quad g\in\bH^{\gamma+1,q}_{p,d}(\domain,T;\ell_2)\cap\bH^{1,q}_{p,d-p}(\domain,T;\ell_2),\]
Eq.~\eqref{eq:mainEqa2} has a unique solution $u\in\frH^{\gamma+2,q}_{p,d}(\domain,T)$. 
Moreover, there exists a constant $N\in (0,\infty)$, which does not depend on $u$, $f$ and $g$, such that
\begin{equation}\label{eq:LpLqEst}
\begin{aligned}
\|u\|_{\frH^{\gamma+2,q}_{p,d}(\domain,T)}	
\leq N \sgrklam{
\|f&\|_{\bH^{\gamma,q}_{p,d+p}(\domain,T)}
+
\|f\|_{\bH^{0,q}_{p,d}(\domain,T)}	\\
&+
\|g\|_{\bH^{\gamma+1,q}_{p,d}(\domain,T;\ell_2)}
+
\|g\|_{\bH^{1,q}_{p,d-p}(\domain,T;\ell_2)}
}.
\end{aligned}
\end{equation}
\end{thm}

\begin{remark}
We note that every bounded $C^1$-domain $\domain\subseteq\bR^d$ is a Lipschitz domain where the Lipschitz constant $K_0$ in Definition~\ref{domain} can be chosen arbitrarily small. Therefore, the assertion of Theorem~\ref{thm:LpLq} holds for any bounded $C^1$-domain.
\end{remark}

\begin{proof}[Proof of Theorem~\ref{thm:LpLq}]
Since $ \frH^{\gamma+2,q}_{p,d}(\domain,T) \hookrightarrow \frH^{\gamma+2}_{2,d}(\domain,T)$, the uniqueness follows from Theorem~\ref{thm:LpLp}.  Also by Theorem \ref{thm:LpLp}, there exists a solution $u\in \frH^{\gamma+2}_{2,d}(\domain,T)$. We only need to show that  $u\in \frH^{\gamma+2,q}_{p,d}(\domain,T)$ and that it satisfies
(\ref{eq:LpLqEst}).
For all $\varphi\in C_0^\infty(\domain)$, with probability one,
\begin{equation*}\label{exRes2-1}
\big(u(t),\varphi\big)=\int_0^t\big(a^{ij}(s)u_{x^ix^j}(s)+f(s),\varphi\big)ds+
\sum_{k=1}^\infty\int_0^t\big(g^k(s),\varphi\big)dw^k_s, \quad t\in[0,T].
\end{equation*}
Let us define
\[w(t):=\sum_{k=1}^\infty\int_0^t g^k(s)dw^k_s\]
as an infinite sum of, say $H^1_{2,d-2}(\domain)$-valued stochastic integrals. 
This sum converges in the space $\bH^1_{2,d-2}(\domain,T)$ due to It{\^o}'s isometry and since $g\in\bH^{1,q}_{p,d-p}(\domain,T;\ell_2)\hookrightarrow\bH^1_{2,d-2}(\domain,T;\ell_2)$. 
We fix a continuous modification of the $H^1_{2,d-2}(\domain)$-valued process $(w(t))_{t\in[0,T]}$ from now on, which is well-known to exist. For all $\varphi\in C_0^\infty(\domain)$, with probability one,
\begin{equation*}
\sum_{k=1}^{\infty} \int_0^t\big(g^k(s),\varphi\big)dw^k_s=(w(t),\varphi), \quad t\in[0,T].
\end{equation*}
Therefore, setting $\bar u:=u-w$, we know that for all $\varphi\in C_0^\infty(\domain)$, with probability one,
\begin{equation}\label{exRes2-3}
\big(\bar u(t),\varphi\big)=
\int_0^t\big(a^{ij}(s)\bar u_{x^ix^j}(s)+f(s)+a^{ij}(s)w_{x^ix^j}(s),\varphi\big)ds,\quad t\in[0,T].
\end{equation}
It follows that $\bar u$ is the unique solution in $\frH^1_{2,d}(\domain,T)$ to \begin{equation}\label{exRes2-4}
d\bar{u}=(a^{ij}\bar{u}_{x^ix^j}+f+a^{ij}w_{x^ix^j})dt,\quad \bar u(0)=0
\end{equation}

We are going to consider \eqref{exRes2-4} $\omega$-wise and apply an $L_q(L_p)$-regularity result for deterministic PDEs from \cite{DonKim2011}.
To this end, we have to check that in the present situation our notion of a solution fits to the one described therein.
Since $\bar u$, $w\in\bH^1_{2,d-2}(\domain,T)$ and $f\in\bH^{0,q}_{p,d}(\domain,T)\hookrightarrow\bH^{0}_{2,d}(\domain,T)$, we know that, for almost every $\omega\in\Omega$, the mappings $t\mapsto\bar u(\omega,t,\cdot)$ and $t\mapsto w(\omega,t,\cdot)$ belong to $L_2([0,T];H^1_{2,d-2}(\domain))$ and $t\mapsto f(\omega,t,\cdot)$ belongs to $L_2([0,T];L_2(\domain))$.
In particular, by Lemma \ref{lem:collection}, $\bar u_{x^ix^j}(\omega)$ and $w_{x^ix^j}(\omega)$ belong to $L_2([0,T];H^{-1}_{2,d+2}(\domain))$ for all $i,j=1,\ldots,d$, so that \[\int_0^t\{a^{ij}(\omega,s)\bar u_{x^ix^j}(\omega,s)+f(\omega,s)+ a^{ij}(\omega,s)w_{x^ix^j}(\omega,s)\}ds=:\int_0^t\Phi(\omega,s)ds\]
exists as an $H^{-1}_{2,d+2}(\domain)$-valued Bochner integral. This and \eqref{exRes2-3} imply that, for almost all $\omega\in\Omega$,
\begin{equation*}
\big(\bar u(\omega,t),\varphi_k\big)=
\grklam{\int_0^t\Phi(\omega,s)ds
,\varphi_k},\; k\in\bN,\;t\in[0,T],
\end{equation*}
where $\{\varphi_k\}_{k\in\bN}\subseteq C_0^\infty(\domain)$ is supposed to be dense in $H^{1}_{2,d-2}(\domain)$. As a consequence,
\begin{equation*}
\bar u(\omega,t)=\int_0^t\Phi(\omega,s)ds
,\quad t\in[0,T].
\end{equation*}
Standard arguments lead to
\begin{equation*}
\int_0^T\phi'(t)\bar u(\omega,t)dt=-\int_0^T\phi(t)\Phi(\omega,t) dt+\phi(T)\bar u(\omega,T),\quad \phi\in C_0^\infty(\bR),
\end{equation*}
where the integrals are $H^{-1}_{2,d+2}(\domain)$-valued Bochner integrals.
We obtain
\begin{equation}\label{exRes2-8}
\begin{aligned}
\int_0^T\int_\domain&\bar u(\omega) \frac{\partial}{\partial t} h\,dx \,dt\\
&=\int_0^T\int_\domain (a^{ij}(\omega)\bar u_{x_j}(\omega) h_{x^i} - f(\omega)h
+a^{ij}(\omega)w_{x_j}(\omega) h_{x^i})\, dx\, dt+\int_\domain\bar u(\omega,T)h(T)dx
\end{aligned}
\end{equation}
for all test-functions $h=\phi\otimes\varphi$, $\phi\in C_0^\infty(\bR)$, $\varphi\in C_0^\infty(\domain)$.
Using approximation arguments one can verify that \eqref{exRes2-8} even holds for all test-functions $h$ which belong to the space $\cH^1_{2,2}((0,T)\times\domain)$ considered in \cite{DonKim2011} and which  vanish on $(0,T)\times\partial\domain$ in the sense that $h(t)\in \mathring W^1_2(\domain)=H^1_{2,d-2}(\domain)$ for almost all $t\in(0,T)$. 
Moreover, for almost all $\omega\in\Omega$, $\bar u(\omega)$ belongs to the space $\mathring\cH^1_{2,2}((0,T)\times\domain)$ as defined in \cite{DonKim2011} and it vanishes on $(0,T)\times\partial\domain$. Thus, for almost all $\omega\in\Omega$, $\bar u(\omega)$ is the unique solution in $\mathring\cH^1_{2,2}((0,T)\times\domain)$ to
\begin{equation*}
\left\{
\begin{aligned}
-\frac{\partial}{\partial t}v+a^{ij}(\omega)v_{x^ix^j}&=\text{div}\big(a^{1j}(\omega)w_{x^j}(\omega),\ldots,a^{dj}(\omega)w_{x^j}(\omega)\big)^T  -f(\omega)\;&\text{ in }(0,T)\times\domain\\
v&=0&\text{ on }(0,T)\times\partial \domain
\end{aligned}
\right.
\end{equation*}
in the sense of \cite{DonKim2011}. Now we can apply \cite[Theorem 8.1]{DonKim2011} and use the fact that the coefficients $a^{ij}$ are uniformly bounded due to Assumption \ref{ass:coeff} to obtain
\begin{equation}\label{eq:Dong-ap}
\|D\bar u(\omega)\|_{L_q([0,T];L_p(\domain))}\leq N \left( \|Dw(\omega)\|_{L_q([0,T];L_p(\domain))}+\|f(\omega)\|_{L_q([0,T];L_p(\domain))}\right)
\end{equation}
for almost all $\omega\in\Omega$, where the constant $N$ does not depend on $\omega$. 
We remark that the assumption on the Lipschitz constant $K_0$ comes into play at this point: Theorem~8.1 in \cite{DonKim2011} implies that there exists a constant $c=c(d,p,q,\delta_0,K)$ such that, if  $K_0\leq c$, then estimate \eqref{eq:Dong-ap} holds. 
Integration w.r.t. $\bP$ and Hardy's inequality yield
\begin{equation}\label{exRes2-10}
\begin{aligned}
\|u\|_{\bH^{0,q}_{p,d-p}(\domain,T)}
&\leq N\,
\|Du\|_{\bH^{0,q}_{p,d}(\domain,T)}\\
&\leq N\,
\grklam{\|D\bar u\|_{\bH^{0,q}_{p,d}(\domain,T)}
+
\|Dw\|_{\bH^{0,q}_{p,d}(\domain,T)}}\\
&\leq N\,
\grklam{
\|f\|_{\bH^{0,q}_{p,d}(\domain,T)}
+
\|Dw\|_{\bH^{0,q}_{p,d}(\domain,T)}}.
\end{aligned}
\end{equation}

The term $\|Dw\|_{\bH^{0,q}_{p,d}(\domain,T)}$ can be estimated with the help of an inequality for stochastic integrals in UMD Banach spaces with type $2$ taken from \cite{NeeVerWei2007}.  
To this end, let $\gamma\nrklam{\ell_2,H^1_{p,d-p}(\domain)}$ denote the Banach space of $\gamma$-radonifying operators from $\ell_2$ to $H^1_{p,d-p}(\domain)$, see \cite{Nee2010b} for a survey on this class of operators. Furthermore, let $\{\mathbf{e}_k\}_{k\in\bN}$ be the standard orthonormal basis of the Hilbert space $\ell_2$. Then, the stochastic process
\begin{align*}
b:\Omega\times\neklam{0,T}\to \gamma\nrklam{\ell_2,H^1_{p,d-p}(\domain)},
\end{align*}
given by
\begin{align*}
b(\omega,t)\mathbf{a}:=\sum_{k=1}^\infty \langle\mathbf{a},\mathbf{e}_k\rangle_{\ell_2} g^k(\omega,t),
\qquad
\mathbf{a}\in\ell_2,
\end{align*}
is well-defined, and
\begin{equation*}
b\in L_q(\Omega\times[0,T],\cP,\wP\otimes dt;\gamma(\ell_2,H^1_{p,d-p}(\domain))).
\end{equation*}
Moreover,
\begin{equation}\label{eq:eqGamma}
\nnrm{b}{L_q(\Omega\times[0,T];\gamma(\ell_2,H^1_{p,d-p}(\domain)))}
\sim
\nnrm{g}{\bH^{1,q}_{p,d-p}(\domain,T;\ell_2)},
\end{equation}
see \cite[Remark~3.7]{CioDahKin+2011}. 
We remark that $H^{1}_{p,d-p}(\domain)=\mathring{W}^1_p(\domain)$ is a UMD Banach spaces with type $2$, since it is a closed subspace of the classical Sobolev space $W^1_p(\domain)$.
Therefore, by \cite[Corollary~3.10]{NeeVerWei2007}, $(b(t))_{t\in[0,T]}$ is $L_q$-stochastically integrable w.r.t.\ the 
$\ell_2$-cylindrical Brownian motion $(W_{\ell_2}(t))_{t\in[0,T]}$ given by
\begin{equation*}
W_{\ell_2}(t)\mathbf{a}
:=
\sum_{k=1}^\infty \langle\mathbf{a},\mathbf{e}_k\rangle_{\ell_2} w^k_t,
\qquad \mathbf{a}\in\ell_2.
\end{equation*}
By \cite[Corollary~3.9]{NeeVerWei2007}, for every $t\in[0,T]$,
\begin{equation}\label{eq:stochInt}
\int_0^t b(s) \, dW_{\ell_2}(s)
=
\sum_{k=1}^\infty \int_0^t g^k(s) \, dw^k_s = w(t) \qquad
\text{a.s.},
\end{equation}
where the series converges in $L_q(\Omega; H^{1}_{p,d-p}(\domain))$. Thus, we can apply \cite[Corollary~3.10]{NeeVerWei2007} and obtain
\begin{equation*}
\nnrm{w}{\bH^{1,q}_{p,d-p}(\domain,T)}^q
\leq N \,
\|b\|_{L_q(\Omega\times[0,T];\gamma(\ell_2,H^1_{p,d-p}(\domain)))}^q
\leq N\,
\|g\|_{\bH^{1,g}_{p,d-p}(\domain,T;\ell_2)}^q,
\end{equation*}
where we used \eqref{eq:eqGamma} for the last estimate.
Now let $e_1,\ldots,e_d$ be the standard orthonormal basis of $\bR^d$. Note that for every $i\in\{1,\ldots,d\}$, $D^{e_i}$ is a linear and bounded operator from $H^1_{p,d-p}(\domain)$ to the UMD space $L_p(\domain)
$, cf.\ Lemma~\ref{lem:collection}(iii) and (iv). 
Using this and similar arguments as above we obtain
\[D^{e_i}w(t)=\int_0^t D^{e_i}\circ b(s)\, dW_{\ell_2}(s),\]
with $D^{e_i}\circ b$ denoting the (point-wise) composition of the operators $D^{e_i}$ and $b$, and 
\begin{equation}\label{exRes2-12}
\begin{aligned}
\|Dw\|_{\bH^{0,q}_{p,d}(\domain,T)}
&=
\sum_{i=1}^d \nnrm{D^{e_i}w}{\bH^{0,q}_{p,d}(\domain,T)}\\
&=
\sum_{i=1}^d 
\Big\|\int_0^{(\cdot)} D^{e_i}\circ b(s)\,dW_{\ell_2}(s)\Big\|_{L_q(\Omega\times[0,T];L_p(\domain))}\\
&\leq N \,
\sum_{i=1}^d \|D^{e_i} \circ b\|_{L_q(\Omega\times[0,T];\gamma(\ell_2;L_p(\domain)))}\\
&\leq N \,
\|b\|_{L_q(\Omega\times[0,T];\gamma(\ell_2;H^1_{p,d-p}(\domain)))}\\
&\leq N \,
\|g\|_{\bH^{1,q}_{p,d-p}(\domain,T;\ell_2)}.
\end{aligned}
\end{equation}
Combining \eqref{exRes2-10} and \eqref{exRes2-12} and applying Corollary \ref{cor:liftReg} finishes the proof if $q=mp$ with $m\in\bN$. We can get rid of this restriction by following the lines of \cite[Proof of Theorem~2.1, p.~7]{Kry2000} and applying Marcinkiewicz's interpolation theorem.
\end{proof}

\subsection{An $L_q(L_p)$-theory of the heat equation on general bounded Lipschitz domains}\label{subsec:LqLp_heat}

In this subsection we present a first $L_q(L_p)$-theory for the stochastic heat equation
\begin{equation}\label{eq:heat}
du = ( \Delta u + f)\,dt + g^k \,dw^k_t, \qquad u(0,\cdot)=0,
\end{equation}
on general bounded Lipschitz domains $\domain\subseteq\bR^d$. 
We start by presenting the main result of this subsection, which we will prove later on.

\begin{thm}\label{thm:heat_LqLp}
Let $\domain$ be a bounded Lipschitz domain in $\bR^d$ and let $\gamma\geq 0$.
There exists an exponent $p_0$ with $p_0>3$ when $d\geq 3$ and $p_0>4$ when $d=2$ such that
for $p \in [2,p_0)$ and $p\leq q<\infty$, Eq.~\eqref{eq:heat}
has a unique solution $u\in\frH^{\gamma+2,q}_{p,d}(\domain,T)$, provided
\begin{equation*}
f\in\bH^{\gamma,q}_{p,d+p}(\domain,T) \cap \bH^{0,q}_{p,d}(\domain,T)
\qquad\text{and}\qquad
g\in\bH^{\gamma+1,q}_{p,d}(\domain,T;\ell_2) \cap \bH^{1,q}_{p,d-p}(\domain,T;\ell_2).
\end{equation*}
Moreover, there exists a constant $N\in(0,\infty)$, which does not depend on $f$ and $g$, such that
\begin{equation}\label{eq:heat_ap}
\nnrm{u}{\bH^{\gamma+2,q}_{p,d-p}(\domain,T)}^q
\leq N \grklam{
\nnrm{f}{\bH^{0,q}_{p,d}(\domain,T)}^q
+
\nnrm{f}{\bH^{\gamma,q}_{p,d+p}(\domain,T)}^q
+
\nnrm{g}{\bH^{\gamma+1,q}_{p,d}(\domain,T;\ell_2)}^q
}.
\end{equation}
\end{thm}

For bounded $C^1$-domains $\gdomain\subseteq\bR^d$ this result has been already proven in \cite{Kim2009}. Unfortunately, the techniques used there will not work if the boundary is assumed to be just Lipschitz continuous. Therefore, we choose to take another way. We will mainly use the fact that the domain of the square root of the negative weak Dirichlet-Laplacian on $L_p(\domain)$ coincides with the closure of the test functions in the $L_p(\domain)$-Sobolev space of order one, at least for the range of $p$ allowed in our assertion. Before we prove this fact let us get more precise and introduce some notations and definitions.

Let $\domain$ be a bounded Lipschitz domain in $\bR^d$. As in \cite[Definition 3.1]{Woo2007}, for arbitrary $p\in(1,\infty)$, we define the weak Dirichlet-Laplacian $\laplacedwp$ on $L_p(\domain)$ as follows:
\begin{align*}
D(\laplacedwp) &:= \ggklam{u\in\mathring{W}^1_{p}(\domain) \, :\, \Delta u \in L_p(\domain)},	\\
\laplacedwp u &:= \Delta u = \delta_{ij} u_{x^ix^j},
\end{align*}
where $\delta_{ij}$ denotes the Kronecker symbol.
If we fix $p\in(p_0/(p_0-1),p_0)$ with $p_0=4+\delta$ when $d=2$ and $p_0=3+\delta$ when $d\geq 3$ where $\delta>0$ is taken from \cite[Proposition 4.1]{Woo2007}, then, the unbounded operator $\laplacedwp$ generates a strongly continuous, analytic semigroup $\ggklam{S_{p}(t)}_{t\geq0}$ of contractions on $L_p(\domain)$, see \cite[Theorem 3.8 and Corollary 4.2]{Woo2007}.
Thus, $(-\laplacedwp)^{1/2}$, the square root of the negative of $\laplacedwp$, can be defined as the inverse of the operator
\begin{equation}\label{eq:laplacedwpnegsqrtinv}
(-\laplacedwp)^{-1/2}:=\pi^{-1/2}\int_0^\infty t^{-1/2} S_p(t) \, dt : L_p(\domain) \to L_p(\domain)
\end{equation}
with domain
\begin{equation*}
D((-\laplacedwp)^{1/2}) := \text{Range}((-\laplacedwp)^{-1/2}),
\end{equation*}
see \cite[Chapter 2.6]{Paz1983}.
Endowed with the norm
\begin{equation*}
\nnrm{u}{D((-\laplacedwp)^{1/2})}:=\nnrm{(-\laplacedwp)^{1/2}u}{L_p(\domain)},
\quad
u\in D((-\laplacedwp)^{1/2}),
\end{equation*}
$D((-\laplacedwp)^{1/2})$ becomes a Banach space. Exploiting the fundamental results from \cite{Woo2007} and \cite{JerKen1995}, we can prove the following identity, which is crucial if we want to apply the results from \cite{NeeVerWei2012} in our setting. 

\begin{lemma}\label{thm:laplacedwpsqrt}
Let $\domain$ be a bounded Lipschitz domain in $\bR^d$. There is an exponent $p_0$ with $p_0>4$ when $d=2$ and $p_0>3$ when $d\geq3$ such that if $p\in[2,p_0)$
\begin{equation*}
D((-\laplacedwp)^{1/2}) = \mathring{W}^1_{p}(\domain)
\end{equation*}
with equivalent norms.
\end{lemma}
\begin{proof}
We fix $p\in[2,p_0)$ with $p_0=4+\delta$ when $d=2$ and $p_0=3+\delta$ when $d\geq 3$ where $\delta>0$ is taken from \cite[Proposition 4.1]{Woo2007}.
We start with the proof of the embedding
\begin{equation}\label{eq:embedd002}
D((-\laplacedwp)^{1/2}) \hookrightarrow \mathring{W}^1_{p}(\domain).
\end{equation}
By \cite[Proposition 4.1]{Woo2007} the semigroups $\ggklam{S_2(t)}_{t\geq0}$ and $\ggklam{S_p(t)}_{t\geq0}$ are consistent, i.e., for all $t\geq0$,
\begin{equation*}
S_2(t)f = S_p(t)f,\qquad f\in L_p(\domain).
\end{equation*}
Using \eqref{eq:laplacedwpnegsqrtinv}, this leads to
\begin{equation*}
(-\laplacedwp)^{-1/2}f = (-\laplacedw)^{-1/2}f,\qquad f\in L_p(\domain),
\end{equation*}
which implies implies the boundedness of the operator
\begin{equation*}
(-\laplacedwp)^{-1/2} : L_p(\domain) \to \mathring{W}^1_p(\domain),
\end{equation*}
since the restriction of $(-\laplacedw)^{-1/2}$ to $L_p(\domain)$ is bounded from $L_p(\domain)$ to $\mathring{W}^1_p(\domain)$, see \cite[Theorem~7.5(a)]{JerKen1995}. Thus
\begin{equation*}
\nnrm{(-\laplacedwp)^{-1/2}f}{\mathring{W}^1_p(\domain)}
\leq N\,
\nnrm{f}{L_p(\domain)},
\qquad f\in L_p(\domain).
\end{equation*}
Consequently, for any $g\in\text{Range}((-\laplacedwp)^{-1/2})=D((-\laplacedwp)^{1/2})$,
\begin{equation*}
\nnrm{g}{\mathring{W}^1_p(\domain)}
\leq N\,
\nnrm{(-\laplacedwp)^{1/2}g}{L_p(\domain)}
= N\,
\nnrm{g}{D((-\laplacedwp)^{1/2})},
\end{equation*}
with a constant $N$ independent of $g$. Embedding \eqref{eq:embedd002} follows.

It remains to prove the converse direction, i.e.,
\begin{equation*}
\mathring{W}^1_p(\domain)
\hookrightarrow
D((-\laplacedwp)^{1/2}).
\end{equation*}
To this end, we first notice that the strongly continuous analytic contraction-semigroup $\ggklam{S_p(t)}_{t\geq 0}$ on $L_p(\domain)$ is positive in the sense of \cite[p.\ 353]{EngNag2000}, see \cite[Lemma 4.4]{Woo2007}. Therefore, by \cite[Corollary~5.2]{KalWei2001}, $(-\laplacedwp)$ has a bounded $H^\infty$-calculus of angle less than $\pi/2$. Consequently, it has bounded imaginary powers. This implies
\begin{equation}\label{eq:domaincomplex}
\geklam{L_p(\domain),D(-\laplacedwp)}_{1/2} = D((-\laplacedwp)^{1/2}),
\end{equation}
see \cite[Theorem~1.15.3]{Tri1995}. By Remark~\ref{kuf} and \cite[Proposition~2.4]{Lot2000},
\begin{equation*}
\mathring{W}^1_p(\domain)
=
H^1_{p,d-p}(\domain)
=
\geklam{H^0_{p,d}(\domain),H^2_{p,d-2p}(\domain)}_{1/2}.
\end{equation*}
Also, by the definition of the weighted Sobolev spaces one easily sees that
\begin{equation*}
H^0_{p,d}(\domain) = L_p(\domain) \qquad\text{and}\qquad H^2_{p,d-2p}(\domain)\hookrightarrow D(-\laplacedwp).
\end{equation*}
Combining these results, we obtain
\begin{equation*}
\mathring{W}^1_p(\domain)
=
\geklam{H^0_{p,d}(\domain),H^2_{p,d-2p}(\domain)}_{1/2}
\hookrightarrow
\geklam{L_p(\domain),D(-\laplacedwp)}_{1/2}
=
D((-\laplacedwp)^{1/2}),
\end{equation*}
which finishes the proof.
\end{proof}

Now we are ready to prove our main result in this subsection.

\begin{proof}[Proof of Theorem~\ref{thm:heat_LqLp}]
As in the proof of Lemma~\ref{thm:laplacedwpsqrt} we fix $p\in[2,p_0)$ with $p_0=4+\delta$ when $d=2$ and $p_0=3+\delta$ when $d\geq 3$ where $\delta>0$ is taken from \cite[Proposition 4.1]{Woo2007}. We start with the case $\gamma=0$. I.e., we assume that $f\in\bH^{0,q}_{p,d}(\domain,T)$ and $g\in\bH^{1,q}_{p,d-p}(\domain,T;\ell_2)$. 
Recall the corresponding notations $(b(t))_{t\in[0,T]}$, $(w(t))_{t\in[0,T]}$ and $W_{\ell_2}$ introduced in the proof of Theorem~\ref{thm:LpLq} and their properties. 
By Lemma~\ref{thm:laplacedwpsqrt} and Eq.~\eqref{eq:domaincomplex} we have
\begin{equation*}
\geklam{L_p(\domain),D(-\laplacedwp)}_{1/2} = H^{1}_{p,d-p}(\domain).
\end{equation*}
Also, $D(\laplacedwp)\hookrightarrow L_p(\domain)$ densely, since $C^\infty_0(\domain)$ is contained in $D(\laplacedwp)$.
Therefore, we can apply \cite[Theorem~4.5(ii)]{NeeVerWei2012} and obtain the existence of a stochastic process
\begin{equation}\label{eq:sol_str}
u\in L_q(\Omega\times[0,T],\cP,\wP\otimes dt;D(-\laplacedwp))
\end{equation}
solving the stochastic evolution equation
\begin{equation*}
\left\{
\begin{aligned}
du(t) - \laplacedwp u(t) \, dt &= f(t)\, dt + b(t)\, dW_{\ell_2}(t),\qquad t\in[0,T]\\
u(0)&=0
\end{aligned}
\right.
\end{equation*}
in the sense of \cite[Definition~4.2]{NeeVerWei2012} with $X_0:=L_p(\domain)$.
Moreover, there exists a version $\tilde{u}$ of $u$, such that the following equality is fulfilled in $L_p(\domain)$ a.s.\ for all $t\in[0,T]$ at once:
\begin{equation*}
\tilde{u}(t)
=
\int_0^t \Delta \tilde{u}(s) \, ds
+
\int_0^t f(s) \, ds
+
\int_0^t b(s) \, dW_{\ell_2}(s).
\end{equation*}
We can fix a continuous versions of the stochastic process $(w(t))_{t\in[0,T]}$, so that by \eqref{eq:stochInt} a.s.\
\begin{equation*}
\tilde{u}(t)
=
\int_0^t \Delta \tilde{u}(s) \, ds
+
\int_0^t f(s) \, ds
+
w(t)
\qquad \text{for all } t\in[0,T] \text{ in } L_p(\domain).
\end{equation*}
Therefore, by Lemma~\ref{thm:laplacedwpsqrt} and \eqref{eq:sol_str}, $u\in\frH^{1,q}_{p,d}(\domain,T)$ and solves Eq.\ \eqref{eq:heat} in the sense of Definition~\ref{def:solution}. Since $\frH^{1,q}_{p,d}(\domain,T)\hookrightarrow \frH^{1,2}_{2,d}(\domain,T)$, the uniqueness follows from \cite[Theorem~2.12]{Kim2011}.
Thus, in order to finish the proof of the basic case $\gamma=0$, we show the a priori estimate
\begin{equation}\label{eq:heat_0}
\nnrm{u}{\bH^{1,q}_{p,d-p}(\domain,T)}^q
\leq N\,
\grklam{
\nnrm{f}{\bH^{0,q}_{p,d}(\domain,T)}^q
+
\nnrm{g}{\bH^{1,q}_{p,d}(\domain,T;\ell_2)}^q
},
\end{equation}
which implies estimate \eqref{eq:heat_ap} for $\gamma=0$.
To this end we will use the fact that the stochastic process $V:[0,T]\times\Omega\to L_p(\domain)$ defined as
\begin{equation*}
V(t)
:=
\int_0^t S_p(t-s) f(s) \, ds
+
\int_0^t S_p(t-s) b(s) \, dW_{\ell_2}(s),
\qquad t\in [0,T],
\end{equation*}
is a version of $u$, see \cite[Proposition~4.4]{NeeVerWei2012}. Since $-\laplacedwp$ has the (deterministic) maximal regularity property (see \cite[Proposition~6.1]{Woo2007}) and $0\in\rho(\laplacedwp)$, we obtain
\begin{equation}\label{eq:heat_f}
\E \reklam{
\sgnnrm{t\mapsto \int_0^t S_p(t-s) f(s) \, ds}{L_q(0,T;H^1_{p,d-p}(\domain))}^q}
\leq N\,
\nnrm{f}{\bH^{0,q}_{p,d}(\domain,T)}^q,
\end{equation}
where we used again Lemma~\ref{thm:laplacedwpsqrt} and \cite[Theorem~9.7]{Kuf1980}. Simultaneously, notice that  $-\laplacedwp$ and $g$ (respectively $b$) fulfil the assumptions of \cite[Theorem~1.1]{NeeVerWei2012b}; we have already checked them in our explanations above. Thus, applying this result, we obtain
\begin{equation}\label{eq:heat_g}
\E \reklam{
\sgnnrm{t\mapsto \int_0^t S_p(t-s) b(s) \, dW_{\ell_2}(s)}{L_q(0,T;H^1_{p,d-p}(\domain))}^q}
\leq N\,
\nnrm{g}{\bH^{0,q}_{p,d}(\domain,T;\ell_2)}^q.
\end{equation}
The constants in \eqref{eq:heat_f} and \eqref{eq:heat_g} do not depend on $f$ and $g$. Therefore, using the last two estimates we obtain the existence of a constant $N$, independent of $f$ or $g$, such that
\begin{equation*}
\nnrm{V}{\bH^{1,q}_{p,d-p}(\domain,T)}^q
\leq N\,
\sgrklam{
\nnrm{f}{\bH^{0,q}_{p,d}(\domain,T)}^q
+
\nnrm{g}{\bH^{0,q}_{p,d}(\domain,T;\ell_2)}^q
}.
\end{equation*}
Since $V$ is just a version of the solution $u$, Eq.~\eqref{eq:heat_0} follows.

In order to prove the assertion for arbitrary $\gamma\geq 0$ and $2\leq p \leq q <\infty$ we can argue as we have done at the end of the proof of Theorem~\ref{thm:LpLq}.
\end{proof}

\setcounter{equation}{0}
\section{H{\"o}lder-Sobolev regularity of elements of $\frH^{\gamma,q}_{p,\theta}(\domain,T)$ and implications for SPDEs}\label{StWS}

In this section we analyse the temporal H\"older regularity of functions in $\frH^{\gamma,q}_{p,\theta}(\domain,T)$, where $\domain$ is a bounded Lipschitz domain in $\bR^d$. 
Our main interest lies on the case $q\neq p$.
As an application, we obtain H\"older-Sobolev regularity for the solutions to SPDEs presented in Section~\ref{SPDEs}.
In combination with the Sobolev type embeddings for the spaces $H^\gamma_{p,\theta}(\domain)$ from the Section~\ref{WS}, we also obtain assertions concerning the H{\"o}lder regularity in time and space for elements of $\frH^{\gamma,q}_{p,\theta}(\domain,T)$ (Corollary \ref{cor:Schauder}). Here is the main result of this section. 

\begin{thm}\label{thm:mainHoelder}
Let  $\domain$ be a bounded Lipschitz domain in $\bR^d$. Let 
$2\leq p \leq q < \infty$,
$\gamma \in \bN$,
$\theta \in\bR$, 
and $u\in \frH^{\gamma,q}_{p,\theta}(\domain,T)$.
Moreover, let
$$
2/q <\bar{\beta}<\beta \leq 1.
$$
Then there exists a constant $N$, which does not depend on $T$ and $u$, such that
\begin{equation}\label{eq:mH01}
\begin{aligned}
\E[\psi^{\beta-1}&u]^q_{C^{\bar{\beta}/2-1/q}([0,T]; H^{\gamma-\beta}_{p,\theta}(\domain))}\\
&\phantom{nnnnnnn}\leq N T^{(\beta-\bar{\beta})q/2}
\sgrklam{\|u\|^q_{\bH^{\gamma,q}_{p,\theta-p}(\domain,T)}+\|\bD u\|^q_{\bH^{\gamma-2,q}_{p,\theta+p}(\domain,T)}+\|\bS u\|^q_{\bH^{\gamma-1,q}_{p,\theta}(\domain,T;\ell_2)}},
\end{aligned}
\end{equation}
and
\begin{equation}\label{eq:mH02}
\begin{aligned}
\E\nnrm{\psi^{\beta-1}u}{C^{\bar{\beta}/2-1/q}([0,T];H^{\gamma-\beta}_{p,\theta}(\domain))}^q
\leq  N & T^{(\beta-\bar{\beta})q/2}
\sgrklam{\E\nnrm{\psi^{\beta-1}u(0,\cdot)}{H^{\gamma-\beta}_{p,\theta}(\domain)}^q +\\ &\|u\|^q_{\bH^{\gamma,q}_{p,\theta-p}(\domain,T)}+\|\bD u\|^q_{\bH^{\gamma-2,q}_{p,\theta+p}(\domain,T)}+\|\bS u\|^q_{\bH^{\gamma-1,q}_{p,\theta}(\domain,T;\ell_2)}}.
\end{aligned}
\end{equation}
\end{thm}

Theorem \ref{thm:mainHoelder} and Lemma \ref{lem:collection}(ii) yield the following so-called interior Schauder estimates of functions in $\frH^{\gamma,q}_{p,\theta}(\domain,T)$.

\begin{corollary}\label{cor:Schauder}
Given the setting of Theorem \ref{thm:mainHoelder}, let $u\in \frH^{\gamma,q}_{p,\theta}(\domain,T)$
and
$
\gamma-\beta-d/p=k+\varepsilon
$
where $k\in\bN_0$ and $\varepsilon\in (0,1]$. Then for
$\nu:=\beta-1+\theta/p$ and multi-indices $i,j\in\bN_0^d$ such that
$|i|\leq k$ and $|j|=k$, we have
\begin{equation*}
\begin{aligned}
\E&\sgeklam{\sup_{t\neq s}|t-s|^{-(\bar{\beta} q/2-1)}
\sgrklam{|\psi^{\nu+|i|}D^i(u(t,\cdot)-u(s,\cdot))|^q_{C(\domain)}
+[\psi^{\nu+|j|+\varepsilon}
D^j(u(t,\cdot)-u(s,\cdot))]^q_{C^{\varepsilon}(\domain)}}} \\
&\leq N 
\sgrklam{\|u\|^q_{\bH^{\gamma,q}_{p,\theta-p}(\domain,T)}+\|\bD u\|^q_{\bH^{\gamma-2,q}_{p,\theta+p}(\domain,T)}+\|\bS u\|^q_{\bH^{\gamma-1,q}_{p,\theta}(\domain,T;\ell_2)}},
\end{aligned}
\end{equation*}
where the constant $N\in(0,\infty)$ does not depend on $u$.
In particular, if $u_0=0, \gamma\geq 1$, $\theta\leq d$ and
$r_0:=1-2/q-d/p>0$, then
 for any $r\in (0,r_0)$
\begin{align*}
&\E\reklam{ \sup_{t\leq T} \sup_{x,y\in
\domain}\frac{|u(t,x)-u(t,y)|^q}{|x-y|^{rq}}} <\infty, \\
&\E \reklam{\sup_{x\in \domain}\sup_{t,s\leq T}\frac{|u(t,x)-u(s,x)|^q}{|t-s|^{r q/2}}} <\infty,
\end{align*}
see \cite[Remark 4.8]{Kry2001} for details concerning the last implication.
\end{corollary}

Combining Theorem~\ref{thm:mainHoelder} with the results of Section~\ref{SPDEs}, we immediately obtain the following result on the H\"older-Sobolev regularity of solutions of SPDEs.

\begin{thm}\label{thm:Hoelder_SPDEs}
Let $\domain$ be a bounded Lipschitz domain in $\bR^d$. Let $2\leq p\leq q<\infty$, $\gamma\in\bN_0$, $\theta\in\bR$ and 
\begin{align*}
2/q<\bar{\beta}<\beta\leq 1.
\end{align*}
\noindent\textup{\textbf{(i)}} Given the setting of Theorem~\ref{thm:heat_LqLp}, the solution $u\in\frH^{\gamma+2,q}_{p,d}(\domain,T)$ 
of Eq.~\eqref{eq:heat} fulfils
\begin{equation*}
\begin{aligned}
\E\|\psi^{\beta-1}u\|^q_{C^{\bar{\beta}/2-1/q}([0,T]; H^{\gamma+2-\beta}_{p,d}(\domain))}
\leq N 
\sgrklam{
\|f\|^q_{\bH^{0,q}_{p,d}(\domain,T)}
+
\|f\|^q_{\bH^{\gamma,q}_{p,d+p}(\domain,T)}
+
\|g\|^q_{\bH^{\gamma+1,q}_{p,d}(\domain,T;\ell_2)}},
\end{aligned}
\end{equation*}
where the constant $N\in (0,\infty)$ does not depend on $u$, $f$ and $g$.

\noindent\textup{\textbf{(ii)}} Given the setting of Theorem~\ref{thm:LpLq}, the solution $u\in\frH^{\gamma+2,q}_{p,d}(\domain,T)$ of Eq.~\eqref{eq:mainEqa2}
fulfils
\begin{equation*}
\begin{aligned}
\E\|\psi^{\beta-1}u\|^q_{C^{\bar{\beta}/2-1/q}([0,T]; H^{\gamma+2-\beta}_{p,d}(\domain))}	 
\leq N \sgrklam{
\|f&\|_{\bH^{\gamma,q}_{p,d+p}(\domain,T)}^q
+
\|f\|_{\bH^{0,q}_{p,d}(\domain,T)}^q	 \\
&+
\|g\|_{\bH^{\gamma+1,q}_{p,d}(\domain,T;\ell_2)}^q
+
\|g\|_{\bH^{1,q}_{p,d-p}(\domain,T;\ell_2)}^q
},
\end{aligned}
\end{equation*}
where the constant $N\in (0,\infty)$ does not depend on $u$, $f$ and $g$.
\end{thm}

For the case that the summability  parameters in time and space coincide, i.e., $q=p$,  a result similar to Theorem \ref{thm:mainHoelder} has been proven in \cite{Kim2011}, see Theorem 2.9 therein. The proof in \cite{Kim2011} is straightforward and relies on \cite[Corollary 4.12]{Kry2001}, which is a variant of Theorem \ref{thm:mainHoelder} on the whole space $\bR^d$. However, we are explicitly interested in the case $q>p$ since it allows  a wider range of parameters $\bar{\beta}$ and $\beta$, and therefore leads to better regularity results. Unfortunately, the proof technique used in \cite[Proposition 2.9]{Kim2011} does not work any more in this case. Therefore, we take a different path: We use \cite[Proposition 4.1]{Kry2001}, which covers the assertion of Theorem \ref{thm:mainHoelder} with $\bR^d_+$  instead of $\domain$, and the Lipschitz character of $\partial\domain$ to derive Theorem \ref{thm:mainHoelder} via a boundary flattening argument. To this end, we need the following two lemmas whose proofs are postponed to the appendix.

\begin{lemma}\label{lem:Transf_rule}
Let $\gdomain^{(1)}, \gdomain^{(2)}$ be domains in $\bR^d$ with non-empty boundaries, and let $\phi: \gdomain^{(1)} \to \gdomain^{(2)}$ be a bijective map, such that  $\phi$ and $\phi^{-1}$ are Lipschitz continuous. Furthermore, assume that there exists a constant $N\in(0,\infty)$, such that 
\begin{equation*}
\frac{1}{N} \rho_{\gdomain^{(1)}}(\phi^{-1}(y))\leq \rho_{\gdomain^{(2)}}(y)\leq N \rho_{\gdomain^{(1)}}(\phi^{-1}(y)) \text{ for all } y\in\gdomain^{(2)},
\end{equation*}
and that the (a.e.\ existing) Jacobians $J\phi$ and $J\phi^{-1}$ fulfil
\begin{equation*}
|\operatorname{Det} J\phi |=1 \text{ and } |\operatorname{Det} J\phi^{-1} |=1 \quad\text{ (a.e.)}.
\end{equation*}
Then, for any $\gamma \in [-1,1]$, there exists a constant $N=N(d,\gamma,p,\theta,\phi)\in(0,\infty)$, which does not depend on $u$, such that
$$
\frac{1}{N}\|u\|_{H^{\gamma}_{p,\theta}(\gdomain^{(1)})}\leq \|u\circ\phi^{-1}\|_{H^{\gamma}_{p,\theta}(\gdomain^{(2)})}\leq N \|u\|_{H^{\gamma}_{p,\theta}(\gdomain^{(1)})}
$$
in the sense that, if one of the norms exists, so does the other one and the above inequality holds.
\end{lemma}

\begin{lemma}\label{pullbackEqn}
Let $\gdomain^{(1)}, \gdomain^{(2)}$ be bounded domains in $\bR^d$ and let 
$\phi: \gdomain^{(1)}\to\gdomain^{(2)}$ satisfy the assumptions of Lemma \ref{lem:Transf_rule}.
Furthermore, let $u\in\frH^{1,q}_{p,\theta}(\gdomain^{(1)},T)$ with $2\leq p\leq q <\infty$. 
Then $u\circ \phi^{-1} \in\frH^{1,q}_{p,\theta}(\gdomain^{(2)},T)$ with deterministic part $\,\bD(u\circ\phi^{-1}) = \bD u \circ \phi^{-1}$ and stochastic part $\,\bS (u\circ \phi^{-1})=\bS u \circ \phi^{-1}$. 
In particular, for any $\varphi\in
C^{\infty}_{0}(\gdomain^{(2)})$, with probability one, the equality
\begin{equation}
\big(u(t,\cdot)\circ\phi^{-1},\varphi\big)=
\big(u(0,\cdot)\circ\phi^{-1},\varphi\big) + \int^{t}_{0}\big((\bD u)(s,\cdot)\circ\phi^{-1},\varphi\big) \, ds
+\sum^{\infty}_{k=1} \int^{t}_{0}\big((\bS^ku)(s,\cdot)\circ\phi^{-1},\varphi\big)\, dw^k_s
\end{equation}
holds for all $t\in[0,T]$. 
\end{lemma}

Now we are able to prove our main result in this section.

\begin{proof}[Proof of Theorem \ref{thm:mainHoelder}]
Let us simplify notation and write $f:=\bD u$ and $g:=\bS u$ throughout the proof. We will show that \eqref{eq:mH01} is true by induction over $\gamma\in\bN$; estimate \eqref{eq:mH02} can be proved analogously. 

We start with the case $\gamma=1$. Fix $x_0\in\partial\domain$ and choose $r>0$ small enough, e.g., $r:=r_0(10 K_0)^{-1}$ with $r_0$ and $K_0>1$ from Definition \ref{domain}. Let us assume for a moment that the supports (in the sense of distributions) of $u$, $f$ and $g$ are contained in $B_r(x_0)$ for each $t$ and $\omega$. With $\mu_0$ from Definition~\ref{domain}, we introduce the function
\begin{align*}
\phi:\; \gdomain^{(1)}:=\domain \cap B_{r_0}(x_0) &\;\longrightarrow\; \gdomain^{(2)}:=\phi(\domain \cap B_{r_0}(x_0))\,\subseteq\,\bR^d_+\\
x=(x^1,x')&\;\longmapsto\;(x^1-\mu_0(x'),x'),
\end{align*}
which fulfils all the assumptions of Lemma \ref{lem:Transf_rule}.
Note that, since $r$ has been chosen sufficiently small, one has $\rho_\domain(x)=\rho_{\gdomain^{(1)}}(x)$ for all $x\in \domain\cap B_r(x_0)$, so that one can easily show that the equivalence
\begin{equation*}
\|v\|_{H^{\nu}_{\bar p,\bar\theta}(\domain)}
\sim 
\|v\|_{H^{\nu}_{\bar p,\bar\theta}(\gdomain^{(1)})},\qquad v\in\cD'(\domain),\;\supp v\subseteq B_r(x_0),
\end{equation*}
holds for all $\nu$, $\bar{\theta}\in\bR$ and $\bar{p}>1$.
Together with Lemma \ref{lem:Transf_rule} we obtain for any $\nu\in [-1,1]$,
\begin{equation*} 
\|v\|_{H^{\nu}_{\bar p,\bar\theta}(\domain)}
\sim 
\|v\circ\phi^{-1}\|_{H^{\nu}_{\bar p,\bar\theta}(\gdomain^{(2)})},
\qquad v\in\cD'(\domain),\;\supp v\subseteq B_r(x_0).
\end{equation*}
Thus, denoting $\bar{u}:=u\circ\phi^{-1}$, $\bar{f}:= f\circ\phi^{-1}$ and $\bar{g}:=g\circ\phi^{-1}$, by Lemma \ref{pullbackEqn} we know that on $G^{(2)}$ we have $d\bar{u}=\bar{f}dt+\bar{g}^kdw^k_t$  in the sense of distributions. Furthermore, since $\rho_{\gdomain^{(2)}}(y)=\rho_{\bR^d_+}(y)$ for all $y\in\phi( \domain \cap B_r(x_0))$, the equivalence
\begin{equation*}
\|v\circ\phi^{-1}\|_{H^{\nu}_{\bar p,\bar\theta}(\gdomain^{(2)})}
\sim 
\|v\circ\phi^{-1}\|_{H^{\nu}_{\bar p,\bar\theta}(\bR^d_+)},\qquad v\in\cD'(\domain),\;\supp v\subseteq B_r(x_0),
\end{equation*}
holds for any $\nu\in [-1,1]$,
where we identify $v\circ\phi^{-1}$ with its extension to $\bR^d_+$ by zero. Therefore, by making slight abuse of notation and writing $\bar{u}$, $\bar{f}$ and $\bar{g}$ for the extension by zero on $\bR^d_+$ of $\bar{u}$, $\bar{f}$ and $\bar{g}$ respectively, we have
\begin{equation*}
\bar u\in \bH^{1,q}_{p,\theta-p}(\bR^d_+,T),
\quad 
\bar u(0)\in U^{1,q}_{p,\theta}(\bR_+^d), 
\quad 
\bar{f}\in \bH^{-1,q}_{p,\theta+p}(\bR^d_+,T), 
\quad 
\bar{g}\in \bH^{0,q}_{p,\theta}(\bR^d_+,T;\ell_2),
\end{equation*}
and $d\bar{u}=\bar{f}dt+\bar{g}^kdw^k_t$ is fulfilled on $\bR^d_+$ in the sense of distributions. Thus, we can apply \cite[Theorem~4.1]{Kry2001} and use the equivalences above to get estimate \eqref{eq:mH01} in the following way:
\begin{align*}
\E [u]^q_{C^{\bar{\beta}/2-1/q}([0,T];H^{1-\beta}_{p,\theta+p(\beta-1)}(\domain))}
&
\leq N\,
\E [\bar u]^q_{C^{\bar{\beta}/2-1/q}([0,T];H^{1-\beta}_{p,\theta+p(\beta-1)}(\bR^d_+))}\\
&\leq N\, T^{(\beta-\bar{\beta})q/2}
\sgrklam{
\|\bar u\|^q_{\bH^{1,q}_{p,\theta-p}(\bR^d_+,T)}
+ 
\|\bar{f}\|^q_{\bH^{-1,q}_{p,\theta+p}(\bR^d_+,T)}
+
\|\bar{g}\|^q_{\bH^{0,q}_{p,\theta}(\bR^d_+,T;\ell_2)}
}\\
&\leq N\, T^{(\beta-\bar{\beta})q/2}
\sgrklam{
\|u\|^q_{\bH^{1,q}_{p,\theta-p}(\domain,T)}
+
\|f\|^q_{\bH^{-1,q}_{p,\theta+p}(\domain,T)}
+
\|g\|^q_{\bH^{0,q}_{p,\theta}(\domain,T;\ell_2)}
}.
\end{align*}
Now let us give up the assumption on the supports of $u$, $f$ and $g$. 
Let $\xi_0,\xi_1,\ldots, \xi_m$, be a partition of unity of $\domain$, such that $\xi_0\in C^{\infty}_0(\domain)$, and, for $i=1,\ldots,m$, $\xi_i\in  C^\infty_0(B_{r}(x_i))$ with $x_i \in \partial \domain$. Obviously, $d(\xi_i u)=\xi_i fdt+\xi_i g^k_t dw^k_t$ for $i=0,\ldots,m$.
Since
\begin{align*}
\E[\psi^{\beta-1}u]^q_{C^{\bar{\beta}/2-1/q}([0,T];H^{1-\beta}_{p,\theta}(\domain))}
&\leq N(m,q)
\sum_{i=0}^m \E [\psi^{\beta-1}(\xi_i u)]^q_{C^{\bar{\beta}/2-1/q}([0,T];H^{1-\beta}_{p,\theta}(\domain))},
\end{align*}
we just have to estimate 
$\E [\psi^{\beta-1} \xi_i u ]^q_{C^{\bar{\beta}/2-1/q}([0,T];H^{1-\beta}_{p,\theta}(\domain))}$ for each $i\in\{ 0,\ldots,m\}$. For $i\geq 1$ one gets the required estimate as before, using the fact that $C^\infty_0(\domain)$-functions are pointwise multipliers in all spaces $H^{\nu}_{\bar p,\bar\theta}(\domain)$, $\nu,\bar\theta\in\bR$, $\bar p>1$, see, e.g., \cite[Theorem 3.1]{Lot2000}.
The case $i=0$ can be treated as follows: Since $\xi_0$ has compact support in $\domain$,  for all $\nu,\bar\theta\in\bR$ and $\bar p>1$, we have
\begin{equation}\label{eq:xi0Equi}
\|v\xi_0\|_{H^{\nu}_{\bar p,\bar\theta}(\domain)} 
\sim 
\|v\xi_0\|_{H^{\nu}_{\bar p}(\bR^d)},\qquad v\in\cD'(\domain),
\end{equation}
and consequently 
\begin{equation*}
\E [\psi^{\beta-1}\xi_0 u]^q_{C^{\bar{\beta}/2-1/q}([0,T];H^{1-\beta}_{p,\theta}(\domain))}
\sim 
\E [\xi_0 u]^q_{C^{\bar{\beta}/2-1/q}([0,T]; H^{1-\beta}_{p}(\bR^d))}.
\end{equation*}
By \cite[Theorem 4.11]{Kry2001},  a further application of \eqref{eq:xi0Equi} and the fact that $C^\infty_0(\domain)$-functions are pointwise multipliers in all spaces $H^{\nu}_{\bar p,\bar\theta}(\domain)$, we obtain
\begin{align*}
\E [\xi_0 u&]^q_{C^{\bar{\beta}/2-1/q}([0,T];H^{1-\beta}_{p}(\bR^d))}\\
&\leq N T^{(\beta-\bar{\beta})q/2}
\sgrklam{
\|\xi_0 u\|^q_{\bH^{1,q}_p(\bR^d,T)}
+
\|\xi_0 f\|^q_{\bH^{-1,q}_p(\bR^d,T)}
+
\|\xi_0 g\|^q_{\bH^{0,q}_p(\bR^d,T;\ell_2)}
}\\
&\leq N T^{(\beta-\bar{\beta})q/2}
\sgrklam{
\|\psi^{-1} \xi_0 u\|^q_{\bH^{1,q}_{p,\theta}(\domain,T)}
+
\|\psi \xi_0 f\|^q_{\bH^{ -1,q}_{p,\theta}(\domain,T)}
+
\|\xi_0 g\|^q_{\bH^{0,q}_{p,\theta}(\domain,T;\ell_2)}
}\\
&\leq N T^{(\beta-\bar{\beta})q/2}
\sgrklam{
\|\psi^{-1} u\|^q_{\bH^{1,q}_{p,\theta}(\domain,T)}
+
\|\psi f\|^q_{\bH^{-1,q}_{p,\theta}(\domain,T)}
+
\|g\|^q_{\bH^{0,q}_{p,\theta}(\domain,T;\ell_2)}}.
\end{align*}
This finishes the proof of estimate \eqref{eq:mH01} for the case $\gamma=1$.

Next, let us move to the inductive step and assume that the assertion is true for some $\gamma=n\in\bN$. Fix $u\in\frH^{n+1,q}_{p,\theta}(\domain,T)$.
Then $v:=\psi u_x\in \frH^{n,q}_{p,\theta}(\domain,T)$ and $dv=\psi f_x dt +\psi g^k_x dw^k_t$ (component-wise). Also,  by Lemma \ref{lem:collection} (iii) and (iv),
\begin{align*}
\E [\psi^{\beta-1}u]^q_{C^{\bar{\beta}/2-1/q}([0,T];H^{n+1-\beta}_{p,\theta}(\domain))}
\leq  
N \sgrklam{ \E [\psi^{\beta-1}u&]^q_{C^{\bar{\beta}/2-1/q}([0,T];H^{n-\beta}_{p,\theta}(\domain))}
+	\\
 &\E [\psi^{\beta-1}v]^q_{C^{\bar{\beta}/2-1/q}([0,T];H^{n-\beta}_{p,\theta}(\domain))}}.
\end{align*}
Using the induction hypothesis and applying Lemma~\ref{lem:collection}(iii) and (iv) once more, we see that the induction goes through.
\end{proof}

\setcounter{equation}{0}
\section{Besov spaces and their relationship to weighted Sobolev spaces}
\label{Bes-WS}

We turn our attention to the scale of Besov spaces
\begin{equation}\tag{\ref{NAscale}}
B^{\alpha}_{\tau,\tau}(\domain),\quad\frac 1\tau=\frac\alpha d+\frac1 p,\quad\alpha>0,
\end{equation}
where $p\geq2$ is fixed and $\domain\subseteq\bR^d$ is a bounded Lipschitz domain. As pointed out in the introduction, our motivation for considering this scale is its close connection to nonlinear approximation theory. The main result of this section, Theorem \ref{thm:WSBes-ptau}, is a general embedding of the weighted Sobolev spaces
\begin{equation}\label{Hgammapd-nu}
H^\gamma_{p,d-\nu p}(\domain), \quad\gamma,\nu>0,
\end{equation}
into Besov spaces of the scale \eqref{NAscale}. 
In the previous sections we have seen that the stochastic parabolic weighted Sobolev spaces $\frH^{\gamma,q}_{p,\theta}(\domain,T)$ are suitable for the analysis of SPDEs of the type \eqref{eq:mainEqa}. 
The elements of these spaces are stochastic processes with values in the spaces $H^\gamma_{p,\theta-p}(\domain)=H^\gamma_{p,d-\nu p}(\domain)$ with $\nu=1+(d-\theta)/p$.
Thus, by proving the embedding mentioned above, we show that the regularity analysis for SPDEs of  type \eqref{eq:mainEqa} in terms of the scale \eqref{NAscale} can be traced back to the analysis of such equations in terms of the spaces $\frH^{\gamma,q}_{p,\theta}(\domain,T)$.

The outline of this section is as follows.
In Subsection \ref{Besov} we give a definition of Besov spaces and describe their characterization in terms of wavelets. Based on this characterization and some auxiliary results, we investigate the relation of the scales \eqref{Hgammapd-nu} and \eqref{NAscale} in Subsection \ref{WSBes}. Here we will also see that, for the relevant range of parameters $\gamma$ and $\nu$, the spaces $H^\gamma_{p,d-\nu p}(\domain)$ act like Besov spaces $B^{\gamma\wedge\nu}_{p,p}(\domain)$ with zero trace on the boundary (Remark \ref{rem:Dirichlet}).

\subsection{Besov spaces: Definition and wavelet decomposition.}\label{Besov}

Our standard reference concerning Besov spaces and wavelets is the monograph \cite{Coh2003}. Throughout this subsection, let $\gdomain\subseteq \bR^d$ be an arbitrary domain.

For a function $f:\gdomain \to \bR$ and a natural number $n\in\bN$ let
\begin{equation*}
\Delta_h^n f (x) :=
\prod_{i=0}^n\mathds{1}_\gdomain(x+ih ) \cdot \sum_{j=0}^n \genfrac(){0pt}{0}{n}{j}  (-1)^{n-j}\, f(x+jh)
\end{equation*}
be the $n$-th difference of $f$ with step $h\in\bR^d$.
For $p\in\left(0,\infty\right)$, the $n$-th order $L_p$-modulus of smoothness of $f$ is given by
\begin{equation*}
\omega^n (t,f,G)_p :=
\omega^n (t,f)_p := \sup_{|h|<t} \, \| \,
\Delta_h^n f \,  \|_{L_p (\gdomain)}\, , \qquad t>0 \, .
\end{equation*}
One definition of Besov spaces that fits in our purpose is the following:
\begin{defn}\label{def:Besov}
Let $s, p, q \in\left(0,\infty\right)$ and $n\in\bN$ with $n>s$.
Then $B^s_{p,q}(\gdomain)$ is the collection of all functions  $f \in L_p (\gdomain) $ such that
\begin{equation*}
| \, f \, |_{B^s_{p,q}(\gdomain)} := \sgrklam{ \int_0^\infty \Big[
t^{-s} \, \omega^n (t,f)_p\Big]^q \frac{dt}{t}}^{1/q} <\infty.
\end{equation*}
These classes are equipped with a  (quasi-)norm  by taking
\begin{equation*}
\| \, f \, \|_{B^s_{p,q}(\gdomain)}
:= \| \, f\, \|_{L_p(\gdomain)} +  | \, f \, |_{B^s_{p,q}(\gdomain)}\, .
\end{equation*}
\end{defn}

\begin{remark}
For a more general definition of Besov spaces, including the cases where $p,q=\infty$ and $s<0$ see, e.g., \cite{Tri2006}.
\end{remark}

We want to describe $B^s_{p,q}(\bR^d)$ by  means of wavelet expansions.  To this end let $\varphi$ be a scaling function of tensor product type on $\bR^d$ and let $\psi_i$, $i=1, \ldots, 2^d-1$, be corresponding multivariate  mother  wavelets such that, for a given $r\in\bN$ and some $\domainc>0$, the following locality, smoothness and vanishing moment conditions hold: for all $i=1, \ldots, 2^d-1$,
\begin{align}
&\supp \varphi,\,\supp\psi_i\subset [-\domainc,\domainc]^d,\label{wl1}\\
&\varphi,\,\psi_i \in C^r(\bR^d),\label{wl2}\\
&\int x^\alpha \, \psi_i (x)\, dx=0 \quad\text{ for all $\alpha\in\bN^d$ with $|\alpha|\le r$}.\label{wl3}
\end{align}
We assume that
\begin{align*}
 \big\{\varphi_k,\psi_{i,j,k}\,:\, (\ijk)\in\{1,\cdots,2^d-1\}\times\bN_0\times\bZ^d\big\}
\end{align*}
is a Riesz basis of $L_2(\bR^d)$,  where we used  the abbreviations for dyadic shifts and dilations of the scaling function and the corresponding wavelets
 \begin{align}
\varphi_k(x)    &:=\varphi(x - k),\;x\in\bR^d ,&&\text{for $k\in\bZ^d$, and} \label{wl4}\\
\psi_{i,j,k}(x) &:=2^{jd/2}\psi_i(2^jx-k),\;x\in\bR^d, &&\text{for $(\ijk)\in\{1,\cdots,2^d-1\}\times\bN\times\bZ^d$.}\label{wl5}
\end{align}
Further, we assume that there exists a dual Riesz basis satisfying the same requirements. More precisely, there exist functions
$\widetilde{\varphi} $ and
$\widetilde{\psi}_i$, $ i=1, \ldots , 2^d-1$,
 such that conditions \eqref{wl1}, \eqref{wl2} and \eqref{wl3} hold if $\varphi$ and $\psi$ are replaced by $\widetilde{\varphi} $ and
$\widetilde{\psi}_i$, and such that the biorthogonality relations
\[
\langle \widetilde{\varphi}_k, \psi_{i,j,k} \rangle = \langle \widetilde{\psi}_{i,j,k},
 \varphi_k \rangle  = 0\, , \quad
\langle \widetilde{\varphi}_k, \varphi_{l} \rangle  = \delta_{k,l}, \quad
\langle \widetilde{\psi}_{i,j,k}, \psi_{u,v,l} \rangle  = \delta_{i,u}\, \delta_{j,v}\, \delta_{k,l}\, ,
\]
are fulfilled.
Here we use  analogous  abbreviations to \eqref{wl4} and \eqref{wl5} for the dyadic shifts and dilations of $\widetilde{\varphi} $ and
$\widetilde{\psi}_i$ ,  and $\delta_{k,l}$ denotes the Kronecker symbol. We refer to \cite[Chapter 2]{Coh2003} for the construction of biorthogonal wavelet bases, see also \cite{Dau1992} and \cite{CohDauFea1992}. To  keep notation simple,  we will write
\begin{equation*}
 \psi_{i,j,k,p} := 2^{jd(1/p-1/2)}\psi_\ijk \qquad \text{ and } \qquad \widetilde{\psi}_{i,j,k,p'}:= 2^{jd(1/{p'}-1/2)}\widetilde{\psi}_\ijk,
\end{equation*}
for the $L_p$-normalized wavelets and the correspondingly modified duals, with $p':=p/(p-1)$ if $p\in(0,\infty),\;p\neq 1,$ and  $p':=\infty,\;1/p':=0$ if $p=1$.

The following theorem shows how Besov spaces can be described by decay properties of the wavelet coefficients, if the parameters fulfil certain conditions.

\begin{thm}\label{thm:BesovChar01}
Let $p,q\in\left(0,\infty\right)$ and $s>\max\left\{0,d\left(1/p-1\right)\right\}$. Choose $r\in\bN$ such that $r>s$ and construct a  biorthogonal  wavelet Riesz basis as described above.
Then a locally integrable function $f:\bR^d\to\bR$ is in the Besov space $B^s_{p,q}(\bR^d)$ if, and only if,
\begin{equation}\label{BesovZerlegung}
f=\sum_{k\in\bZ^d}\langle f,\widetilde{\varphi}_k\rangle\,\varphi_k
    + \sum_{i=1}^{2^{d}-1}\sum_{j\in\bN}\sum_{k\in\bZ^d}\langle f,\widetilde{\psi}_{\ijk,p'}\rangle\,\psi_{\ijk,p}
\end{equation}
(convergence in $\mathcal D'(\bR^d)$) with
\begin{equation}\label{BesovNormDiscrete}
\Big(\sum_{k\in\bZ^d}|\langle f,\widetilde{\varphi}_k\rangle|^p\Big)^{1/p}
    + \Big(
    \sum_{i=1}^{2^{d}-1}\sum_{j\in\bN_0}2^{j s q}
    \Big(\sum_{k\in\bZ^d}|\langle f,\widetilde{\psi}_{\ijk,p'}\rangle|^p\Big)^{q/p}\Big)^{1/q}
    <   \infty,
\end{equation}
and \eqref{BesovNormDiscrete}  is an equivalent (quasi-)norm  for $B^s_{p,q}(\bR^d)$.
\end{thm}

\begin{remark}
A proof of this theorem for the case $p\geq1$ can be found in  \cite[\textsection10 of Chapter 6]{Mey1992}. For the general case see for example \cite{Kyr1996} or \cite[Theorem 3.7.7]{Coh2003}. Of course, if \eqref{BesovNormDiscrete} holds then the infinite sum in \eqref{BesovZerlegung} converges also in $B^s_{p,q}(\bR^d)$.  If $s>\max\left\{0,d\left(1/p-1\right)\right\}$ we have  the embedding $B_{p,q}^s(\bR^d)\subset L_{\bar{s}}(\bR^d)$ for some $\bar{s}>1$, see, e.g.\ \cite[Corollary 3.7.1]{Coh2003}.
\end{remark}

A simple computation gives us the following characterization of Besov spaces from the scale \eqref{NAscale} on $\bR^d$.

\begin{corollary}\label{cor:BesovChar02}
Let $p\in\left(1,\infty\right)$, $\alpha>0$ and $\tau\in\bR$ such that  $1/\tau=\alpha/d+1/p$. Choose $r\in\bN$ such that $r>\alpha$ and construct a  biorthogonal  wavelet Riesz basis as described above.
Then a locally integrable function $f:\bR^d\to\bR$ is in the Besov space $B^\alpha_{\tau,\tau}(\bR^d)$  if and only if
\begin{equation}\label{BesovZerlegung02}
f=\sum_{k\in\bZ^d}\langle f,\widetilde{\varphi}_k\rangle\,\varphi_k
    + \sum_{i=1}^{2^{d}-1}\sum_{j\in\bN_0}\sum_{k\in\bZ^d}\langle f,\widetilde{\psi}_{\ijk,p'}\rangle\,\psi_{\ijk,p}
\end{equation}
(convergence in $\mathcal D'(\bR^d)$) with
\begin{equation}\label{BesovNormDiscrete02}
\Big(\sum_{k\in\bZ^d}|\langle f,\widetilde{\varphi}_k\rangle|^\tau\Big)^{1/\tau}
    + \Big(
    \sum_{i=1}^{2^{d}-1}\sum_{j\in\bN_0}
    \sum_{k\in\bZ^d}|\langle f,\widetilde{\psi}_{\ijk,p'}\rangle|^\tau\Big)^{1/\tau}
    <   \infty  \, ,
\end{equation}
and \eqref{BesovNormDiscrete02}  is an equivalent (quasi-)norm  for $B^\alpha_{\tau,\tau}(\bR^d)$.
\end{corollary}

\subsection{From weighted Sobolev spaces to Besov spaces}
\label{WSBes}

In this subsection we will prove two embeddings of weighted Sobolev spaces into Besov spaces. We first focus on the case where the integrability parameter $p\in[2,\infty)$ of the weighted Sobolev spaces and the Besov spaces under consideration coincide, see Lemma \ref{lem:WSBes-pp}. This  will pave the way for proving a general embedding of weighted Sobolev spaces into the Besov spaces from the scale \eqref{NAscale} in Theorem \ref{thm:WSBes-ptau}. Remember that in this article $\domain\subseteq\bR^d$ always denotes a bounded Lipschitz domain.

\begin{lemma}\label{lem:WSBes-pp}
Let $\domain$ be a bounded Lipschitz domain in $\bR^d$. Let $\gamma,\nu\in\rklam{0,\infty}$ and $p\in\left[2,\infty\right)$. Then the following embedding holds:
\begin{equation}\label{eq:WSBes-pp}
H^\gamma_{p,d-\nu p}(\domain) \hookrightarrow B^{\gamma\land\nu}_{p,p}(\domain).
\end{equation}
\end{lemma}

\begin{proof}
We start the proof by considering the case where $\gamma=\nu$, i.e., we prove that for $\gamma>0$ and $p\geq2$ we have
\begin{equation}\label{eq:WSBes-pp01}
H^\gamma_{p,d-\gamma p}(\domain) \hookrightarrow B^{\gamma}_{p,p}(\domain).
\end{equation}
It is well-known, see \cite[Theorem 9.7]{Kuf1980}, that for $k\in\bN_0$,
\begin{equation}\label{KufN}
H^k_{p,d-k p}(\domain) = \mathring{W}^k_{p}(\domain),
\end{equation}
where $\mathring{W}^k_{p}(\domain)$ denotes the completion of  $C_0^\infty(\domain)$ in the classical $L_p(\domain)$-Sobolev space $W^k_p(\domain)$. We use the convention $W^0_p(\domain):=L_p(\domain)$. Since $p\geq 2$ we have $W^k_p(\domain)\hookrightarrow B^k_{p,p}(\domain)$, see, e.g., \cite[Remark~2.3.3/4 and Theorem~4.6.1(b)]{Tri1995} together with \cite{Dis2003}. We therefore obtain \eqref{eq:WSBes-pp01} for $k\in\bN_0$. In the case of fractional $\gamma\in(0,\infty)\backslash \bN$ we argue as follows. Let $\gamma=k+\eta$ with $k\in\bN_0$ and $\eta\in\rklam{0,1}$. By \cite[Proposition~2.4]{Lot2000},
\begin{align*}
H^{k+\eta}_{p,d-(k+\eta)p}(\domain)
=
\geklam{H^k_{p,d-kp}(\domain), H^{k+1}_{p,d-(k+1)p}(\domain)}_\eta.
\end{align*}
Using \eqref{KufN} we have
\begin{align*}
H^{k+\eta}_{p,d-(k+\eta)p}(\domain)
=
\geklam{\mathring{W}^k_{p}(\domain), \mathring{W}^{k+1}_{p}(\domain)}_{\eta}
\hookrightarrow
\geklam{W^k_p(\domain), W^{k+1}_p(\domain)}_\eta.
\end{align*}
For any $k\in\bN$, it is well-known that the Sobolev space $W^k_p(\domain)$ coincides with the Triebel-Lizorkin spaces $F^k_{p,2}(\domain)$, see, e.g., \cite[Theorem 1.122]{Tri2006}. Moreover, we have $L_p(\domain)\hookrightarrow F^0_{p,2}(\domain)$, see  \cite[(1.2) together with Definition~1.95]{Tri2006}. Thus,
\begin{align*}
H^{k+\eta}_{p,d-(k+\eta)p}(\domain)
\hookrightarrow
\geklam{F^k_{p,2}(\domain), F^{k+1}_{p,2}(\domain)}_\eta.
\end{align*}
The fact that Triebel-Lizorkin spaces constitute a scale of complex interpolation spaces, see, e.g., \cite[Corollary 1.111]{Tri2006}, leads to
\begin{align*}
H^{k+\eta}_{p,d-(k+\eta)p}(\domain)
\hookrightarrow
F^{k+\eta}_{p,2}(\domain).
\end{align*}
For $p\geq 2$, it is well known that $F^s_{p,2}(\domain)\hookrightarrow B^s_{p,p}(\domain)$ for any $s\in\bR$, see, e.g.,  (1.299) in \cite{Tri2006} together with  \cite{Dis2003}. Therefore,
\begin{align*}
H^{k+\eta}_{p,d-(k+\eta)p}(\domain)
\hookrightarrow
B^{k+\eta}_{p,p}(\domain),
\end{align*}
and \eqref{eq:WSBes-pp01} is proved for general $\gamma>0$.
The embedding \eqref{eq:WSBes-pp} for $\gamma\neq\nu$ follows now by using standard arguments. Indeed, since $\gamma\geq\gamma\land\nu$ we have
\begin{align*}
H^\gamma_{p,d-\nu p}(\domain) \hookrightarrow H^{\gamma\land\nu}_{p,d-\nu p}(\domain),
\end{align*}
see \cite{Lot2000}, the line after Definition 2.1. Furthermore, $d-\nu p\leq d-(\gamma\land\nu)p$ implies
\begin{equation*}
H^{\gamma\land\nu}_{p,d-\nu p}(\domain) 
\hookrightarrow 
H^{\gamma\land\nu}_{p,d-(\gamma\land\nu) p}(\domain),
\end{equation*}
see \cite[Corollary 4.2]{Lot2000}. A combination of these two embeddings with \eqref{eq:WSBes-pp01} finally gives  \eqref{eq:WSBes-pp}.
\end{proof}

\begin{remark}\label{rem:Dirichlet}
Since $\domain\subseteq \bR^d$ is assumed to be a bounded Lipschitz domain, we know by \cite[Chapter~VIII, Theorem~2]{JonWal1984} that for $1/p<s$ the operator $\mathrm{Tr}$, initially defined to $C^\infty(\overline{\domain})$ as the restriction on $\partial \domain$, extends to a bounded linear operator from $B^s_{p,p}(\domain)$ to $B^{s-1/p}_{p,p}(\partial\domain)$. In this case we denote by $B^s_{p,p,0}(\domain)$ the subspace of $B^s_{p,p}(\domain)$ with zero boundary trace, i.e., 
\begin{align*}
B^s_{p,p,0}(\domain) := \ggklam{u\in B^s_{p,p}(\domain) \,:\, \mathrm{Tr}\,u = 0},
\quad \frac{1}{p}<s<1+\frac{1}{p}.
\end{align*}
By \cite[Theorem 3.12]{JerKen1995} these spaces coincide with the closure of $C_0^\infty(\domain)$ in $B^s_{p,p}(\domain)$, i.e., 
\begin{align*}
\mathring{B}^s_{p,p}(\domain) := \overline{C_{0}^\infty(\domain)}^{\nnrm{\cdot}{B^s_{p,p}(\domain)}}= B^s_{p,p,0}(\domain) \quad \text{ for } \quad\frac{1}{p}<s<1+\frac{1}{p}.
\end{align*}
Thus, if $1/p<\gamma\land\nu<1+1/p$, Lemma \ref{lem:WSBes-pp} together with Lemma  \ref{lem:collection}$(i)$, lead to
\begin{align*}
H^\gamma_{p,d-\nu p}(\domain)\hookrightarrow \mathring{B}^{\gamma\land\nu}_{p,p}(\domain)= B^{\gamma\land\nu}_{p,p,0}(\domain)
= \ggklam{u\in B^{\gamma\land\nu}_{p,p}(\domain)\,:\, \mathrm{Tr}\,u=0}.
\end{align*}
In Section \ref{SPDEs} we considered SPDEs in the setting
of \cite{Kim2011}. The solutions to
these equations are stochastic processes taking values in
$H^\gamma_{p,d-\nu p}(\domain)$ with $\nu:=1+\frac{d-\theta}{p}$,
where the value of $\theta$ never leaves the range
$d-1<\theta<d+p-1$, compare also \cite{KryLot1999b}. This condition is equivalent to $1/p<\nu<1+1/p$
with $\nu$ as introduced before. Hence, if $\gamma>1/p$ we deal with solutions with zero boundary condition, in the
sense that the well defined linear and continuous boundary trace
$\mathrm{Tr}$ equals zero.
\end{remark}

In the second part of this subsection we investigate the relationship between weighted Sobolev spaces and the Besov spaces from the scale \eqref{NAscale}.
In \cite{DahDeV1997}, the scale \eqref{NAscale} is used to analyse the regularity of harmonic functions on a bounded Lipschitz domain $\domain\subseteq\bR^d$. Denoting by $\Theta(\domain)$ the set of harmonic functions on $\domain$ we can formulate the main result therein, \cite[Theorem 3.2]{DahDeV1997}, as follows:
\begin{align*}
\Theta(\domain) \cap B^\nu_{p,p}(\domain) \hookrightarrow B^\alpha_{\tau,\tau}(\domain), \quad \frac{1}{\tau}=\frac{\alpha}{d}+\frac{1}{p}, \quad \text{for all } 0<\alpha<\nu \frac{d}{d-1}=\sup_{m\in\bN}\min\sggklam{{m,\nu \frac{d}{d-1}}}.
\end{align*}
One of the main ingredients for the proof of this statement is the fact that harmonic functions contained in $B^\nu_{p,p}(\domain)$ have finite weighted Sobolev half-norm
\begin{align*}
|\, u\,|_{H^m_{p,d-\nu p}(\domain)}:= \sgrklam{\sum_{\substack{\alpha\in\bN_0^d \\  |\alpha | = m}}\int_{\domain} \big| \rho(x)^{|\alpha|} D^\alpha u(x)\big|^p \rho(x)^{-\nu p} dx}^{1/p}
\end{align*}
for any $m\in\bN$, see \cite[Theorem 3.1]{DahDeV1997} for details. 
It turns out that arguing along the lines of \cite[Theorem 3.2]{DahDeV1997} one can even show that for $\nu>0$ and $m\in\bN$,
\begin{align*}
H^m_{p,d-\nu p}(\domain)\cap B^\nu_{p,p}(\domain)\hookrightarrow B^\alpha_{\tau,\tau}(\domain), \quad \frac{1}{\tau}=\frac{\alpha}{d}+\frac{1}{p},\quad\!\! \text{for all} \quad\!\! 0 < \alpha < \min\sggklam{m ,\nu\frac{d}{d-1}}.
\end{align*}
Combining this with Lemma \ref{lem:WSBes-pp}, we obtain
\begin{equation}\label{eq:WSBes-ptau-bN}
H^m_{p,d-\nu p}(\domain)\hookrightarrow B^\alpha_{\tau,\tau}(\domain), \quad \frac{1}{\tau}=\frac{\alpha}{d}+\frac{1}{p},\quad\!\! \text{for all} \quad\!\! 0 < \alpha < \min\sggklam{m ,\nu\frac{d}{d-1}}.
\end{equation}

In what follows we give a detailed proof of the extension of \eqref{eq:WSBes-ptau-bN} to arbitrary smoothness parameters $\gamma>0$ instead of $m\in\bN$.
To this end, let us fix some notations. We will use a wavelet Riesz-basis
\begin{equation*}
\ggklam{\varphi_k,\;\psi_\ijk\,:\,(\ijk)\in\{1,\cdots,2^d-1\}\times\bN_0\times\bZ^d}
\end{equation*}
of $L_2(\bR^d)$ which satisfies the assumptions   from Section \ref{Besov} with $\domain\subseteq [-\domainc,\domainc]^d$ and  $r$ large enough---we will always clarify what we mean by that in the particular theorems.
Given $(j,k)\in\bN_0\times\bZ^d$, let
\begin{equation*}
Q_\jk:=2^{-j}k+2^{-j}\,[-\domainc,\domainc]^{d},
\end{equation*}
so that $ \supp\psi_\ijk\subset Q_\jk $ for all $i\in\{1,\ldots,2^d-1\}$ and $\supp\varphi_k\subset Q_{0,k}$ for all $k\in\bZ^d$. Remember that the supports of the corresponding dual basis   fulfil   the same requirements.
For our purpose the set of all indices associated with those wavelets that may have common support with the domain $\domain$ will play an important role and we denote  it  by
\begin{equation*}
 \Lambda:=\big\{(\ijk)\in\{1,\ldots,2^d-1\}\times\bN_0\times\bZ^d\,\big|\,Q_{\jk}\cap\domain\neq\emptyset\big\}.
\end{equation*}
Furthermore, we want to distinguish the indices corresponding to wavelets with support in the interior of the domain from the ones corresponding to wavelets which might have support on the boundary of $\domain$. To this end we write
\begin{align*}
\rho_\jk&:=\text{dist}(Q_{j,k},\partial\domain)=\inf_{x\in Q_\jk}\rho(x),\\
\Lambda_j &:= \big\{ (i,l,k)\in\Lambda\,:\,l=j \big\} ,\\
\Lambda_{j, m}&:=\left\{(\ijk)\in\Lambda_j\,:\,\; m2^{-j}\leq\rho_\jk<( m+1)2^{-j}\right\}   ,\\
\Lambda_j^0&:=\Lambda_j\setminus\Lambda_{j,0},\\
\Lambda^0&:=\bigcup_{j\in\bN_0}\Lambda_j^0,
\end{align*}
where $j, m\in\bN_0$ and $k\in\bZ^d$.
Later we will also use the notation
\begin{equation*}
 \Gamma:=\{k\in\bZ^d:Q_{0,k}\cap\domain\neq\emptyset\}.
\end{equation*}

The following lemma paves the way for proving \eqref{eq:WSBes-ptau-bN} for arbitrary $\gamma>0$ instead of just $\gamma=m\in\bN$. It establishes an estimate for, roughly speaking, a discretization of a weighted Sobolev norm in terms of the supports of the wavelets in the interior of $\domain$ at a fixed scaling level $j\in\bN_0$. For better readability, we place the quite technical proof in the appendix. Remember that we write $A^\circ$ for the interior of an arbitrary subset $A$ of $\bR^d$.

\begin{lemma}\label{lem:fancy}
Let $\domain$ be a bounded Lipschitz domain in $\bR^d$. 
Let $p\in[2,\infty)$, $\gamma\in(0,\infty)$ and $\nu\in\bR$ with $\gamma\geq \nu$. 
Furthermore, assume $u\in H^\gamma_{p,d-\nu p}(\domain)$.
Then, for all $j\in\bN_0$, the inequality
\begin{equation*}
\sum_{(\ijk)\in\Lambda^0_j}\big(\rho_{\jk}^{\gamma-\nu}|u|_{B_{p,p}^\gamma(Q_\jk^\circ)}\big)^p\leq N\,\nnrm{u}{H^\gamma_{p,d-\nu p}(\domain)}^p
\end{equation*}
holds, with a constant $N\in(0,\infty)$ which does not depend on $j$ and $u$.
\end{lemma}

Now we can prove the main result of this section. We use the convention \textquoteleft$1/0:=\infty$\textquoteright.

\begin{thm}\label{thm:WSBes-ptau}
Let $\domain$ be a bounded Lipschitz domain in $\bR^d$. Let $p\in\left[2,\infty\right)$, and $\gamma,\nu\in(0,\infty)$. Then
\begin{equation*}
H^\gamma_{p,d-\nu p}(\domain) \hookrightarrow B^\alpha_{\tau,\tau}(\domain), \quad \frac{1}{\tau}=\frac{\alpha}{d}+\frac{1}{p},\quad\!\! \text{for all} \quad\!\! 0 < \alpha < \min\sggklam{\gamma ,\nu\frac{d}{d-1}}.
\end{equation*}
\end{thm}

\begin{proof}
Let us start with the case $\nu>\gamma$. Then, for any $0<\alpha<\gamma$, we have
\begin{equation*}
H^\gamma_{p,d-\nu p}(\domain) \hookrightarrow B^\gamma_{p,p}(\domain)\hookrightarrow  B^\alpha_{\tau,\tau}(\domain),\quad \frac{1}{\tau}=\frac{\alpha}{d}+\frac{1}{p},
\end{equation*}
where we used Lemma \ref{lem:WSBes-pp} and standard embeddings for Besov spaces. Therefore, in this case the assertion of the theorem follows immediately. From now on, let us assume that $0<\nu\leq\gamma$.
We fix $\alpha$ and $\tau$ as stated in the theorem and choose the wavelet Riesz-basis of $L_2(\bR)$  from above with $r>\gamma$.
We also fix $u\in H^\gamma_{p,d-\nu p}(\domain)$. Due to Lemma \ref{lem:WSBes-pp} we have $u\in B^\nu_{p,p}(\domain)$. As $\mathcal O$ is a Lipschitz domain there exists a linear and bounded extension operator $\mathcal E:B^\nu_{p,p}(\domain)\to B^\nu_{p,p}(\bR^d)$, i.e., there exists a constant $N>0$ such that:
\begin{equation*}
\mathcal E u\big|_{\domain} = u
\qquad \text{and} \qquad
\|\mathcal E u\|_{B^\nu_{p,p}(\bR^d)} \leq N  \|u\|_{B^\nu_{p,p}(\domain)},
\end{equation*}
see, e.g., \cite{Ryc1999}.
In the   sequel   we will  omit  the $\mathcal E$ in our notation and write $u$ instead of $\mathcal E u$.
Theorem \ref{thm:BesovChar01} tells us that the following equality holds on the domain $\domain$:
\begin{align*}
u=\sum_{k\in\Gamma}\langle u,\widetilde\varphi_k\rangle\varphi_k + \sum_{(\ijk)\in\Lambda}\langle u,\widetilde\psi_{\ijk,p'}\rangle\psi_{\ijk,p},
\end{align*}
where the sums converge unconditionally in $B^\nu_{p,p}(\bR^d)$. Furthermore,  cf.\  Corollary \ref{cor:BesovChar02}, we have
\begin{equation*}
\|u\|_{B^\alpha_{\tau,\tau}(\domain)}^\tau
\leq N \Big(\sum_{k\in\Gamma}|\langle u,\widetilde\varphi_k\rangle|^\tau+ \sum_{(\ijk)\in\Lambda}|\langle u,\widetilde\psi_{\ijk,p'}\rangle|^\tau\Big),
\end{equation*}
see also \cite{Dis2003}.
Hence, by Lemma \ref{lem:WSBes-pp}, it is enough to prove  that
\begin{equation}\label{eq:WSBes-ptau-B}
\sum_{k\in\Gamma}|\langle u,\widetilde\varphi_k\rangle|^\tau
    \,\leq\, N\, \|u\|_{B^\nu_{p,p}(\domain))}^\tau
\end{equation}
and
\begin{equation}\label{eq:WSBes-ptau-C}
\sum_{(\ijk)\in\Lambda}|\langle u,\widetilde\psi_{i,j,k,p'}\rangle|^\tau
\leq N\left(\|u\|_{H^\gamma_{p,d-\nu p}(\domain)}+\|u\|_{B^\nu_{p,p}(\domain)}
\right)^\tau.
\end{equation}

We start with \eqref{eq:WSBes-ptau-B}. The index set $\Gamma$ introduced above is finite because of the boundedness of $\domain$, so that we can use Jensen's inequality to obtain
\begin{align*}
\sum_{k\in\Gamma} |\langle u,\widetilde{\varphi}_k\rangle|^\tau
\leq N \bigg(\Big( \sum_{k\in\Gamma} |\langle u,\widetilde{\varphi}_k\rangle|^p\Big)^{1/p} \bigg)^\tau
\leq N\,\|u\|_{B^\nu_{p,p}(\domain)}^\tau.
\end{align*}
In the last step we have used Theorem \ref{thm:BesovChar01} and the boundedness of the extension operator.

Now let us focus on  inequality \eqref{eq:WSBes-ptau-C}. To this end, we use the notations from above and split the expression on the left hand side of \eqref{eq:WSBes-ptau-C} into
\begin{align}\label{eq:WSBes-ptau-split}
\sum_{(\ijk)\in\Lambda^0}
\big| \langle u,\widetilde\psi_{i,j,k,p'} \rangle \big|^\tau
+
\sum_{(\ijk)\in\Lambda\setminus\Lambda^0} \big| \langle u,\widetilde\psi_{i,j,k,p'} \rangle \big|^\tau
=: I + I\!\!I
\end{align}
and estimate each term separately.

Let us begin with $I$. Fix $(\ijk)\in\Lambda^0$. As a consequence of Lemma \ref{lem:fancy}, we know that $u\big|_{Q_\jk^\circ} \in B^\gamma_{p,p}(Q_\jk^\circ)$.
By a Whitney-type inequality,   also   known as the Deny-Lions lemma, see, e.g., \cite[Theorem 3.5]{DeVSha1984},  there exists a polynomial $P_{j,k}$ of total degree less than $\gamma$, and a constant $N$, which does not depend on $j$ or $k$, such that
\begin{equation*}
\| u-P_{j,k}\|_{L_p(Q_\jk)}\leq N \, 2^{-j\gamma}|u|_{B^\gamma_{p,p}(Q_\jk^\circ)}.\end{equation*}
Since $\widetilde\psi_{\ijk,p'}$ is orthogonal to every polynomial of total degree less than $\gamma$, we have
\begin{align*}
\big|\langle u,\widetilde\psi_{\ijk,p'}\rangle\big|
&=\big|\langle u-P_{\jk},\widetilde\psi_{\ijk,p'}\rangle\big|\\
&\leq\|u-P_{\jk}\|_{L_p(Q_\jk)} \, \|\widetilde\psi_{\ijk,p'}\|_{L_{p'}(Q_\jk)}\\
&\leq N \, 2^{-j\gamma}\,\big|u\big|_{B^\gamma_{p,p}(Q_\jk^\circ)}\\
&\leq N \, 2^{-j\gamma}\, \rho_{\jk}^{\nu -\gamma}\,\rho_{\jk}^{\gamma-\nu}\,\big|u\big|_{B^\gamma_{p,p}(Q_\jk^\circ)}.
\end{align*}
Fix  $j\in\bN_0$. Summing over all indices $(\ijk)\in\Lambda_j^0$ and applying H{\"o}lder's inequality with exponents  $p/\tau>1$ and $p/(p-\tau)$ one finds
\begin{align*}
\sum_{(\ijk)\in\Lambda_j^0}
\big| \langle u,\widetilde\psi_{\ijk,p'} \rangle \big|^\tau
&\leq N
\sum_{(\ijk)\in\Lambda_j^0} 2^{-j\gamma\tau} \rho_\jk^{(\nu-\gamma)\tau}
\rho_\jk^{(\gamma-\nu)\tau} \big|u\big|^\tau_{B^\gamma_{p,p}(Q_\jk^\circ)} \nonumber\\
&\leq N
\sgrklam{\sum_{(\ijk)\in\Lambda_j^0} \sgrklam{\rho_{\jk}^{\gamma-\nu}\,\big|u\big|_{B^\gamma_{p,p}(Q_\jk^\circ)}}^p}^{\frac{\tau}{p}}
\sgrklam{
\sum_{(\ijk)\in\Lambda_j^0} 2^{-j\frac{\gamma \tau p}{p-\tau}}
\rho_\jk^{\frac{(\nu-\gamma)\tau p}{p-\tau}}
}^{\frac{p-\tau}{p}}.
\end{align*}
Now we use Lemma \ref{lem:fancy} to obtain
\begin{align}\label{eq:WSBes-ptau-inside}
\sum_{(\ijk)\in\Lambda_j^0}
\big| \langle u,\widetilde\psi_{\ijk,p'} \rangle \big|^\tau
&\leq N\,
\nnrm{u}{H^\gamma_{p,d-\nu p}(\domain)}^\tau
\sgrklam{
\sum_{(\ijk)\in\Lambda_j^0} 2^{-j\frac{\gamma \tau p}{p-\tau}}
\rho_\jk^{\frac{(\nu-\gamma)\tau p}{p-\tau}}
}^{\frac{p-\tau}{p}},
\end{align}
with a constant $N$, which does not depend on the level $j$. In order to estimate the sum on the right hand side we use the Lipschitz character of the domain $\domain$ which implies that
\begin{align}\label{eq:LipFin}
 |\Lambda_{j, m}|\leq N\, 2^{j(d-1)}\qquad\text{ for all $j, m\in\bN_0$.}
\end{align}
Moreover, the  boundedness  of $\domain$ yields  $\Lambda_{j, m}=\emptyset$ for all $j, m\in\bN_0$ with $ m\geq N 2^j$, where the constant $N$ does not depend on $j$ or $m$. Consequently,
\begin{equation}\label{eq:uglySums}
\begin{aligned}
\bigg(
\sum_{(\ijk)\in\Lambda_j^0}
2^{-j\frac{\gamma p\tau}{p-\tau}}
\rho_{j,k}^{\frac{(\nu -\gamma)p\tau}{p-\tau}}
\bigg)^{\frac{p-\tau}{p}}
&\leq
\bigg(
\sum_{m=1}^{N 2^j}
\sum_{(\ijk)\in\Lambda_{j, m}}
2^{-j\frac{\gamma p\tau}{p-\tau}}
\rho_{j,k}^{\frac{(\nu-\gamma)p\tau}{p-\tau}}
\bigg)^{\frac{p-\tau}p} \\
&\leq
N \bigg(
\sum_{ m=1}^{N 2^j} 2^{j(d-1)} \, 2^{-j\frac{\gamma p\tau}{p-\tau}} ( m \, 2^{-j})^{\frac{(\nu -\gamma)p\tau}{p-\tau}}
\bigg)^{\frac{p-\tau}p} \\
&\leq
N \bigg(
2^{j\left(d-1-\frac{\nu p \tau}{p-\tau}\right)}+ 2^{j\left(d-\frac{\gamma p \tau}{p-\tau}\right)} \bigg)^{\frac{p-\tau}p}.
\end{aligned}
\end{equation}
Now, let us sum over all $j\in\bN_0$. Inequalities \eqref{eq:uglySums} and \eqref{eq:WSBes-ptau-inside} imply
\begin{align*}
\sum_{(\ijk)\in\Lambda^0}
\big| \langle u,\widetilde\psi_{i,j,k,p'} \rangle \big|^\tau
\leq N
\sum_{j\in\bN_0}
\bigg(
2^{j\left(d-1-\frac{\nu p \tau}{p-\tau}\right)} + 2^{j\left(d-\frac{\gamma p \tau}{p-\tau}\right)} \bigg)^\frac{p-\tau}{p}
\nnrm{u}{H^\gamma_{p,d-\nu p}(\domain)}^\tau.
\end{align*}
Obviously, the sums on the right hand side converge  if, and only if, $\alpha\in\left(0,\gamma\wedge \nu \frac{d}{d-1}\right)$.   Finally,
\begin{align*}
\sum_{(\ijk)\in\Lambda^0}
\big| \langle u,\widetilde\psi_{i,j,k,p'} \rangle \big|^\tau
\leq N\, \nnrm{u}{H^\gamma_{p,d-\nu p}(\domain)}^\tau.
\end{align*}

Now we estimate the second term $I\!\!I$ in \eqref{eq:WSBes-ptau-split}. First we fix $j\in\bN_0$ and use H{\"o}lder's inequality and
\eqref{eq:LipFin} to obtain
\begin{align*}
\sum_{(\ijk)\in\Lambda_{j,0}}
\big| \langle u,\widetilde{\psi}_{\ijk,p'} \rangle \big|^\tau
\leq N
\,2^{j(d-1)\frac{p-\tau}{p}}
\Big( \sum_{(\ijk)\in\Lambda_{j,0}}
\big| \langle u,\widetilde{\psi}_{\ijk,p'} \rangle
\big|^p
\Big)^\frac{\tau}{p}.
\end{align*}
Summing over all $j\in\bN_0$ and using H{\"o}lder's inequality again yields
\begin{align*}
\sum_{(\ijk)\in\Lambda\backslash\Lambda^0}
& \big| \langle u,\widetilde{\psi}_{\ijk,p'} \rangle \big|^\tau
= \sum_{j\in\bN_0}
\Big[
\sum_{(\ijk)\in\Lambda_{j,0}} \big| \langle u,\widetilde{\psi}_{\ijk,p'} \rangle \big|^\tau
\Big] \\
&\leq N
\sum_{j\in\bN_0}
\Big[ 2^{j(d-1)\frac{p-\tau}{p}}
\Big( \sum_{(\ijk)\in\Lambda_{j,0}} \big| \langle u,\widetilde{\psi}_{\ijk,p'} \rangle \big|^p \Big)^\frac{\tau}{p}
\Big] \\
&\leq N\,
\bigg(
\sum_{j\in\bN_0}
2^{j\left( \frac{(d-1)(p-\tau)}{p}-\nu \tau \right)\frac{p}{p-\tau}}
\bigg)^{\frac{p-\tau}{p}}
\bigg(
\sum_{j\in\bN_0}
\sum_{(\ijk)\in\Lambda_{j,0}}
2^{j \nu p}\big| \langle u,\widetilde{\psi}_{\ijk,p'} \rangle \big|^p \bigg)^\frac{\tau}{p}.
\end{align*}
Using Theorem \ref{thm:BesovChar01} and the  boundedness  of the extension operator, we obtain
\begin{align*}
\sum_{(\ijk)\in\Lambda\backslash\Lambda^0}
\big| \langle u,\widetilde{\psi}_{\ijk,p'} \rangle \big|^\tau
\leq N\,
\| u \|_{B^\nu_{p,p}(\domain)}^{\tau}
\bigg(
\sum_{j\in\bN_0}
2^{j\left( \frac{(d-1)(p-\tau)}{p}-\nu\tau \right)\frac{p}{p-\tau}}
\bigg)^{\frac{p-\tau}{p}}.
\end{align*}
The series on the right hand side converges if, and only if, $\alpha\in\left( 0,\nu\frac{d}{d-1} \right)$. We thus have
\begin{equation*}
\sum_{(\ijk)\in\Lambda\backslash\Lambda^0}
\big| \langle u,\widetilde{\psi}_{\ijk,p'} \rangle \big|^\tau
\leq N\, 
\| u \|_{B^\nu_{p,p}(\domain)}^{\tau}
\leq N\,
\nnrm{u}{H^\gamma_{p,d-\nu p}(\domain)}^\tau. 
\qedhere\end{equation*}
\end{proof}

\setcounter{equation}{0}
\section{H{\"o}lder--Besov regularity for elements of $\frH^{\gamma,q}_{p,\theta}(\domain,T)$ and implications for SPDEs}\label{RegHB}

In this section, we state and prove our second main result concerning the time-space regularity of the solutions to SPDEs of the form \eqref{eq:mainEq} on bounded Lipschitz domains. We use the scale \eqref{NAscale} to measure the regularity in space, whereas the time-regularity will be measured in terms of H{\"o}lder norms. Since the stochastic parabolic weighted Sobolev spaces $\frH^{\gamma,q}_{p,\theta}(\domain,T)$ are the right spaces to construct a solvability theory for SPDEs, we will first formulate our results in terms of these spaces. As a consequence, each result about existence of solutions to SPDEs in these spaces 
automatically encodes a statement about the H{\"o}lder-Besov regularity of the solution. 
The corresponding results for the solutions in the different settings from Section~\ref{SPDEs} will be presented here in detail.
We will use the following short notation
\begin{align*}
L_q(\Omega_T;B^\alpha_{\tau,\tau}(\domain))
:=
L_q(\Omega\times[0,T],\mathcal{P},\mathbb{P}\otimes dt; B^\alpha_{\tau,\tau}(\domain))
\end{align*}
Let us first clarify  for which range of $\alpha>0$ a stochastic process $u\in\mathfrak{H}^{\gamma,q}_{p,\theta}(\domain,T)$ takes values in $B^\alpha_{\tau,\tau}(\domain)$, $1/\tau=\alpha/d+1/p$.

\begin{thm}\label{thm:BesSpatial}
Let $\domain$ be a bounded Lipschitz domain in $\bR^d$. Let $\gamma+2\in(0,\infty)$, $p,q\in[2,\infty)$, $\theta\in\bR$, and  $u\in\bH^{\gamma+2,q}_{p,\theta-p}(\domain,T)$. Then,
\begin{equation}\label{eq:alpha}
u \in L_q(\Omega_T; B^\alpha_{\tau,\tau}(\domain)),\quad \frac{1}{\tau}=\frac{\alpha}{d}+\frac{1}{p},
\,\,\text{ for all }\,\,
0<\alpha<\min\sggklam{\gamma+2, \sgrklam{1+\frac{d-\theta}{p}}\frac{d}{d-1}}.
\end{equation}
Moreover, for $\alpha$ fulfilling \eqref{eq:alpha}, there exists a constant $N$ which does not depend on $u$, such that
\begin{equation*}
\E\sgeklam{\int_0^T\nnrm{u(t,\cdot)}{B^\alpha_{\tau,\tau}(\domain)}^q \,dt}
\leq N\,
\nnrm{u}{\bH^{\gamma+2,q}_{p,\theta-p}(\domain,T)}^q
.
\end{equation*}
\end{thm}

\begin{proof}
This is a direct consequence of Theorem \ref{thm:WSBes-ptau}.
\end{proof}

Combining this assertion with the results from Section~\ref{SPDEs} we obtain the following spatial regularity results for SPDEs.

\begin{thm}\label{thm:Spatial_Besov_SPDEs}
Let $\domain$ be a bounded Lipschitz domain in $\bR^d$. 

\noindent\textup{\textbf{(i)}}
Given the setting of Theorem~\ref{thm:LpLp}, the solution $u\in\frH^{\gamma+2,p}_{p,\theta}(\domain,T)$ of Eq.\ \eqref{eq:mainEqa} fulfils \eqref{eq:alpha}. Moreover, for any $\alpha$ in \eqref{eq:alpha},
there exists a constant $N$, which does not depend on $u$, $f$, $g$ and $u_0$, such that
\begin{equation*}
\E\sgeklam{\int_0^T\nnrm{u(t,\cdot)}{B^\alpha_{\tau,\tau}(\domain)}^p \,dt}
\leq N\,
\sgrklam{
\nnrm{f}{\bH^{\gamma,p}_{p,\theta+p}(\domain,T)}^p
+
\nnrm{g}{\bH^{\gamma+1,p}_{\theta,p}(\domain,T;\ell_2)}^p
+
\nnrm{u_0}{U^{\gamma,p}_{p,\theta}(\domain)}^p
}.
\end{equation*}

\noindent\textup{\textbf{(ii)}} Given the setting of Theorem~\ref{thm:heat_LqLp}, the solution $u\in \frH^{\gamma+2,q}_{p,d}(\domain,T)$ of Eq.~\eqref{eq:heat} fulfils \eqref{eq:alpha} with $\theta=d$. Moreover, for any $\alpha$ in \eqref{eq:alpha},
there exists a constant $N$, which does not depend on $u$, $f$, and $g$, such that
\begin{equation*}
\E\sgeklam{\int_0^T\nnrm{u(t,\cdot)}{B^\alpha_{\tau,\tau}(\domain)}^q \,dt}
\leq N\,
\sgrklam{
\nnrm{f}{\bH^{0,q}_{p,d}(\domain,T)}^q
+
\nnrm{f}{\bH^{\gamma,q}_{p,d+p}(\domain,T)}^q
+
\nnrm{g}{\bH^{\gamma+1,p}_{d,p}(\domain,T;\ell_2)}^q
}.
\end{equation*}

\noindent\textup{\textbf{(iii)}} Given the setting of Theorem~\ref{thm:LpLq}, the solution $u\in \frH^{\gamma+2,q}_{p,d}(\domain,T)$ of Eq.~\eqref{eq:mainEqa2} fulfils \eqref{eq:alpha} with $\theta=d$. Moreover, for any $\alpha$ in \eqref{eq:alpha}, 
there exists a constant $N$, which does not depend on $u$, $f$, and $g$, such that
\begin{equation*}
\begin{aligned}
\E\sgeklam{\int_0^T\nnrm{u(t,\cdot)}{B^\alpha_{\tau,\tau}(\domain)}^q \,dt}
\leq N\, \sgrklam{
\|f&\|_{\bH^{\gamma,q}_{p,d+p}(\domain,T)}^q
+
\|f\|_{\bH^{0,q}_{p,d}(\domain,T)}^q	 \\
&+
\|g\|_{\bH^{\gamma+1,q}_{p,d}(\domain,T;\ell_2)}^q
+
\|g\|_{\bH^{1,q}_{p,d-p}(\domain,T;\ell_2)}^q
}.
\end{aligned}
\end{equation*}
\end{thm}

\begin{proof}
The assertions are immediate consequences of Theorem~\ref{thm:BesSpatial} and the corresponding existence results from Section~\ref{SPDEs}.
\end{proof}

\begin{remark}\label{rem:rel_LpLp}
A result similar to Theorem~\ref{thm:Spatial_Besov_SPDEs}(i) has been proved in \cite[Theorem 3.1, see also Theorem~B.3]{CioDahKin+2011}. 
There are three major improvements in Theorem~\ref{thm:Spatial_Besov_SPDEs}(i) compared to \cite[Theorem~3.1]{CioDahKin+2011}. 
Firstly, we have no restriction on $\gamma\in (0,\infty)$, whereas in \cite{CioDahKin+2011} only integer $\gamma\in\bN_0$ are considered. Secondly, we obtain $L_p$-integrability in time of the $B^\alpha_{\tau,\tau}(\domain)$-valued process for arbitrary $p\in[2,\infty)$. With the techniques used in \cite{CioDahKin+2011} just $L_\tau$-integrability in time can be established.  Thirdly, we do not need the extra assumption $u\in L_p([0,T]\times\Omega; B^s_{p,p}(\domain))$ for some $s>0$. It suffices that $u\in\bH^{\gamma+2,p}_{p,\theta-p}(\domain,T)$.
Note that, obviously, this improvements also hold for the solutions of more general equations as considered in \cite[Theorem~B.3]{CioDahKin+2011}.
\end{remark}

Here is the main result of this section. It concerns the H\"older-Besov regularity of processes in $\frH^{\gamma,q}_{p,\theta}(\domain,T)$.

\begin{thm}\label{thm:HoeBes}
Let $\domain$ be a bounded Lipschitz domain in $\bR^d$. Let $2\leq p\leq q <\infty$, $\gamma+2\in\bN$, $\theta\in\bR$, and $u\in\frH^{\gamma+2,q}_{p,\theta}(\domain,T)$. Moreover, let
\begin{align*}
\frac{2}{q}< \bar{\beta} <\min\sggklam{1,1+\frac{d-\theta}{p}}.
\end{align*}
Then, for all $\alpha$ and $\tau$ with
\begin{equation}\label{eq:alphatau}
\frac{1}{\tau}=\frac{\alpha}{d}+\frac{1}{p} \quad\text{and}\quad 0<\alpha<\min\sggklam{\gamma+2-\bar{\beta}, \sgrklam{1+\frac{d-\theta}{p}-\bar{\beta}}\frac{d}{d-1}},
\end{equation}
we have
\begin{align*}
\E\geklam{u}^q_{\mathcal{C}^{\bar{\beta}/2-1/q}([0,T];B^\alpha_{\tau,\tau}(\domain))}
\leq
N(T)
\sgrklam{\|u\|^q_{\bH^{\gamma+2,q}_{p,\theta-p}(\cO,T)}+\|\bD u\|^q_{\bH^{\gamma,q}_{p,\theta+p}(\cO,T)}+\|\bS u\|^q_{\bH^{\gamma+1,q}_{p,\theta}(\cO,T;\ell_2)}},
\end{align*}
and
\begin{align*}
\E\|u\|^q_{\mathcal{C}^{\bar{\beta}/2-1/q}([0,T];B^\alpha_{\tau,\tau}(\domain))}
\leq N(T)\,
\nnrm{u}{\frH^{\gamma+2,q}_{p,\theta}(\domain,T)}^q.
\end{align*}
The constants $N(T)$ are given by $N(T)=N \sup_{\beta\in[\bar{\beta},1]}\ggklam{T^{(\beta-\bar{\beta})q/2}}$, with $N$ from \eqref{eq:mH01} and \eqref{eq:mH02} respectively.
\end{thm}
\begin{proof}
The assertion is an immediate consequence of Theorem~\ref{thm:mainHoelder}
and Theorem~\ref{thm:WSBes-ptau}.
\end{proof} 

We obtain the following implications concerning the path regularity of the solutions of SPDEs from Section~\ref{SPDEs}.

\begin{thm}\label{thm:HoelderBesov_SPDEs}
Let $\domain$ be a bounded Lipschitz domain in $\bR^d$.

\noindent\textup{\textbf{(i)}} Given the setting of Theorem~\ref{thm:heat_LqLp} with $\gamma\in\bN_0$, let $u\in\frH^{\gamma+2,q}_{p,d}(\domain,T)$ be the solution of Eq.~\eqref{eq:heat}. Assume furthermore that
$
2/q < \bar{\beta}<1
$, and that $\alpha$ and $\tau$ fulfil \eqref{eq:alphatau}. Then,
\begin{align*}
\E\|u\|^q_{\mathcal{C}^{\bar{\beta}/2-1/q}([0,T];B^\alpha_{\tau,\tau}(\domain))}
\leq
N\,
\sgrklam{
\|f\|^q_{\bH^{\gamma,q}_{p,d+p}(\cO,T)}
+
\|f\|^q_{\bH^{0,q}_{p,d}(\cO,T)}
+
\|g\|^q_{\bH^{\gamma+1,q}_{p,d}(\cO,T;\ell_2)}},
\end{align*}
where the constant $N$ does not depend on $u$, $f$ and $g$.

\noindent\textup{\textbf{(ii)}} Given the setting of Theorem~\ref{thm:LpLq} with $\gamma\in\bN_0$, let 
$u\in\frH^{\gamma+2,q}_{p,d}(\domain,T)$ be the solution of Eq.~\eqref{eq:mainEqa2}. Assume furthermore that
$
2/q < \bar{\beta}<1
$, and that $\alpha$ and $\tau$ fulfil \eqref{eq:alphatau}. Then,
\begin{equation*}
\begin{aligned}
\E\|u\|^q_{\mathcal{C}^{\bar{\beta}/2-1/q}([0,T];B^\alpha_{\tau,\tau}(\domain))}
\leq N\, \sgrklam{
\|f&\|_{\bH^{\gamma,q}_{p,d+p}(\domain,T)}^q
+
\|f\|_{\bH^{0,q}_{p,d}(\domain,T)}^q	 \\
&+
\|g\|_{\bH^{\gamma+1,q}_{p,d}(\domain,T;\ell_2)}^q
+
\|g\|_{\bH^{1,q}_{p,d-p}(\domain,T;\ell_2)}^q
},
\end{aligned}
\end{equation*}
where the constant $N$ does not depend on $u$, $f$ and $g$.
\end{thm}

\begin{proof}
The assertions are immediate consequences of Theorem~\ref{thm:HoeBes} and the corresponding existence results from Section~\ref{SPDEs}.
\end{proof}

\setcounter{equation}{0}
\begin{appendix}

\section{Appendix}

\begin{proof}[Proof of Lemma \ref{lem:Transf_rule}]
We consider consecutively the cases $\gamma=0,1,-1$. For fractional $\gamma\in(-1,1)$, the statement follows then by using interpolation arguments, see \cite[Proposition 2.4]{Lot2000}.  Furthermore, we resume ourselves to the proof of the right inequality in the assertion of the Lemma, i.e., that there exists a constant $N=N(d,\gamma,p,\theta,\phi)\in (0,\infty)$, such that for any $u\in H^{\gamma}_{p,\theta}(\gdomain^{(1)})$ the following inequality holds:
$$\nnrm{u\circ\phi^{-1}}{H^\gamma_{p,\theta}(\gdomain^{(2)})}\leq N\nnrm{u}{H^\gamma_{p,\theta}(\gdomain^{(1)})}.$$
The left inequality can be proven analogously.
For $\gamma=0$, the assertion follows immediately by using the assumptions of the Lemma and the change of variables formula for bi-Lipschitz transformations, see, e.g., \cite[Theorem 3]{Haj1993}. Let us go on and look at the case $\gamma=1$. Because of the density of the test functions $C^\infty_0(\gdomain^{(1)})$ in $H^{1}_{p,\theta}(\gdomain^{(1)})$, it suffices to prove the asserted inequality for $u\in C^\infty_0(\gdomain^{(1)})$. In this case, because of the assumed Lipschitz-continuity of $\phi^{-1}$, the classical partial derivatives of $u\circ \phi^{-1}$ exist a.e.\ and
\begin{align*}
\Big|\frac{\partial}{\partial x^j}\rklam{u\circ\phi^{-1}}\Big|
=
\Big|\sum_{i=1}^{d} \rklam{\frac{\partial}{\partial x^i}u}\circ \phi^{-1}
\frac{\partial}{\partial x^j}\phi^{-1}_i\Big|
\leq N \sum_{i=1}^d \Big|\rklam{\frac{\partial}{\partial x^i}u}\circ \phi^{-1}\Big|\qquad \text{(a.e.)}.
\end{align*}
Thus, using e.g.\ \cite[Section 1.1.3, Theorem 2]{Maz2011} we can conclude that these a.e.\ existing classical derivatives coincide with the weak derivatives, so that
\begin{align*}
\nnrm{u\circ &\phi^{-1}}{H^1_{p,\theta}(\gdomain^{(2)})}^p\\
&\leq N \sgrklam{ \int_{\gdomain^{(2)}} | \rklam{u\circ\phi^{-1}}(y) |^p \rho_{\gdomain^{(2)}}(y)^{\theta-d} \, dy
+ \sum_{j=1}^{d} \int_{\gdomain^{(2)}}\Big| \frac{\partial}{\partial x^j}\rklam{u\circ\phi^{-1}}(y)\Big|^p \rho_{\gdomain^{(2)}}(y)^{p+\theta-d} dy}\\
&\leq N \sgrklam{\int_{\gdomain^{(2)}} | \rklam{u\circ\phi^{-1}}(y) |^p \rho_{\gdomain^{(2)}}(y)^{\theta-d} \, dy
+ \int_{\gdomain^{(2)}}\sum_{i=1}^{d} \Big| \rklam{\frac{\partial}{\partial x^i}u}(\phi^{-1}(y))
\Big|^p \rho_{\gdomain^{(2)}}(y)^{p+\theta-d} dy}\\
&\leq N \sgrklam{\int_{\gdomain^{(1)}} | u(x) |^p \rho_{\gdomain^{(1)}}(x)^{\theta-d} \, dx
+  \sum_{i=1}^{d} \int_{\gdomain^{(1)}} \Big|\frac{\partial}{\partial x^i}u(x) \Big|^p \rho_{\gdomain^{(1)}}(x)^{p+\theta-d} dx}\\
&\leq N \nnrm{u}{H^1_{p,\theta}(\gdomain^{(1)})}^p.
\end{align*}
Finally, for $\gamma=-1$, we can use the fact that $H^1_{p',\theta'}(\gdomain)$ is the dual space of $H^{-1}_{p,\theta}(\gdomain)$, if $1/p+1/p'=1$ and $\theta/p+\theta'/p'=d$, see \cite[Proposition 2.4]{Lot2000}, and fall back to the cases we have already proven.
\end{proof}

\begin{proof}[Proof of Lemma \ref{pullbackEqn}]
We set $f:=\bD u$ and $\;g:=\bS u$. Since $u\in\frH^{1,q}_{p,\theta}(\gdomain^{(1)},T)$,  Lemma \ref{lem:Transf_rule} guarantees that 
$f\circ\phi^{-1}\in \bH^{-1,q}_{p,\theta+p}(\gdomain^{(2)},T)$, 
$g\circ\phi^{-1}\in \bH^{0,q}_{p,\theta}(\gdomain^{(2)},T;\ell_2)$ 
and 
$u(0,\cdot)\circ\phi^{-1}\in U^{1,q}_{p,\theta}(\gdomain^{(2)})$.
Therefore, we only have to show that for all $\varphi\in C^\infty_0(\gdomain^{(2)})$, 
with probability one, the equality
\begin{equation}\label{eq:pullbackEqn}
\big(u(t,\cdot),\varphi\circ\phi\big)=\big(u(0,\cdot),\varphi\circ\phi\big) + \int^{t}_{0}\big(f(s,\cdot),\varphi\circ\phi\big) \, ds
+\sum^{\infty}_{k=1} \int^{t}_{0}\big(g^k(s,\cdot),\varphi\circ\phi\big)\, dw^k_s
\end{equation}
holds for all $t\in[0,T]$. 
Thus, let us fix $\varphi\in C^\infty_0(\gdomain^{(2)})$. We consider first the case $p>2$.
By Lemma \ref{lem:Transf_rule}, $\varphi\circ\phi\in H^1_{\bar{p},\bar{\theta}-\bar{p}}(\gdomain^{(1)})$ for any $\bar{p}\in(1,\infty)$ and $\bar{\theta}\in\bR$, hence also for 
$\bar{p}:=2p/(p-2)$ and $\bar{\theta}:=2\theta'(p-1)/(p-2)-dp/(p-2)$,
where $\theta/p+\theta'/p'=d$ with $1/p+1/p'=1$. Moreover, by Lemma \ref{lem:collection}(i) we can choose a sequence $\bar{\varphi}_n\subseteq C^\infty_0(\gdomain^{(1)})$ approximating $\varphi\circ\phi$ in $H^1_{\bar{p},\bar{\theta}-\bar{p}}(\gdomain^{(1)})$. Furthermore, a consequence of our assumptions is, that for all $n\in\bN$, with probability one, the equality
\begin{equation}\label{eq:forwardEqn}
\big(u(t,\cdot),\bar{\varphi}_n\big)=\big(u(0,\cdot),\bar{\varphi}_n\big) + \int^{t}_{0}\big(f(s,\cdot),\bar{\varphi}_n\big) \, ds
+\sum^{\infty}_{k=1} \int^{t}_{0}\big(g^k(s,\cdot),\bar{\varphi}_n\big)\, dw^k_s
\end{equation}
holds for all $t\in[0,T]$. Thus, if we can show that each side of \eqref{eq:forwardEqn} converges in $L_2(\Omega;C([0,T]))$ to the respective side of \eqref{eq:pullbackEqn}, the assertion follows. To this end, let us fix an arbitrary $\bar{v}\in H^1_{\bar{p},\bar{\theta}-\bar{p}}(\gdomain^{(1)})$. A standard estimate yields
\begin{equation*}
\E\sgeklam{\sup_{t\in[0,T]}\Big|\int_0^t\big(f(s,\cdot),\bar{v}\big)\,ds\Big|^2}
\leq
N 
\nnrm{f}{\bH^{-1,q}_{p,\theta+p}(\gdomain^{(1)},T)}^2
\nnrm{\bar{v}}{H^1_{p',\theta'-p'}(\gdomain^{(1)})}^2.
\end{equation*}
Also, using Doob's inequality, It{\^o}'s isometry and H{\"o}lder's inequality we get
\begin{equation}\label{eq:pbEqnSt}
\E\sgeklam{\sup_{t\in[0,T]}\Big|\sum_{k\in\bN}\int_0^t\big(g^k(s,\cdot),\bar{v}\big)\,dw^k_s\Big|^2}
\leq 
N
\nnrm{g}{\bH^{0,q}_{p,\theta}(\gdomain^{(1)},T;\ell_2)}^2
\nnrm{\bar{v}}{L_{\bar{p},\bar{\theta}}(\gdomain^{(1)})}^2.
\end{equation}
Furthermore, an application of \cite[Proposition 2.9]{Kim2011} and the fact that $q\geq p$ lead to
\begin{align*}
\E\sgeklam{\sup_{t\in[0,T]}\big|\big(u(t,\cdot),\bar{v}\big)\big|^2}
&\leq
\sgrklam{\E\sgeklam{\sup_{t\in[0,T]}\nnrm{u(t,\cdot)}{L_{p,\theta}(\gdomain^{(1)})}^p}}^{2/p}
\nnrm{\bar{v}}{L_{p',\theta'}(\gdomain^{(1)})}^2\\
&\leq
N
\nnrm{u}{\frH^{1,p}_{p,\theta}(\gdomain^{(1)},T)}^2
\nnrm{\bar{v}}{L_{p',\theta'}(\gdomain^{(1)})}^2\\
&\leq 
N
\nnrm{u}{\frH^{1,q}_{p,\theta}(\gdomain^{(1)},T)}^2
\nnrm{\bar{v}}{L_{p',\theta'}(\gdomain^{(1)})}^2.
\end{align*}
Hence, since $H^1_{\bar{p},\bar{\theta}-\bar{p}}(\gdomain^{(1)})$ is continuously embedded in 
$H^1_{p',\theta'-p'}(\gdomain^{(1)}) \cap L_{p',\theta'}(\gdomain^{(1)}) \cap L_{\bar{p},\bar{\theta}}(\gdomain^{(1)})$, the assertion follows for the case $p>2$. The same arguments can be used to prove the case $p=2$: Just replace $\bar{p}$ by $2$, and $\bar{\theta}$ by $\theta'=2d-\theta$ and use the estimate
\begin{equation*}
\E\sgeklam{\sup_{t\in[0,T]}\Big|\sum_{k\in\bN}\int_0^t\big(g^k(s,\cdot),\bar{v}\big)\,dw^k_s\Big|^2}
\leq 
N
\nnrm{g}{\bH^{0,q}_{2,\theta}(\gdomain^{(1)},T;\ell_2)}^2
\nnrm{\bar{v}}{L_{2,\theta'}(\gdomain^{(1)})}^2
\end{equation*}
instead of \eqref{eq:pbEqnSt}.
\end{proof}

\begin{proof}[Proof of Lemma \ref{lem:fancy}]
Let us fix $j\in\bN_0$. 
We use the notations introduced before Lemma \ref{lem:fancy}. Remember that $\domainc>0$ has been chosen in such a way that $\domain\subseteq[-\domainc,\domainc]^d$. Let us fix $k_1\geq 1$ such that
\begin{equation}\label{eq:fancy-k_1}
2+2\domainc \sqrt{d} < 2^{k_1},
\end{equation} 
and construct a sequence $\{\xi_n:n\in\bZ\}\subseteq C^\infty_0(\domain)$ as in Remark~\ref{rem:EquivWS}(ii).
In order to prove the assertion we are going to show the  estimates
\begin{equation}\label{eq:fancy-A}
\sum_{(\ijk)\in\Lambda^0_j}\sgrklam{\rho_{\jk}^{\gamma-\nu}|u|_{B_{p,p}^\gamma(Q_\jk^\circ)}}^p
\leq N
\sum_{n\in\bN_0}2^{-(j-n)(\gamma-\nu)p}|\xi_{j-n}u|_{B^\gamma_{p,p}(\bR^d)}^p,
\end{equation}
and
\begin{equation}\label{eq:fancy-B}
|\xi_{j-n}u|_{B^\gamma_{p,p}(\bR^d)}^p
\leq N\,
2^{-(j-n)(d-\gamma p)}\gnnrm{\xi_{j-n}\big(2^{-(j-n)}\cdot\big)u\big(2^{-(j-n)}\cdot\big)}{H^\gamma_p(\bR^d)}^p,
\end{equation}
where the constant $N$ does not depend on $j$ and $n$.
This will prove the assertion since, assuming that \eqref{eq:fancy-A} and \eqref{eq:fancy-B} are true, their combination gives
\begin{align*}
\sum_{(\ijk)\in\Lambda^0_j}\big(\rho_{\jk}^{\gamma-\nu} |u|_{B_{p,p}^\gamma(Q_\jk^\circ)}\big)^p
&\leq N
\sum_{n\in\bN_0}2^{-(j-n)(d-\nu p)}\gnnrm{\xi_{j-n}\big(2^{-(j-n)}\cdot\big)u\big(2^{-(j-n)}\cdot\big)}{H^\gamma_p(\bR^d)}^p\\
&\leq N 
\sum_{n\in\bZ}2^{n(d-\nu p)}\gnnrm{\xi_{-n}\big(2^n\cdot\big)u\big(2^n\cdot\big)}{H^\gamma_p(\bR^d)}^p\\
&\leq N\,
\nnrm{u}{H^\gamma_{p,d-\nu p}(\domain)}^p.
\end{align*}
In the last step we used Remark \ref{rem:EquivWS}(ii) and Lemma \ref{lem:EquivWS}.

Let us first verify inequality \eqref{eq:fancy-B}. To this end, let $r$ be the smallest integer strictly greater than $\gamma$. For the sake of clarity we use here the notation $\Delta_h^r[f]$ for the $r$-th difference of a function $f:\bR^d\to\bR$ with step $h\in\bR^d$, whereas $\Delta_h^r[f](x)$ denotes the value of this $r$-th difference at a point $x\in\bR^d$, compare Subsection~\ref{Besov}. Writing out the Besov semi-norm and applying the transformation formula for integrals we see that
\begin{align*}
|\xi_{j-n}u|_{B^\gamma_{p,p}(\bR^d)}^p
&= 
\int_0^\infty t^{-\gamma p}
\sup_{|h|<t}\gnnrm{\Delta^r_h[\xi_{j-n}u]}{L_p(\bR^d)}^p\,\frac{dt}{t}\\
&= 
2^{-(j-n)d}
\int_0^\infty t^{-\gamma p}
\sup_{|h|<t}\sggklam{\int_{\bR^d}\big|\Delta^r_h[\xi_{j-n}u]\big(2^{-(j-n)}x\big)\big|^p\,dx}\,\frac{dt}t\\
&= 
2^{-(j-n)d}
\int_0^\infty t^{-\gamma p}
\sup_{|h|<2^{j-n}t}\sggklam{\int_{\bR^d}\Big|\Delta^r_h\Big[\xi_{j-n}\big(2^{-(j-n)}\cdot\big)u\big(2^{-(j-n)}\cdot\big)\Big](x)\Big|^p\,dx}\,\frac{dt}t.
\end{align*}
A further application of the transformation formula for integrals yields
\begin{align*}
|\xi_{j-n}u|_{B^\gamma_{p,p}(\bR^d)}^p
&= 
2^{-(j-n)d}2^{(j-n)\gamma p}
\int_0^\infty t^{-\gamma p}\sup_{|h|<t}\sgnnrm{\Delta^r_h\Big[\xi_{j-n}\big(2^{-(j-n)}\,\cdot\,\big)u\big(2^{-(j-n)}\,\cdot\,\big)\Big]}{L_p(\bR^d)}^p\,\frac{dt}t\\
&= 
2^{-(j-n)(d-\gamma p)}
\big|\xi_{j-n}\big(2^{-(j-n)}\,\cdot\,\big)u\big(2^{-(j-n)}\,\cdot\,\big)\big|_{B^\gamma_{p,p}(\bR^d)}^p,
\end{align*}
which implies \eqref{eq:fancy-B} since the space $H^\gamma_p(\bR^d)$ of Bessel potentials is continuously embedded in the Besov space $B^\gamma_{p,p}(\bR^d)$, see \cite[Theorem 2.3.2(d) combined with Theorem 2.3.3(a)]{Tri1995}.

It remains to  prove inequality \eqref{eq:fancy-A}. Recall that the index $i$ referring to the different types of wavelets on a cube $Q_\jk$ ranges from $1$ to $2^d-1$. Since $\Lambda_j^0$ consists of those indices $(\ijk)\in\Lambda_j$ with $2^{-j}\leq\rho_{\jk}$, we have
\begin{equation}\label{eq:fancy-A1}
\sum_{(\ijk)\in\Lambda^0_j}\big(\rho_{\jk}^{\gamma-s}|u|_{B_{p,p}^\gamma(Q_\jk)}\big)^p=(2^d-1)\sum_{k\in \Lambda_j^\star}\big(\rho_{\jk}^{\gamma-s}|u|_{B_{p,p}^\gamma(Q_\jk)}\big)^p,
\end{equation}
where we used the notation
\begin{equation*}
\Lambda^\star_j:=\ggklam{ k\in\bZ^d \, : \, (\ijk)\in\Lambda^0_j}.
\end{equation*}
Now we get the required estimate in four steps.

\noindent\textbf{Step 1.} We first show that the cubes supporting the wavelets fit into the stripes where the cut-off functions $(\xi_n)$ are identical to one. More precisely, we claim that the proper choice of $k_1$, see \eqref{eq:fancy-k_1}, leads to the fact that, for any $k\in\Lambda_j^\star$, there exists a non-negative integer $n\in\bN_0$ such that
\begin{equation*}
Q_\jk\subseteq S_{j-n}:=\rho^{-1}\grklam{2^{-(j-n)}\geklam{2^{-k_1},2^{k_1}}}.
\end{equation*}
We denote
\begin{equation*}
S_{j,n}^\star:=\ggklam{ k\in\Lambda_j^\star \, : \, Q_\jk\in S_{j-n} },
\quad n\in\bN_0.
\end{equation*}
To prove this, we first note that, since $k_1\geq 1$ fulfils \eqref{eq:fancy-k_1}, 
\begin{equation*}
\bigcup_{k\in \Lambda_j^\star} Q_\jk \subseteq \bigcup_{n\in\bN_0} S_{j-n} = \bigcup_{n=0}^{j} S_{j-n}.
\end{equation*}
Fix $k\in\Lambda_j^\star$ and let $n^*$ be the smallest non-negative integer such that $Q_\jk\cap S_{j-n^*}\neq\emptyset$, i.e.,
\begin{equation*}
n^*:=\inf \ggklam{n\in \bN \,:\, Q_\jk\cap S_{j-n}\neq\emptyset}\leq j.
\end{equation*}
Then, there are two possibilities: On the one hand, $Q_\jk$ might be contained completely in $S_{j-n^*}$, i.e., $Q_\jk\subseteq S_{j-n^*}$. Then we are done. On the other hand, it might happen that $Q_\jk$ is not completely contained in the stripe $S_{j-n^*}$. In this case, we claim that $Q_\jk\subseteq S_{j-(n^*+1)}$, i.e.,
\begin{equation*}
\rho (x) \in \geklam{2^{-j+n^*+1} 2^{-k_1},2^{-j+n^*+1} 2^{k_1}} \text{ for all } x\in Q_\jk.
\end{equation*}
Let us therefore fix $x\in Q_\jk$. Then, since the length of the diagonal of $Q_\jk$ is $2^{-j} 2 \domainc \sqrt{d}$, we have
\begin{align*}
\rho(x)
\leq
\rho_\jk + 2^{-j} 2 \domainc \sqrt{d}.
\end{align*}
Also, $\rho_\jk \leq 2^{-j+n^*}2^{k_1}$ since $Q_\jk\cap S_{j-n^*}\neq\emptyset$. Hence,
\begin{align*}
\rho(x)
\leq
2^{-j+n^*} 2^{k_1}+ 2^{-j} 2 \domainc \sqrt{d}.
\end{align*}
Since $2 \domainc\sqrt{d} \leq 2^{k_1}$, we conclude that
\begin{align*}
\rho(x)
\leq
2^{-j+n^*+1} 2^{k_1} \rrklam{ \frac{1}{2} +\frac{2 \domainc \sqrt{d}}{2^{n^*+1} 2^{k_1} } }
\leq
2^{-j+n^*+1} 2^{k_1} \rrklam{ \frac{1}{2} + \frac{1}{2^{n^*+1} }}
\leq
2^{-j+n^*+1} 2^{k_1}.
\end{align*}
It remains to show that $\rho(x)\geq 2^{-j+n^*+1}2^{-k_1}$. We argue as follows: Since $Q_\jk$ is not completely contained in $S_{j-n^*}$, there exists a point $x_0\in Q_\jk$ such that $\rho(x_0) > 2^{-j+n^*} 2^{k_1}$. Therefore, since the length of the diagonal of $Q_\jk$ is $2^{-j}2\domainc\sqrt{d}$, we have
\begin{align*}
\rho(x) 
>
2^{-j+n^*} 2^{k_1} - 2 \domainc \sqrt{d}\, 2^{-j}
\geq
2^{-j+n^*+1} 2^{-k_1} \rrklam{\frac{2^{2 k_1}}{2} - \frac{2 \domainc \sqrt{d}\, 2^{k_1}}{2^{n^*+1}}}
\geq
2^{-j+n^*+1} 2^{-k_1}.
\end{align*}
In the last step we used \eqref{eq:fancy-k_1}.

\noindent\textbf{Step 2.} We rearrange the cubes supporting the wavelets in classes containing only cubes with disjoint interiors. 
More precisely, let $e_1,\ldots,e_d$ be the canonical orthonormal basis in $\bR^d$. Since $Q_\jk=2^{-j}\big(k+[-\domainc,\domainc]^d\big)$ it is clear that, for all  $k\in\bZ$ and $l\in\{1,\ldots,d\}$,
\begin{align*}
Q_\jk^\circ\cap\big(2^{-j}(2\domainc)e_l+Q_\jk^\circ\big)=\emptyset.
\end{align*}
Consequently, setting $ \{a_m : m=1,\ldots,(2\domainc)^d\}:=\{0,\ldots,2\domainc -1\}^d $ and denoting
\begin{align*}
R_{j,m}:=\ggklam{Q_\jk \,:\, k\in a_m+2\domainc \bZ^d}\quad\text{for}\quad m=1,\ldots,(2\domainc)^d,
\end{align*}
we cover the whole range of cubes, i.e.,
\begin{align*}
\bigcup_{k\in\bZ^d} Q_{j,k} = \bigcup_{m=1}^{(2\domainc)^d} R_{j,m},
\end{align*}
and for any fixed $m\in\{1,\ldots,(2\domainc)^d\}$, if $Q_{j,k}, Q_{j,\ell} \in R_{j,m}$ with $k\neq\ell$, then $Q_\jk^\circ \cap Q_{j,\ell}^\circ = \emptyset$.
In the sequel, we write 
\begin{equation*}
R_{j,m}^\star:=\ggklam{k\in\Lambda_j^\star \, : \, Q_\jk\in R_{j,m}},\qquad m\in\ggklam{ 0,\ldots, (2\domainc)^d-1}.
\end{equation*}

\noindent\textbf{Step 3.} Let us fix $k\in\Lambda_j^\star$ and concentrate on the Besov semi-norm of the restriction of $u$ to the corresponding cube $Q_\jk^\circ$. 
Using the Peetre $K$-functional
\begin{equation*}
K_r(t,u,Q_\jk^\circ)_p
:=
\inf_{g\in W^r_p(Q_\jk^\circ)} \sggklam{\nnrm{u-g}{L_p(Q_\jk)}+t\, |g|_{W^r_p(Q_\jk^\circ)}}
\end{equation*} 
and applying \cite[Lemma 1]{JohSch1977} leads to
\begin{align*}
|u|_{B^\gamma_{p,p}(Q_\jk^\circ)}^p
&=
\int_0^\infty t^{-\gamma p} \omega^r(t,u,Q_\jk^\circ)_p^p\,\frac{dt}t\\
& \leq N
\int_0^\infty t^{-\gamma p} K_r(t^r,u,Q_\jk^\circ)_p^p\,\frac{dt}t\\
& = N
\int_0^\infty t^{-\gamma p} \inf_{g\in W^r_p(Q_\jk^\circ)}
\!\sggklam{\nnrm{u-g}{L_p(Q_\jk)}+t^r|g|_{W^r_p(Q_\jk^\circ)}}^p\,\frac{dt}t\\
& \leq N
\int_0^\infty t^{-\gamma p} \inf_{g\in W^r_p(Q_\jk^\circ)}
\!\sggklam{\nnrm{u-g}{L_p(Q_\jk)}^p+t^{rp}|g|_{W^r_p(Q_\jk^\circ)}^p}\,\frac{dt}t,
\end{align*}
where the constant $N$ depends only on $r$, $d$ and $p$. (Recall that $r$ is the smallest integer strictly greater than $\gamma$.)

\noindent\textbf{Step 4.} Now we collect the fruits of our work and approximate the right hand side of \eqref{eq:fancy-A1}. Because of the first step and since $k_1\geq 1$ it is easy to see that
\begin{equation*}
\Lambda_j^\star = \bigcup_{n=0}^{j} \bigcup_{m=1}^{(2\domainc)^d} S_{j,n}^\star \cap R_{j,m}^\star,
\end{equation*}
where
\begin{equation*}
S_{j,n}^\star:=\ggklam{ k\in\Lambda_j^\star \, : \, Q_\jk\in S_{j-n} },
\quad n\in\bN_0.
\end{equation*}
Thus,
\begin{equation}\label{eq:fancy-A2}
\sum_{k\in\Lambda_j^\star} 
\sgrklam{\rho_\jk^{\gamma-\nu} | u |_{B^\gamma_{p,p}(Q_\jk^\circ)}}^p
\leq
\sum_{n=0}^{j}
\sum_{m=1}^{(2\domainc)^d} 
\sum_{k\in S_{j,n}^\star\cap R_{j,m}^\star}
\sgrklam{\rho_\jk^{(\gamma-\nu) p} \,| u |_{B^\gamma_{p,p}(Q_\jk^\circ)}^p}.
\end{equation}
Let us fix $n\in\bN_0$ such that $S_{j,n}^\star\neq \emptyset$ as well as $m\in\{1,\ldots,(2\domainc)^d\}$. Then, $\rho_\jk\leq 2^{k_1}2^{-(j-n)}$ for $k\in S^\star_{j,n}$, and using the third step we obtain
\begin{align*}
&\sum_{k\in S_{j,n}^\star\cap R_{j,m}^\star} 
\sgrklam{\rho_\jk^{(\gamma-\nu) p} \,| u |_{B^\gamma_{p,p}(Q_\jk^\circ)}^p}	\\
&\quad \leq N
\sum_{k\in S_{j,n}^\star\cap R_{j,m}^\star}
\sgrklam{2^{-(j-n)(\gamma-\nu)p} 
\int_0^\infty t^{-\gamma p} \inf_{g\in W^r_p(Q_\jk^\circ)}
\sggklam{\nnrm{u-g}{L_p(Q_\jk)}^p+t^{rp}|g|_{W^r_p(Q_\jk^\circ)}^p}\,\frac{dt}t}	\\
&\quad \leq N
2^{-(j-n)(\gamma-\nu)p}
\sum_{k\in S_{j,n}^\star\cap R_{j,m}^\star} 
\int_0^\infty 
t^{-\gamma p}
\inf_{g\in W^r_p(\domain)}
\sggklam{\nnrm{u-g}{L_p(Q_\jk)}^p+t^{rp}|g|_{W^r_p(Q_\jk^\circ)}^p}
\,\frac{dt}{t}.	
\end{align*}
Furthermore, since $\xi_{j-n}=1$ on $Q_\jk$ for any $k \in S_{j,n}^\star$ and since $Q_\jk^\circ\cap Q_{j,\ell}^\circ=\emptyset$ for $k,\ell\in R_{j,m}^\star$ if $k\neq\ell$, we can continue our estimate as follows:
\begin{align*}
&\sum_{k\in S_{j,n}^\star\cap R_{j,m}^\star}
\sgrklam{\rho_\jk^{(\gamma-\nu) p} \,| u |_{B^\gamma_{p,p}(Q_\jk^\circ)}^p}	\\
&\quad \leq N
2^{-(j-n)(\gamma-\nu)p}
\int_0^\infty
t^{-\gamma p}
\inf_{g\in W^r_p(\domain)}
\sggklam{
\sum_{k\in S_{j,n}^\star\cap R_{j,m}^\star} 
\!\!\sgrklam{\nnrm{\xi_{j-n}u-g}{L_p(Q_\jk)}^p+t^{rp}|g|_{W^r_p(Q_\jk^\circ)}^p}}
\,\frac{dt}{t}	\\
&\quad \leq N
2^{-(j-n)(\gamma-\nu)p}
\int_0^\infty
t^{-\gamma p}
\inf_{g\in W^r_p(\domain)}
\sggklam{
\nnrm{\xi_{j-n}u-g}{L_p(\domain)}^p+t^{rp}|g|_{W^r_p(\domain)}^p}
\,\frac{dt}{t}	\\
&\quad \leq N
2^{-(j-n)(\gamma-\nu)p}
\int_0^\infty
t^{-\gamma p} 
K_r(t^r,\xi_{j-n} u,\domain)_p^p
\,\frac{dt}{t}.	
\end{align*}
By \cite[Theorem 1]{JohSch1977}, we know that there exists a constant $N$, depending only on $r$, $p$ and $\domain$, such that
\begin{equation*}
K_r(t^r,\xi_{j-n} u,\domain)_p 
\leq N\,
\omega^r(t,\xi_{j-n} u,\domain)_p.
\end{equation*}
Putting everything together, we have shown that there exists a constant $N$ which does not depend on $j$, $n$ or $m$ such that
\begin{align*}
\sum_{k\in S_{j,n}^\star\cap R_{j,m}^\star}
\sgrklam{\rho_\jk^{(\gamma-\nu) p} \,| u |_{B^\gamma_{p,p}(Q_\jk^\circ)}^p}	
& \leq N
2^{-(j-n)(\gamma-\nu)p}
\int_0^\infty
t^{-\gamma p} 
\omega^r(t,\xi_{j-n} u,\domain)_p^p
\,\frac{dt}{t}	\\
&= N 
2^{-(j-n)(\gamma-\nu)p}
| u |_{B^\gamma_{p,p}(\domain)}^p.
\end{align*}
Hence, \eqref{eq:fancy-A} follows after inserting this estimate into \eqref{eq:fancy-A2} and using \eqref{eq:fancy-A1}.
\end{proof}

\end{appendix}

\noindent\textbf{Acknowledgements.} The second and the third author are
thankful for the hospitality of the other authors during their stay at TU Dresden. The first and last author are deeply grateful for the hospitality of the other authors during their stay at the 2nd NIMS Summer School in Probability Theory in Daejeon and during their stay at Korea University.

\providecommand{\bysame}{\leavevmode\hbox to3em{\hrulefill}\thinspace}
\providecommand{\MR}{\relax\ifhmode\unskip\space\fi MR }
\providecommand{\MRhref}[2]{%
  \href{http://www.ams.org/mathscinet-getitem?mr=#1}{#2}
}
\providecommand{\href}[2]{#2}

\newpage
\section*{Affiliations \& Postal Addresses}
\noindent Petru~A.~Cioica \\
Philipps-Universit{\"a}t Marburg \\
FB Mathematik und Informatik, AG Numerik/Optimierung \\
Hans-Meerwein-Strasse \\
35032 Marburg, Germany \\
Phone: 00\,49\,6421\,28\,25\,481\\
E-mail: cioica@mathematik.uni-marburg.de \\[1ex]

\noindent Kyeong-Hun~Kim\\
Korea University\\
Department of Mathematics\\
Seoul, South Korea, 136--701\\
Phone: 00\,82\,2\,3290\,3071 \\
E-mail: kyeonghun@korea.ac.kr\\[1ex]

\noindent Kijung~Lee\\
Ajou University\\ 
Department of Mathematics\\ 
Suwon, South Korea, 443--749\\
Phone: 00\,82\,31\,219\,1936\\
E-mail: kijung@ajou.ac.kr\\[1ex]

\noindent Felix~Lindner\\
TU Dresden \\
FR Mathematik, Institut f{\"u}r Mathematische Stochastik \\
01062 Dresden, Germany \\
Phone: 00\,49\,351\,463\,32\,437\\
E-mail: felix.lindner@tu-dresden.de \\[1ex]

\end{document}